\documentclass[a4paper, UKenglish, cleveref, autoref]{article}

\usepackage[T1]{fontenc}
\usepackage[a4paper, margin=1.1in]{geometry}
\usepackage[mathscr]{eucal}
\usepackage[linesnumbered,ruled,vlined]{algorithm2e}
\usepackage{standalone}
\usepackage{lmodern}
\usepackage{caption}
\usepackage{amsfonts,amsmath,amsthm,amssymb,mathtools}
\usepackage{dsfont}
\usepackage{enumitem}
\usepackage{thmtools}
\usepackage{graphicx}
\usepackage{tikz}
\usepackage{tikz-cd}
\usepackage{tikz-3dplot}
\usepackage{quiver}
\usepackage{float}
\usepackage[dvipsnames,table]{xcolor}
\usepackage{tcolorbox}
\tcbuselibrary{skins, breakable}
\usepackage{booktabs}
\usepackage{siunitx}
\usepackage{multirow}
\usepackage[final]{microtype}
\usepackage{hyperref}
\usepackage[backend=bibtex, style=numeric, natbib, maxbibnames=99]{biblatex}
\hypersetup{
  colorlinks=true,
  linkcolor=darkblue,   % internal: sections, theorems, equations
  citecolor=darkred,    % citations
  urlcolor=darkblue,    % URLs (same as links is fine)
}

\newtheorem{theorem}{Theorem}[section]

\theoremstyle{plain}
\newtheorem{lemma}[theorem]{Lemma}

\newtheorem{proposition}{Proposition}[section]

\newtheorem{corollary}{Corollary}[section]

\newcommand{\problemname}{Problem}
\newcommand{\questionname}{Question}
\newcommand{\conjecturename}{Conjecture}
\newcommand{\remarkname}{Remark}
\newcommand{\observationname}{Observation}

\theoremstyle{definition}
\newtheorem{example}{Example}[section]
\newtheorem{definition}{Definition}[section]
\newtheorem{notation}{Notation}[section]

\theoremstyle{definition}

\newtheorem{question}{\questionname}

\theoremstyle{plain}

\theoremstyle{remark}

\newtheorem{remark}{\remarkname}

\newtheorem{observation}{\observationname}

\tcolorboxenvironment{prob}{
  enhanced, breakable, 
  colback=gray!8, colframe=black,
  boxrule=1.2pt, arc=0pt,
  before skip=\topsep, after skip=\topsep,
}

\tcolorboxenvironment{conjecture}{
  enhanced, breakable, 
  colback=gray!8, colframe=black,
  boxrule=1.2pt, arc=0pt,
  before skip=\topsep, after skip=\topsep,
}

\tcbset{thmbox/.style={
  enhanced, breakable,
  colback=gray!5, colframe=black,
  boxrule=0.6pt, arc=0pt,
  before skip=\topsep, after skip=\topsep,
}}

\tcolorboxenvironment{theorem}{thmbox}

\tcolorboxenvironment{question}{
  enhanced, breakable, 
  colback=white, colframe=black,
  boxrule=0.6pt, arc=0pt,
  before skip=\topsep, after skip=\topsep,
}

\tcolorboxenvironment{remark}{
  enhanced, breakable, 
  colback=white, colframe=white,
  borderline west={2pt}{0pt}{black!40},
  boxrule=0pt, arc=0pt,
  left=8pt,
  before skip=\topsep, after skip=\topsep,
}

\tcolorboxenvironment{observation}{
  enhanced, breakable, 
  colback=white, colframe=white,
  borderline west={2pt}{0pt}{black!40},
  boxrule=0pt, arc=0pt,
  left=8pt,
  before skip=\topsep, after skip=\topsep,
}

\newtheoremstyle{algorithmstyle}
  {10pt} % Space above
  {10pt} % Space below
  {\normalfont} % Body font
  {} % Indent
  {\bfseries} % Head font
  {} % No Punctuation after head
  {1em} % Space after head
  {\thmname{#1}\thmnumber{ #2}\thmnote{ (#3)}\newline} % Head spec + forced line break
\theoremstyle{algorithmstyle}

\makeatletter
\let\c@myalgorithm\c@algocf
\makeatother

\SetKw{KwRet}{return}
\SetKw{KwTo}{to}
\SetKw{KwDownto}{downto}
\SetKw{KwAnd}{and}
\SetKw{KwBreak}{break}
\SetKw{KwContinue}{continue}
\SetKw{KwTrue}{true}
\SetKw{KwFalse}{false}

\newcommand{\N}{\mathbb{N}}
\newcommand{\Z}{\mathbb{Z}}
\newcommand{\R}{\mathbb{R}}

\newcommand{\K}{\mathds k} 
\newcommand{\Der}{\mathcal{D}}
\newcommand{\vect}{\textbf{Vect}}
\newcommand{\grA}{A\text{-GrMod}}
\newcommand{\Set}{\textbf{Set}}
\newcommand{\oto}[1]{\xrightarrow{#1}}

\newcommand{\lra}{\longrightarrow}
\global\long\def\into{\hookrightarrow}%
\global\long\def\onto{\twoheadrightarrow}%
\newcommand{\iso}{\xrightarrow{\,\smash{\raisebox{-0.5ex}{\ensuremath{\scriptstyle\sim}}}\,}}
\global\long\def\into{\hookrightarrow}%
\global\long\def\onto{\twoheadrightarrow}%

\DeclareMathOperator\Img{Im}
\DeclareMathOperator\coker{coker}
\DeclareMathOperator\Hom{Hom}

\DeclareMathOperator{\D}{D}
\DeclareMathOperator{\End}{End}
\DeclareMathOperator{\Ext}{Ext}
\DeclareMathOperator{\fun}{Fun}
\DeclareMathOperator{\Tor}{Tor}
\DeclareMathOperator{\Id}{Id}

\newcommand{\set}[1]{ {\left\{#1\right\}} }
\newcommand{\br}[1]{ {\left(#1\right)} }
\newcommand{\CC}{C\nolinebreak\hspace{-.05em}\raisebox{.4ex}{\tiny\bf +}\nolinebreak\hspace{-.10em}\raisebox{.4ex}{\tiny\bf +}}
\def\CC{{C\nolinebreak[4]\hspace{-.05em}\raisebox{.4ex}{\tiny\bf ++}}}

\newcommand{\Oc}{\mathcal{O}}
\newcommand{\thick}[1]{\mathfrak{t}#1}
\newcommand{\lthick}[1]{\mathfrak{t}_{l}#1}
\newcommand{\thickb}[1]{\mathfrak{t}(#1)}
\newcommand{\lthickb}[1]{\mathfrak{t}_{l}(#1)}

\newcommand{\ignore}[1]{}

\definecolor{lightred}{rgb}{1, 0.5, 0.5}
\definecolor{darkred}{rgb}{0.8, 0.3, 0.3}
\definecolor{darkerred}{rgb}{0.6, 0.1, 0.1}

\definecolor{lightblue}{rgb}{0.4, 0.5, 1}
\definecolor{darkblue}{rgb}{0.3, 0.3, 0.8}
\definecolor{Darkblue}{rgb}{0.2,0.2,0.7}
\definecolor{darkerblue}{rgb}{0.1, 0.1, 0.6}
\definecolor{darkestblue}{rgb}{0.05,0.05,0.35}

\newcommand{\lrcell}{\cellcolor{lightred}}
\newcommand{\lbcell}{\cellcolor{lightblue}}
\newcommand{\gcell}{\cellcolor{gray}}

\title{Computing Homomorphisms of Poset Representations with Applications to Multiparameter Persistence}
\author{Jan Jendrysiak}
\date{\today}

\begin{document}

\maketitle

\begin{abstract}
We present algorithms to compute the vector space of homomorphisms 
$\Hom(X,Y)$ between finitely generated representations of the partially ordered set $\Z^d$. Our results generalise to any
partially ordered set.

Our main theoretical 
contribution is a uniqueness result for lifts of homomorphisms along 
free resolutions, which we use to obtain an 
algorithm running in $\mathcal{O}(n^4(\thick{Y} + \thick{\Omega^1 Y})^2 
+ T_{\ker}(d,n))$ time, where $\thick{Y}$ denotes the maximal pointwise dimension 
of $Y$ and $T_{\ker}$ is the time it takes to compute the kernel of a map between projective $\Z^d$-modules. We also apply and analyse a classical approach due to Green, Heath, and Struble (J. Symbolic Comput., 2001), achieving $\mathcal{O}(n^3 \thick{Y}^3 + n^4)$. 
Both improve on the naive $\mathcal{O}(n^6)$ bound when $\thick{Y}$ is small.

Applied to the decomposition algorithm \textsc{aida} 
(Dey--J--Kerber, SoCG~'25), the classical approach improves the asymptotic runtime the most,
strengthening the result of Dey and Xin (J. Appl. Comput. Topology, 2022) for uniquely graded modules. We implement all algorithms in the 
\textsc{Persistence Algebra} \CC{} library and benchmark them on the persistent homology of density-alpha bi-filtrations of 
immune-cell locations. The classical approach has 
the best worst-case complexity, yet for $2$-parameter modules, the lifting algorithm is
fastest in practice.
\end{abstract}

\begin{figure}[H]
\centering
\includegraphics[scale = 1.1]{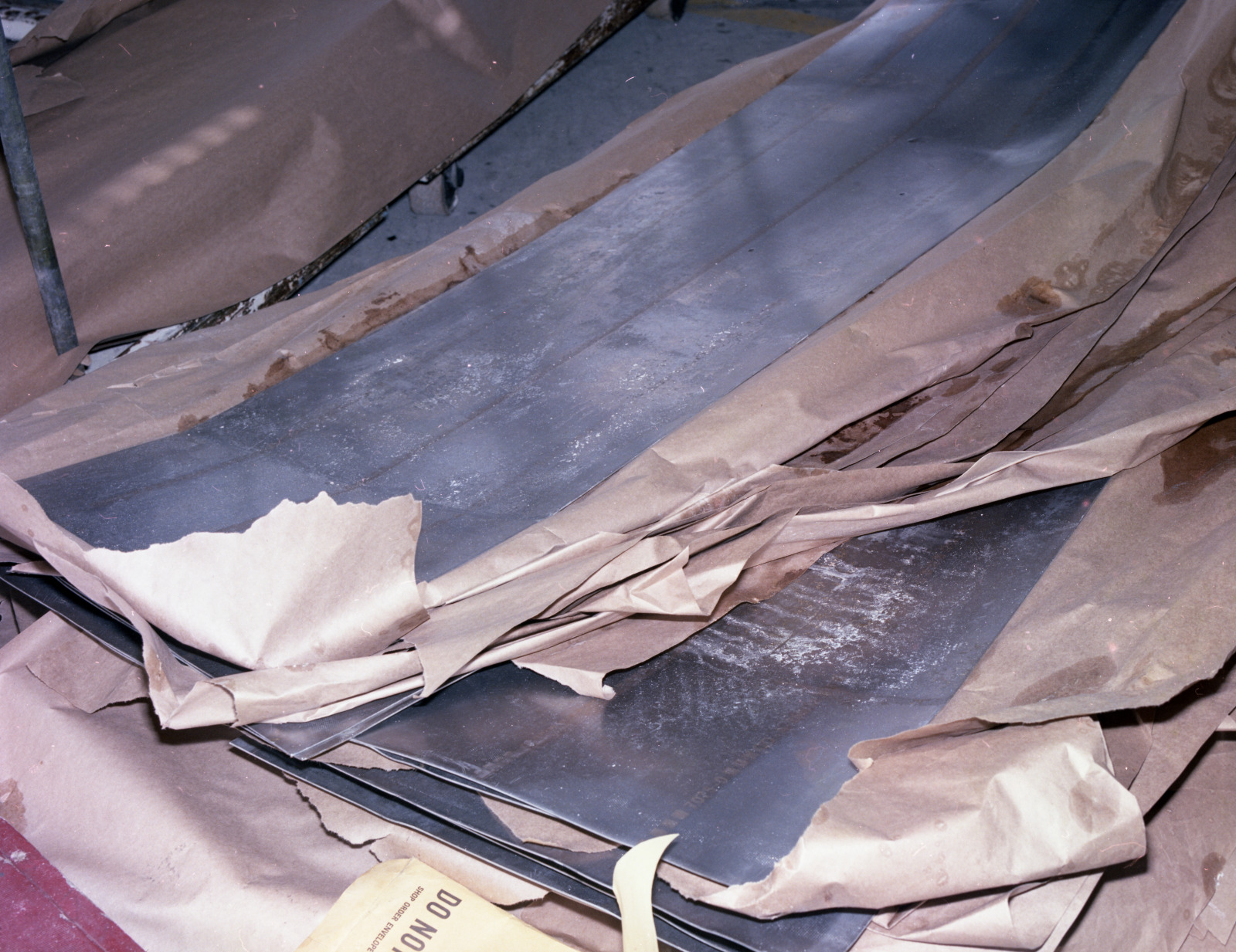}
\caption{Sheets of metal and paper - San Diego Air \& Space Museum Archives}
\label{fig:metal}
\end{figure}

\subparagraph{Poset Representations and Multiparameter Persistence.}

Representations of partially ordered sets are a classical subject that has been studied since the 70's \cite{Nazarova1975, Rojter1980, Kleiner1975, Loupias}. In the past decade, poset representations have become a central tool in applied topology through their appearance as \emph{Persistent Homology Groups}.

Persistent Homology \cite{Edelsbrunner_top_persistence} is a standard tool in Topological Data Analysis to estimate the homology groups of a space from a finite point sample, eg. via the Vietoris-Rips or Cech complex. One of its problems is that these constructions are sensitive to outliers. The most promising ansatz to solve this problem \cite{blumberg_lesnick} is to filter the complex along a second density parameter, leading to \emph{Multiparameter Persistent Homology} \cite{CZMP}. Its result is a
Multiparameter Persistence Module (or just persistence module) -- a functor $\R^d \to \vect_\K$. Since persistence modules are in practice mostly finitely presented, we can, for algebraic manipulation, assume that they are finitely presented functors $\Z^d \to \vect_K$ instead. Let $A \coloneqq \K[x_1, \dots, x_d]$, then these are exactly the finitely generated $\Z^d$-graded $A$-modules.

As with persistence modules, any finitely presented representation of a poset can be seen as a representation over a finite poset $\mathcal{P}$, which in turn can be faithfully embedded into $\Z^d$. Therefore any statement about $\Z^d$-graded $A$-modules is almost verbatim true for general finitely presented poset representations.
We will therefore, without loss of generality, state all definitions and theorems in this paper for $\Z^d$-graded $A$-modules. This has multiple advantages:

Most importantly, they appear in classical commutative algebra \cite{ MillerSturmfels2005} and have been studied for a long time \cite{goto1978graded}.
For us, who have the application to the Persistent Homology of bi-filtrations in mind, there is another reason. Many algebraic computations are known to be faster when $d$ is small, so we want to state results in relation to $d$. For example, when $d=2$ or $d=3$, this includes the computation of kernels \cite{lw-computing, kernels} and projective resolutions of chain-complexes consisting of upset-modules \cite{bauer26}. 
$\Z^d$-graded $A$-modules also have additional structure that is relevant to us; the vector space $\Hom(X,Y)$ can be enriched to a graded $A$-module. 

At last, we remark that this generality implies that our results hold in particular for \emph{Zig-Zag} modules (representations over $\Z \Z$ - two adjacent anti-diagonals in $\Z^2$) \cite{Carlsson2010} or representations over $\Z \times \Z \Z$ (\cite{dey_zigzag}.

\subparagraph{Motivation: Decomposition of Modules.}
Multiparameter persistence modules do not decompose into nice summands like their 1-parameter counterparts, which admit a \emph{barcode}. To use them for data analysis, one therefore uses invariants that are statistically well behaved and easy to interpret. These are only useful if they can be computed in a practical time-frame, and this requires fast algorithms for fundamental algebraic operations.

By Krull-Schmidt-Azumaya, finitely generated persistence modules still \emph{have} a decomposition into indecomposable summands that is unique up to permutation. Computing such a decomposition can accelerate the computation of additive invariants. This can make otherwise intractable invariants practically computable for large modules. We have shown this for the \emph{fibred barcode} \cite[Table 3]{djk} and the \emph{Skyscraper Invariant} \cite{FJ26}. To decompose any module $X$ over a finite dimensional algebra, classical algorithms \cite{CGK97, LuxSzoke}
compute the Endomorphism-algebra $\End(X) = \Hom(X,X)$ and then decompose this algebra. Algorithms specialised for persistence modules have a different strategy: the \emph{generalized persistence algorithm} \cite{DeyXin} performs an iterative matrix reduction and works for uniquely graded modules. In \cite{djk}, we realised that this matrix reduction is equivalent to computing the set of homomorphisms between \emph{indecomposable} submodules of certain projective covers (\cite[Section 4]{djk_arXiv}) of $X$. This allowed the development of the faster and generally applicable decomposition algorithm \textsc{aida} \cite[Section 5]{djk_arXiv}.

For both types of algorithms, the computation of homomorphisms is the bottleneck. Since our focus is on improving \textsc{aida} there is an emphasis on computing homomorphisms between \emph{indecomposable} modules.

\subparagraph{Homomorphisms.}\label{par:direct}
In practice \cite{lw-interactive, mpfree}, persistence modules are given to us by their presentations.
\[ \text{Consider \ } \quad  A[R] \oto{M} A[G] \to X \to 0 \quad \text{ and } \quad A[R'] \oto{N} A[G'] \to Y \to 0\] where $A[G]\coloneqq \bigoplus_{g \in G} A[-\deg(g)]$ is the graded free $A$-module with basis $G$ and the maps $M$, $N$ are \emph{graded matrices} (\autoref{def:graded_matrix}) over $\K$. We would like to store a map $X \to Y$ as a matrix in the coordinates given by the generators $G$ and $G'$.
A straightforward way to obtain such lifts of $\Hom(X,Y)$ is to directly compute the vector space of morphisms of presentations, $\Hom(M, N)$, as the set of maps $P, Q$ such that the following diagram commutes.
% https://q.uiver.app/#q=WzAsOCxbMCwwLCJBW1JdIl0sWzEsMCwiQVtHXSJdLFswLDEsIkFbUiddIl0sWzEsMSwiQVtHXSJdLFsyLDAsIlgiXSxbMiwxLCJZIl0sWzMsMCwiMCJdLFszLDEsIjAiXSxbMCwxLCJNIl0sWzIsMywiTiJdLFswLDIsIlAiLDJdLFsxLDMsIlEiXSxbMSw0XSxbMyw1XSxbNCw1LCJcXHdpZGV0aWxkZSBRIiwwLHsic3R5bGUiOnsiYm9keSI6eyJuYW1lIjoiZG90dGVkIn19fV0sWzQsNl0sWzUsN11d
\[\begin{tikzcd}
	{A[R]} & {A[G]} & X & 0 \\
	{A[R']} & {A[G]} & Y & 0
	\arrow["M", from=1-1, to=1-2]
	\arrow["P"', from=1-1, to=2-1]
	\arrow[from=1-2, to=1-3]
	\arrow["Q", from=1-2, to=2-2]
	\arrow[from=1-3, to=1-4]
	\arrow["{\widetilde Q}", dotted, from=1-3, to=2-3]
	\arrow["N", from=2-1, to=2-2]
	\arrow[from=2-2, to=2-3]
	\arrow[from=2-3, to=2-4]
\end{tikzcd}\]
Equivalently, we solve the \emph{linear} system of equations
 \begin{equation}\label{eq:naive}
 QM - NP = 0 \quad \text{ for graded matrices} \ Q \colon A[G]  \to A[G'], \ P \colon A[R]  \to A[R'].
 \end{equation}
If the sizes of $M$ and $N$ are bounded by $n \times n$, then using Gauss-elimination, solving \autoref{eq:naive} takes $\Oc(n^6)$ time. In this paper we will answer

\begin{question}\label{qu:n6}
    Can we compute $\Hom(X,Y)$ faster than in $\Oc(n^6)$ time\footnotemark?
\end{question}
\footnotetext{with standard matrix reduction}

\subparagraph{Other Applications in Topological Data Analysis.}\label{par:applications}
The need to compute homomorphisms arose in two other computational problems. 
\begin{itemize}
\item Bjerkevik introduces the \emph{Pruning} of a module in \cite{Bjerkevik2025} which stabilises decompositions with respect to the interleaving distance. To compute it, one needs to compute the shifted endomorphisms $\Hom\left(X,\, X{[2\epsilon]}\right)$ \cite[Lemma 5.2.]{Bjerkevik2025}.

\item In \cite{ASASHIBA2025107905}, Asashiba proves that the relative Betti numbers for the interval-resolution of a module $X$ can be computed by a complex of vector spaces $\Hom(\mathcal{K}_*, X)$, where each module in the complex $\mathcal{K}_*$ is interval-decomposable. 

\end{itemize}

\subparagraph{Related Work.}

In \cite[4.1 Algorithm Interleaving]{dey_xin}, Dey and Xin explain a strategy to compute the set of homomorphisms between interval modules in linear time for 2 parameters and  in quadratic time in the general case.
For representations over finite-dimensional algebras, optimised algorithms for homomorphism computation are described by Lux and Sz\H{o}ke \cite{luxsz03} and by Green, Heath, and Struble in \cite{ghs01}, the latter of which we adapt to our special case. Implementations are also found in GAP \cite{GAP4} and MAGMA \cite{MAGMA}. 

\subparagraph{Acknowledgements}

I am grateful for Michael Kerber's vital and ongoing support for my dissertation, of which this paper is a part. His work on algorithmic problems in Multiparameter Persistence has been a huge influence on this paper.

 Much of this paper was written when I visited the Kyoto University Institute for Advanced Study. I want to thank Yasuaki Hiraoka and his whole team for their hospitality during this time, and especially Justin Descrochers, Enhao Liu, and Pavel Pangrac for enlightening discussions about representation theory and interval approximations, which led to considering the dual algorithms in this paper. I also want to thank Manon Tacconi.

At last I am indebted to Fabian Lenzen and Ezra Miller for discussions and explanations about duality theorems for multiparameter persistence modules.

\section{Introduction}

\subsection*{The empirical structure of MPM}

We computed indecomposable decompositions of the persistent homology groups of dimensions $1$ and $2$ of some typical \emph{density-scale} bi-filtrations (density-rips \cite{CZMP, BMT17}, multicover \cite{Sheehy12, EO18, CKLO23}) (\autoref{fig:layer_thick}). This revealed something surprising about the indecomposable summands.

\begin{figure}[H]
    \centering
    \begin{minipage}[t]{0.48\textwidth}
    \vspace{0pt}
        \centering
        \includegraphics[width=\textwidth]{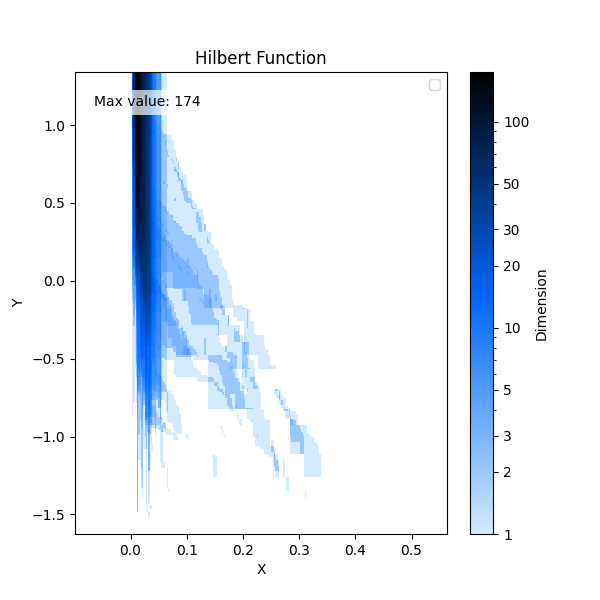}
    \end{minipage}
    \hfill
    \begin{minipage}[t]{0.48\textwidth}
    \vspace{0pt}
        \centering
        \includegraphics[width=0.48\textwidth]{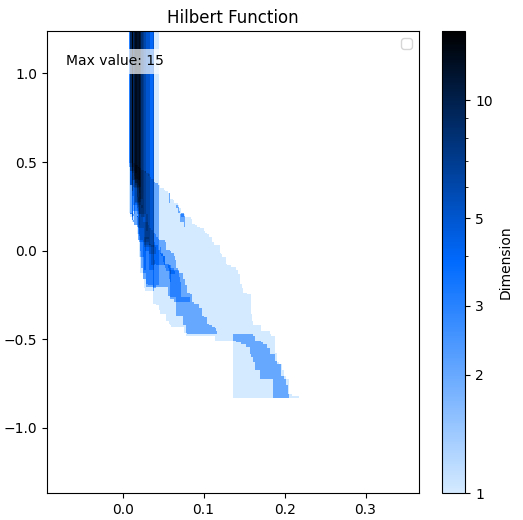}
        \hfill
        \includegraphics[width=0.48\textwidth]{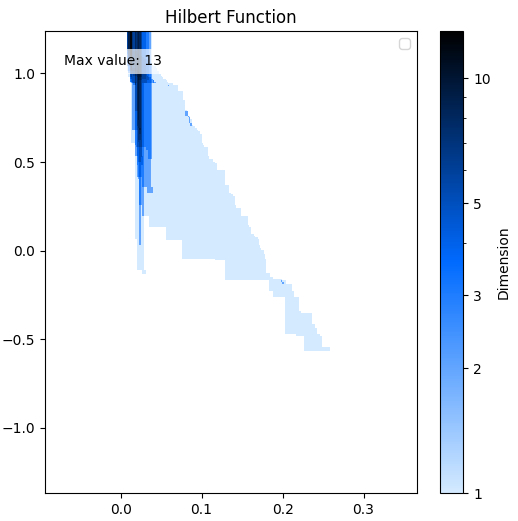}
        
        \vspace{0.5em}
        
        \includegraphics[width=0.48\textwidth]{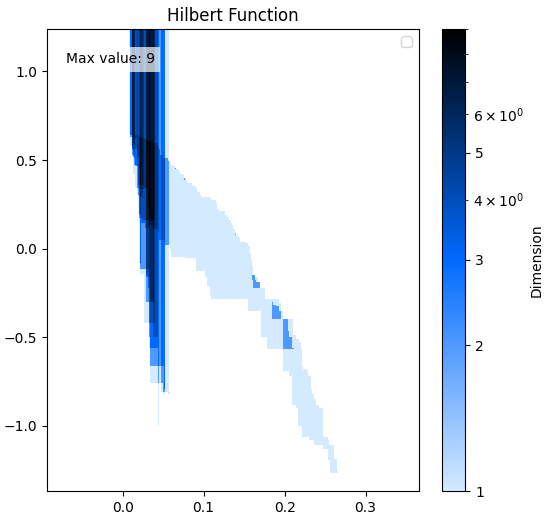}
        \hfill
        \includegraphics[width=0.48\textwidth]{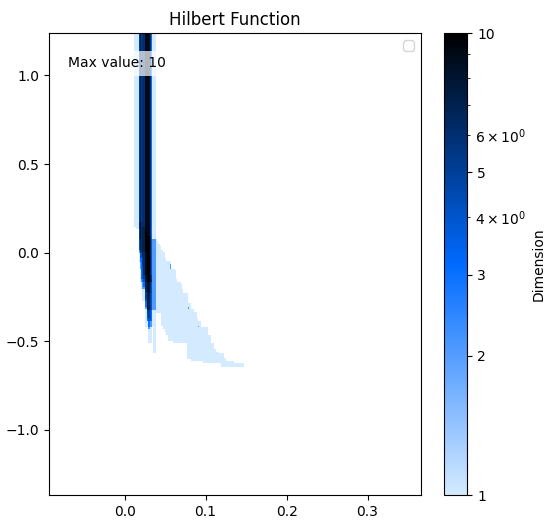}
    \end{minipage}
    \caption{Hilbert function of $H_1$ of a density-delaunay filtration of the \emph{Large Hypoxic Regions} (FOXP3) dataset from \cite{Vipond21} on a $200 \times 200$ grid (left). Hilbert functions of its 4 largest indecomposable summands (right). This module has 5917 generators, and 74\% of these belong to an interval-summand}
    \label{fig:hypoxic}
\end{figure}

A few summands have many generators and a large area of support. At small scale-values, many short lived homology classes appear and are glued 
together by simplices which appear in the filtration at larger parameter pairs. This creates a high-dimensional vertical strip of noise.

For higher scale-values, we observe a large region that is
mostly of low dimension, often even $1$. We have plotted the pointwise dimension (Hilbert function) of $4$ such modules in \autoref{fig:hypoxic}.
Most of the other summands have a very small support (they are "noise") and are of dimension $1$ everywhere, which, for $2$ parameters, is equivalent to being an interval-module. For example, the module used for \autoref{fig:hypoxic} has 5917 generators, and 74\% of these belong to an interval-summand.

\subparagraph{Thickness.} We need a notion to describe this behaviour.
\begin{definition}\label{def:thickness}
    Let $X \colon \R^d \to \vect_\K$. We define its \emph{thickness} as
    \[ \thickb{X} \coloneqq \max\limits_{\alpha \in \R^d} \dim_\K X_\alpha \]
    and its \emph{layer-thickness} as the maximal thickness of its indecomposable summands
    \[ \lthickb{X} \coloneqq \max\limits_{Y \in \text{Indec}(X)} \thickb{Y} . \]
\end{definition}

We will never encounter 
layer-thickness-$1$ modules (equivalently, \emph{interval decomposable} modules for $2$ parameters)
in the setting of these bi filtrations for larger pointsets as shown by Alonso, Kerber, and Skraba in \cite{AKS24}. Yet, they are apparently often not \emph{so} far 
from it. For $H_0$, it is even proven that at least $25\% $ of the indecomposable summands are intervals \cite[Theorem 5.1]{AlonsoKerber23}. 

\begin{remark}\label{rem:layers}
 We picture $2$-parameter 
 persistence modules as consisting of thin layers with crumbled edges, 
 sometimes interwoven, like the sheets in \autoref{fig:metal}.
 \end{remark}
Since $\Hom(X,Y)$ can be computed quickly for interval modules  \cite{dey_xin}, we investigated whether something similar can be done for modules of \emph{low} pointwise dimension.

\subsection*{Contributions}

To illustrate the setting, Figure~\ref{fig:hom1} shows a homomorphism between two low-thickness two-parameter persistence modules. 
\begin{figure}[H]
\centering
\tdplotsetmaincoords{70}{30}
\begin{tikzpicture}[tdplot_main_coords, scale=1]

  % ---- BLUE MODULE (z = 0) ----
  % base blue shape
  \fill[lightblue, opacity=1]
    (0,1,0) -- (0,4,0) -- (5,4,0) -- (5,0,0) -- (1,0,0) -- (1,1,0) -- cycle;
  \fill[darkblue, opacity=1]
    (1,1,0) -- (1,4,0) -- (1.9,4,0) -- (1.9,1.9,0) -- (5,1.9,0) -- (5,1,0) -- cycle;

    \foreach \yi/\opacity in {4/0.7, 4.3/0.5, 4.6/0.3, 5/0} {
      \pgfmathsetmacro{\yUpper}{ifthenelse(\yi+0.3 > 5, 5, \yi+0.3)}
      \fill[lightblue, opacity=\opacity]
        (0, \yi, 0) -- (0, \yUpper, 0) -- (5, \yUpper, 0) -- (5, \yi, 0) -- cycle;
      \fill[darkblue, opacity=\opacity]
        (1, \yi, 0) -- (1, \yUpper, 0) -- (2, \yUpper, 0) -- (2, \yi, 0) -- cycle;
    }
  
  % ---- RED MODULE (z = 1) ----
  \fill[lightred, opacity=1]
    (2,2,1) -- (2,4,1) -- (6,4,1) -- (6,2,1) -- cycle;

  \foreach \yi/\opacity in {4/0.7, 4.3/0.5, 4.6/0.3, 5/0} {
  \pgfmathsetmacro{\yUpper}{ifthenelse(\yi+0.3 > 5, 5, \yi+0.3)}
  \fill[lightred, opacity=\opacity]
    (2, \yi, 1) -- (2, \yUpper, 1) -- (6, \yUpper, 1) -- (6, \yi, 1) -- cycle;
    }
    
% Vertical homomorphism arrows with bold transparent double stroke
\foreach \x/\y in {2/2, 6/2} {
  % Two parallel lines for the shaft (no arrowheads)
  \draw[black, line width=1.5pt]
    (\x-0.05,\y,0.9) -- (\x-0.05,\y,0.15);
  \draw[black, line width=1.5pt]
    (\x+0.05,\y,0.9) -- (\x+0.05,\y,0.15);
  % Single arrowhead (lower)
  \draw[->, black, line width=1.5pt]
    (\x,\y,0.15) -- (\x,\y,0.05);
}

  % --- GENERATORS & RELATIONS for blue module (z = 0) ---
  \foreach \x/\y in {0/1, 1/0} {
    \fill[blue] (\x,\y,0) circle (2.5pt);
  }
  \foreach \x/\y in {1.9 / 1.9, 5/0, 5/1} {
    \node[rectangle, draw=black, fill=blue, inner sep=1.5pt] at (\x,\y,0) {};
  }

  % --- GENERATORS & RELATIONS for red module (z = 1) ---
  \foreach \x/\y in {2/2} {
    \fill[red] (\x,\y,1) circle (2.5pt);
  }
  \foreach \x/\y in {6/2} {
    \node[rectangle, draw=black, fill=red, inner sep=1.5pt] at (\x,\y,1) {};
  }

    % Axes (x and y only)
  \draw[->, thick] (-0.1,-0.1,0) -- (5.4,-0.1,0) node[below] {$x$};
  \draw[->, thick] (-0.1,-0.1,0) -- (-0.1,5.4,0) node[above] {$y$};
  
\end{tikzpicture}
\caption{A homomorphism. The red module has thickness 1, the blue module thickness 2.}
\label{fig:hom1}
\end{figure}
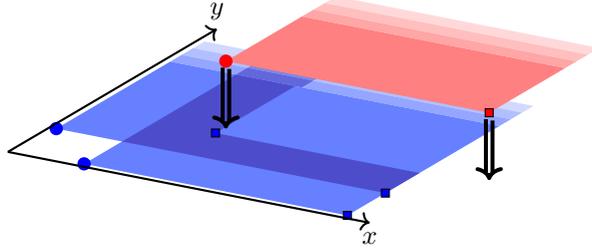

\subparagraph{Unique lifts.}
Our main contribution is to reduce the number of variables in \autoref{eq:naive}. 
Let $\dots\oto{d_2} F_1 \oto{d_1} F_0 \oto{d_0} X \to 0$ and $\dots\oto{e_2} F'_1 \oto{e_1} F'_0 \oto{e_0} Y \to 0$ be \emph{minimal} free resolutions of $X, Y$.
\begin{definition}
For each $i \in \N$ the $i$-th \emph{Syzygy}-module of $Y$ is
$ \Omega^{i}(Y) \coloneqq \Img(e_i) = \ker(e_{i-1})$.
In particular, $\Omega^{0}(Y) = Y$ and there are surjections $\bar e_i \colon F'_i \onto \Omega^i(Y)$.
\end{definition}

\begin{definition}\label{def:restriction_system}
Let $p \colon Y \onto X$ be a surjection and $B \colon (\Z^d)_\text{obj} \to \Set$ a $\Z^d$-graded set. A $B$-\emph{restriction system} for $p$ is a collection of sub vector spaces $V^b \subset Y_{\deg(b)}$ indexed by $b \in B$ such that $(p_{\deg(b)})_{|V^b} \colon V^b \to X_{\deg(b)}$ is surjective. It is called \emph{sharp} if $(p_{\deg(b)})_{|V^b} $ is an isomorphism.
\end{definition}

It follows almost immediately from the definition that every map $A[B] \to X$ has a unique lift along $p$. This allows for a stronger version of the Fundamental Lemma of Homological Algebra.

\ignore{
\begin{definition}\label{def:restriction_system}
    For each $i \in \N$, let $B_i \subset F_i$ be a basis of the free modules in the resolution of $X$. 
    A $(B_i)_{i \in \N}$-\emph{restriction system} for $F'_*$ is an $\N$-indexed set of restriction systems $\left( \left( V^b \right)_{b \in B_i} \right)_{i \in \N}$ for the induced maps $\bar e_i \colon F'_i \onto \Omega^i(Y)$.
    It is \emph{sharp} if all $(V^b)_{b \in B_i}$ are sharp.
\end{definition}
}

\begin{tcolorbox}[thmbox]
\begin{restatable}[]{theorem}{mainthm}\label{thm:main}
Let $(A[B_i], d_i)$ and $(A[C_i], e_i)$ be free resolutions of $X$ and $Y$. 

For each $i$ and $b \in B_i$, there are subsets $C_i^b \subset C_i$ of size $\dim_\K \Omega^{i}(Y)_{\deg(b)}$, such that every homomorphism $f \colon X \to Y$ has a unique lift $f_i \colon A[B_i] \to A[C_i]$ 
satisfying
\[ f_i(b) \in A[C^b_i] \subset A[C_i]\]
\end{restatable}
\end{tcolorbox}

\autoref{thm:main} implies a good bound for the size of the space of homomorphisms. By counting the number of all maps from $A[B_i]$ to $A[C]$ satisfying the condition it holds that

\begin{corollary}\label{cor:hom_size}
    \[\dim_K \Hom(X,Y) \leq b_0(X) \thick{Y}.\]
\end{corollary}

Computing in \autoref{eq:naive} only those $M\to N$ satisfying this condition results in \autoref{alg:restricted}.

\begin{theorem}\label{thm:alg_a}
Given presentation matrices of $X$ and $Y$ with at most $n$ rows and columns,
    \autoref{alg:restricted} computes a basis of $\ \Hom(X, \, Y)$ in time 
$\Oc( n^4 \left( \thick{Y} + \thick{  \Omega^{1}(Y) } \right)^2 + T_{\ker}(d, n))$ time.
\end{theorem}
$T_{\ker}(d,n)$ is the time to compute the kernel of a map between free modules of rank at most $n$.

By controlling the lift of $f$ with \autoref{thm:main} to the free covers only, we construct \autoref{alg:mixed}. We already used this algorithm in experiments with \textsc{aida} \cite[Table 3]{djk} without proof of correctness; \autoref{thm:main} closes this gap.

 \subparagraph{Classical Approach}  
  Green, Heath, and Struble in \cite{ghs01} present an ansatz to compute homomorphisms between modules over an arbitrary algebra using the right-exactness of the tensor-product.
We show that this approach, when specialised to persistence modules, also yields a fast algorithm if $\thickb{Y}$ is small. For this paper, we will call it \autoref{alg:hom_exact}.
It has a slightly better theoretical time-complexity than \autoref{alg:restricted} in $\Oc\left( n^3 \thickb{Y}^3 + n^4 \right)$ but has a larger overhead.

\subparagraph{Applications and Experiments.}
\autoref{alg:mixed} accelerated \textsc{aida} on some instances in previous experiments \cite{djk}, but the asymptotic worst case run-time is not better than the direct approach.
\autoref{alg:hom_exact}, the classical approach, has the best worst-case complexity and using it in \textsc{aida} results in

\begin{tcolorbox}[thmbox]
\begin{restatable}[]{corollary}{cor_aida}\label{cor:aida}
    Using \autoref{alg:hom_exact}, the runtime of \textsc{AIDA} on a minimal presentation matrix of size $n \times n$ is in $\Oc \left( n^{\omega +1} \left(   n +  k^{\omega -1}(k+\lthick{X})q^{k^2/4 + \Oc(k)} + \sum_{Y \in \text{Ind}(X)} \thickb{Y}^\omega  \right) \right)$.
\end{restatable}
\end{tcolorbox}
$\omega$ is the matrix multiplication constant and $k$ the maximal number of relations of the same degree.

 We tested all three algorithms on benchmark datasets. Surprisingly, we observe that for practically appearing presentations, \autoref{alg:mixed} is often actually the fastest of the three algorithms. If we do not need a basis for $\Hom(X,Y)$, but just a set of generators, then for small modules, the direct approach can keep up with the new algorithms. This explains why we did not observe a more significant speed up of \textsc{aida} in the experiments in \cite[Table 3]{djk}.

\subparagraph{Duality.}
In Asashiba's construction \cite{ASASHIBA2025107905} it is not the target module, but the domain which has low thickness. 
Therefore, we also need to be able to compute homomorphisms faster in this case.
We utilised the ideas and results of Bauer, Lenzen, and Lesnick \cite{BLL23}, Miller \cite{miller00, Miller2001} and Lenzen \cite{lenzen24}
to dualise \autoref{alg:restricted} and \autoref{alg:hom_exact} resulting in the algorithms \autoref{alg:restricted_dual} and \autoref{alg:exact_dual}. Implicitly, we also get a dual version of \autoref{alg:mixed}.

\subparagraph{Sparsity and Thickness.}
The algorithms we will present often run faster in practice than the theoretical bounds would suggest because the persistent homology group of classical filtered simplicial complexes has a sparse presentation in practice. We store the presentation matrix in a list-of-columns format, and this is why we generally use column-reduction whenever possible.

To improve the analysis of our algorithms, we wanted to quantify the sparsity of presentation matrices in relation to the modules that they represent. In \cite[Section 6]{djk_arXiv} we used that interval modules have a minimal presentation with at most $2$ non-zero entries per column to show that interval decomposable modules can be decomposed fast. We extend this observation (\cite[Proposition 6.13]{djk_arXiv}) to all persistence modules.

\begin{proposition}\label{prop:thickness_sparsity}
Every persistence module $X$ has a presentation with at most $\thick{X}+1$ non-zero entries per column.
\end{proposition}

\subsection*{Datastructures}

\subparagraph{Input of our algorithms.}
For $2$-parameter modules, minimising a presentation matrix can be done via column-reduction in $\Oc(n^3)$ using the LWKR algorithm \cite{lw-computing, mpfree}. In $3$ or more parameters one can use Schreyer's algorithm \cite{schreyer} or the faster but unpublished work of Bender, Gävfert, and Lesnick \cite{kernels}.
We therefore assume the input to all of our algorithms to be \emph{minimal} presentation matrices $M \in \K^{G \times R}$ and $N \in \K^{G' \times R'}$.
We will denote their sizes throughout the paper by $m \coloneqq |G| + |R|$ and $n \coloneqq |G'| + |R'|$ 

\subparagraph{Output of our algorithms.}

Since we are given presentations, we also want the output of our algorithms to be compatible with this data structure. That is, we want a set of graded matrices $ \{Q_i \}_{i \in I} \subset \K^{G' \times G} $ such that each $Q_i$ induces a homomorphism of persistence modules $\widetilde Q_i \colon X \to Y$ and such that the subset $\{\widetilde Q_i\}_{i \in I} \subset \Hom(X, \, Y)$ generates - or better is a basis of - this vector space.

\ignore{
\section{Old Introduction}
In our work on decompositions of multiparameter persistence modules (MPM), that is functors $\R^d \to \vect_\K$, we encountered the need to compute the vector space of homomorphisms between (indecomposable) modules:
\begin{itemize}
    \item 
To decompose a module $X$, classical algorithms \cite{CGK97, LuxSzoke} would 
compute the Endomorphism-algebra $\End(X,X)$ and decompose, instead, this algebra. To see how one would compute the underlying 
vector space $\Hom(X,Y)$ naively, consider two modules $X,Y$ which, in practice,
 are given to us as free presentations $F_1 \oto{d} F_0 \to X$ and $F_1' \oto{e} F'_0 \to  Y$ \cite{lw-interactive, mpfree}. We can then compute the vector space of morphisms of presentations, $\Hom(F_*, F'_*)$, as the set of maps $p,q$ which make the following diagram commute.
% https://q.uiver.app/#q=WzAsNixbMCwwLCJGXzEiXSxbMSwwLCJGXzAiXSxbMCwxLCJGXzEnIl0sWzEsMSwiRl8wJyJdLFsyLDAsIlgiXSxbMiwxLCJZIl0sWzAsMSwiZCJdLFsyLDMsImUiXSxbMCwyLCJwIl0sWzEsMywicSJdLFsxLDRdLFszLDVdLFs0LDUsIiIsMSx7InN0eWxlIjp7ImJvZHkiOnsibmFtZSI6ImRvdHRlZCJ9fX1dXQ==
\[\begin{tikzcd}[ampersand replacement=\&]
	{F_1} \& {F_0} \& X \\
	{F_1'} \& {F_0'} \& Y
	\arrow["d", from=1-1, to=1-2]
	\arrow["p", from=1-1, to=2-1]
	\arrow[from=1-2, to=1-3]
	\arrow["q", from=1-2, to=2-2]
	\arrow[dotted, from=1-3, to=2-3]
	\arrow["e", from=2-1, to=2-2]
	\arrow[from=2-2, to=2-3]
\end{tikzcd}\]
After choosing bases $B_0, B_1, B_0', B_1'$ for the free modules, these maps become \emph{graded matrices}. Finding the maps $p,q$ is therefore the same as solving, for graded matrices $P,Q$, the \emph{linear} equation
 \begin{equation}\label{eq:naive}
 Q[d]^{B_1}_{B_0} - [e]^{B_1'}_{B_0'}P = 0 \quad \text{ for } \ Q \in \K^{B_0' \times B_0}, \ P \in \K^{B_1' \times B_1}.
 \end{equation}
Then by deleting those morphisms of presentations which induce the $0$-map, one gets a basis for $\Hom(X,Y)$ coordinatised in the chosen set of generators $B_0$ and $B_0'$.
\item We realised that in the \emph{generalized persistence algorithm} \cite{DeyXin} a subroutine actually contained the linear system \autoref{eq:naive} for two \emph{indecomposable} modules $X,Y$. We showed that one can instead also compute the (smaller) vector space $\Hom(X,Y)$ which paved the way for many parts of the general decomposition algorithm \textsc{aida} \cite{djk}.
\end{itemize}

Using the Gauss-algorithm, solving \autoref{eq:naive} takes $\Oc(n^6)$ time. We will ignore fast matrix multiplication in this paper. Not only for clarity, but also because it does not play a role in practice, where our matrices are always \emph{sparse} and the actual time needed for reducing a $c \times c$ matrix is between $\sim c^1$ and $\sim c^2$.
This computation of homomorphisms is, both theoretically and practically, 
the computational bottleneck in both the classical decomposition algorithms, in the \emph{generalized persistence algorithm}, and in \textsc{aida}. This prompted us to look for ways to compute 
$\Hom(X,Y)$ faster than in $\Oc(n^6)$ time.

 \subparagraph{Thickness.}

We tested \textsc{aida} on the persistent homology of typical \emph{density-scale} bi-filtrations (density-rips \cite{CZMP, BMT17} , degree-rips \cite{lw-interactive}, and multicover \cite{Sheehy12, EO18, CKLO23}), which can be used to estimate the homology stably with respect to outliers (\cite{blumberg_lesnick}). This revealed something surprising about the indecomposable modules that arose here as summands. Most were \emph{pointwise low-dimensional} or at least had large regions of pointwise dimension 1 \autoref{fig:hypoxic}. Additionally, most summands are very small ("noise") and of dimension $1$, which for $2$ parameters is equivalent to being an interval-module. 

\begin{figure}[H]
    \centering
    \begin{minipage}[t]{0.48\textwidth}
    \vspace{0pt}
        \centering
        \includegraphics[width=\textwidth]{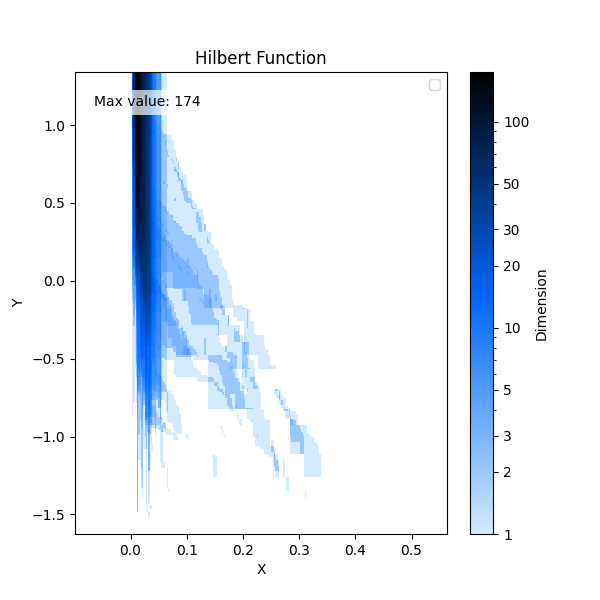}
    \end{minipage}
    \hfill
    \begin{minipage}[t]{0.48\textwidth}
    \vspace{0pt}
        \centering
        \includegraphics[width=0.48\textwidth]{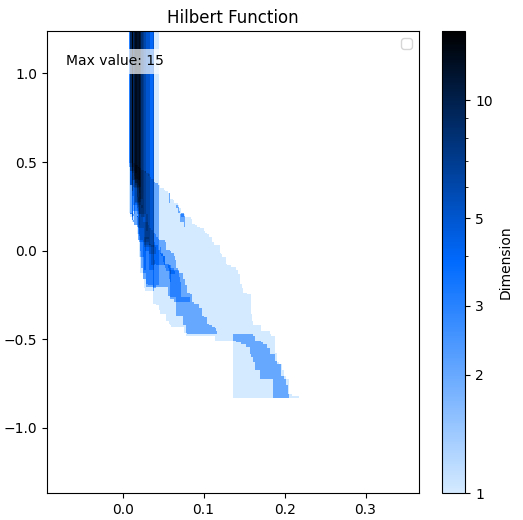}
        \hfill
        \includegraphics[width=0.48\textwidth]{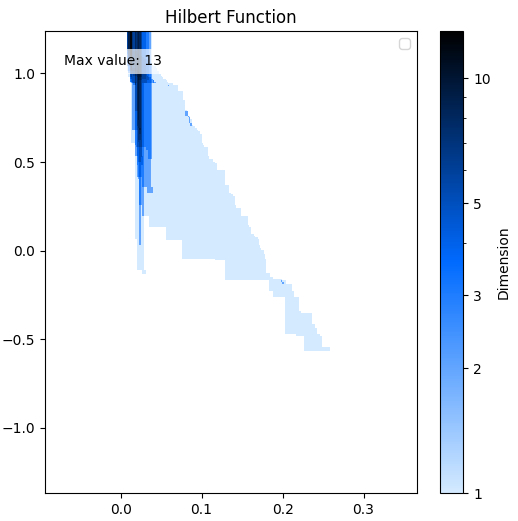}
        
        \vspace{0.5em}
        
        \includegraphics[width=0.48\textwidth]{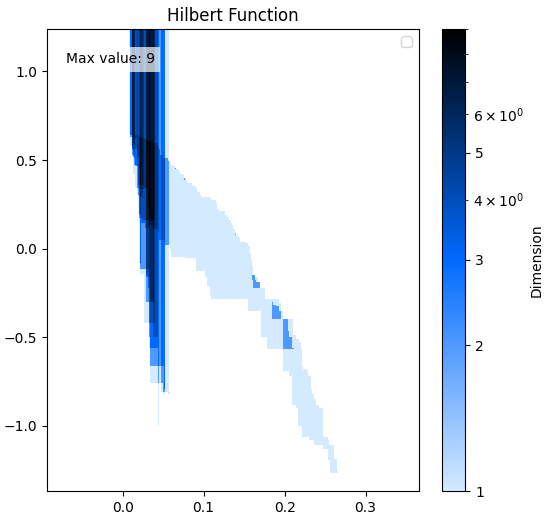}
        \hfill
        \includegraphics[width=0.48\textwidth]{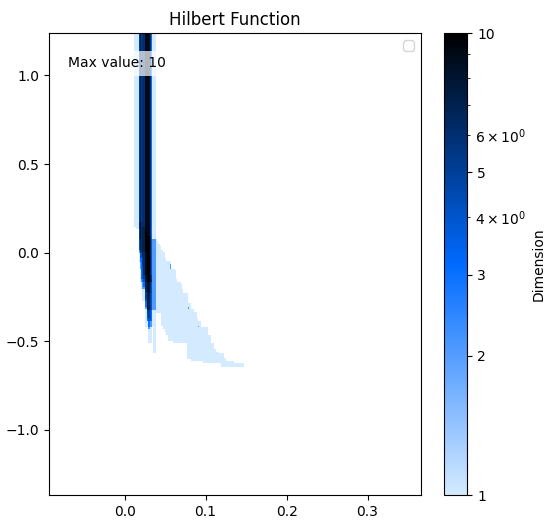}
    \end{minipage}
    \caption{Hilbert function of $H_1$ of a density-delaunay filtration of the \emph{Large Hypoxic Regions} (FOXP3) dataset from \cite{Vipond21} on a $200 \times 200$ grid (left). Hilbert functions of its 4 largest indecomposable summands (right). This module has 5917 generators, and 74\% of these belong to an interval-summand}
    \label{fig:hypoxic}
\end{figure}

In contrast, at small scale-parameters, many short lived homology classes appear and are glued 
together by simplices appearing at larger parameter pairs. 
This creates a large strip where a few indecomposable summands
 have pointwise very high dimension. We therefore picture these $2$-parameter 
 persistence modules as consisting of thin layers with crumbled edges, 
 sometimes interwoven, like the sheets in \autoref{fig:metal}.

We also already know from \cite{dey_xin} that, for interval modules, $\Hom(X,Y)$ can be computed in linear time over $\R^2$ and in quadratic time for more parameters. This led us to investigate if something similar can be done for modules of (mostly) \emph{low} dimension.
The maximal pointwise dimension of a module $X$ - or of the indecomposable summands of $X$ - has already appeared as an important parameter in the work of Bjerkevik (\cite[Theorem 5.4, Conjecture 6.1]{Bjerkevik2025}, and in our unpublished work on the skyscraper invariant \cite{HN_discr}, so we propose the following terminology.

\begin{definition}
    Let $X$ be a persistence module over $\R^d$. We define its \emph{thickness} as
    \[\thickb{X} \coloneqq \max\limits_{\alpha \in \R^d} \dim_\K X_\alpha\]
    and its \emph{layer-thickness} as the maximal thickness of its indecomposable summands
    \[ \lthickb{X} \coloneqq \max\limits_{Y \in \text{Indec}(X)} \thickb{Y}. \]
\end{definition}

Although we know (\cite{AKS24}) that we will almost surely never encounter 
interval-decomposable, i.e. layer-thickness-$1$ modules 
in the setting of these bifiltrations, they are apparently often not far 
from it in practice (cf \autoref{fig:hypoxic}) - For $H_0$ it is even proven that at least $25\% $ of the indecomposable summands are intervals \cite[Theorem 5.1]{AlonsoKerber23}.

To illustrate the setting, Figure~\ref{fig:hom1} shows a homomorphism between two low-thickness two-parameter persistence modules. 
\begin{figure}[H]
\centering
\tdplotsetmaincoords{70}{30}
\begin{tikzpicture}[tdplot_main_coords, scale=1]

  % ---- BLUE MODULE (z = 0) ----
  % base blue shape
  \fill[lightblue, opacity=1]
    (0,1,0) -- (0,4,0) -- (5,4,0) -- (5,0,0) -- (1,0,0) -- (1,1,0) -- cycle;
  \fill[darkblue, opacity=1]
    (1,1,0) -- (1,4,0) -- (1.9,4,0) -- (1.9,1.9,0) -- (5,1.9,0) -- (5,1,0) -- cycle;

    \foreach \yi/\opacity in {4/0.7, 4.3/0.5, 4.6/0.3, 5/0} {
      \pgfmathsetmacro{\yUpper}{ifthenelse(\yi+0.3 > 5, 5, \yi+0.3)}
      \fill[lightblue, opacity=\opacity]
        (0, \yi, 0) -- (0, \yUpper, 0) -- (5, \yUpper, 0) -- (5, \yi, 0) -- cycle;
      \fill[darkblue, opacity=\opacity]
        (1, \yi, 0) -- (1, \yUpper, 0) -- (2, \yUpper, 0) -- (2, \yi, 0) -- cycle;
    }
  
  % ---- RED MODULE (z = 1) ----
  \fill[lightred, opacity=1]
    (2,2,1) -- (2,4,1) -- (6,4,1) -- (6,2,1) -- cycle;

  \foreach \yi/\opacity in {4/0.7, 4.3/0.5, 4.6/0.3, 5/0} {
  \pgfmathsetmacro{\yUpper}{ifthenelse(\yi+0.3 > 5, 5, \yi+0.3)}
  \fill[lightred, opacity=\opacity]
    (2, \yi, 1) -- (2, \yUpper, 1) -- (6, \yUpper, 1) -- (6, \yi, 1) -- cycle;
    }
    
% Vertical homomorphism arrows with bold transparent double stroke
\foreach \x/\y in {2/2, 6/2} {
  % Two parallel lines for the shaft (no arrowheads)
  \draw[black, line width=1.5pt]
    (\x-0.05,\y,0.9) -- (\x-0.05,\y,0.15);
  \draw[black, line width=1.5pt]
    (\x+0.05,\y,0.9) -- (\x+0.05,\y,0.15);
  % Single arrowhead (lower)
  \draw[->, black, line width=1.5pt]
    (\x,\y,0.15) -- (\x,\y,0.05);
}

  % --- GENERATORS & RELATIONS for blue module (z = 0) ---
  \foreach \x/\y in {0/1, 1/0} {
    \fill[blue] (\x,\y,0) circle (2.5pt);
  }
  \foreach \x/\y in {1.9 / 1.9, 5/0, 5/1} {
    \node[rectangle, draw=black, fill=blue, inner sep=1.5pt] at (\x,\y,0) {};
  }

  % --- GENERATORS & RELATIONS for red module (z = 1) ---
  \foreach \x/\y in {2/2} {
    \fill[red] (\x,\y,1) circle (2.5pt);
  }
  \foreach \x/\y in {6/2} {
    \node[rectangle, draw=black, fill=red, inner sep=1.5pt] at (\x,\y,1) {};
  }

    % Axes (x and y only)
  \draw[->, thick] (-0.1,-0.1,0) -- (5.4,-0.1,0) node[below] {$x$};
  \draw[->, thick] (-0.1,-0.1,0) -- (-0.1,5.4,0) node[above] {$y$};
  
\end{tikzpicture}
\caption{A homomorphism. The red module has thickness 1, the blue module thickness 2.}
\label{fig:hom1}
\end{figure}

If the target module has low thickness, any generator of the domain can only be sent to a low dimensional vector space. Using this idea we modified the standard $\Oc(n^6)$ algorithm for computing homomorphisms of presentations. This accelerated \textsc{aida} in practice for some inputs (see \cite{djk}, Table~X). The resulting algorithm, which we call the \emph{mixed} algorithm (\autoref{alg:mixed}) here, is sometimes faster in practice, but its worst-case asymptotic complexity is still $\Oc(n^6)$. The main goal of this paper is to close this gap.

\subparagraph{Contributions.}

Let now, for full generality, 
$\dots \oto{ \ d_2 \ } F_1 \oto{ \ d_1 \ } F_0 \oto{ \ d_0 \ } X $ 
and $\dots \oto{ \ e_2 \ } F'_1 \oto{ \ e_1 \ } F'_0 \oto{ \ e_0 \ } Y $ 
be minimal free resolutions of persistence modules and $f\colon X  \to Y$ be a homomorphism.
To obtain a strictly faster algorithm, we investigated how much we can restrict the chain maps $F_* \to F'_*$ so that a lift of $f$ can always be found within this restriction. Recall that $\Omega^{i}(Y) \coloneqq \Img(e_i)$ denotes the $i$-th syzygie module of $Y$.

\begin{definition}
    Let $B_* \subset F_*$ be any basis. A set of vector subspaces
    $V^b_{i} \subset \left(F'_i\right)_{\deg(b)}$ for each $i \in \N$ and $b \in B_i$
     is called a \emph{restriction system} for the pair ($F_*, \, F_*'$) if, for each $i \in \N$ and $b \in B_i$ the map
     \[ \pi_i \colon F'_i\onto \Omega^{i}(Y) \]
     is still a surjection after restriction to $V^b_{i}$ at $\deg(b)$. It is called \emph{sharp} if this restriction  is an isomorphism.
\end{definition}

This lets us formulate a stronger version of the Fundamental Lemma of Homological Algebra for persistence modules.

\begin{theorem}\label{thm:main}
In the setting above, if $\{V^b_*\}_{b \in B_*}$ is a restriction system, 
then every homomorphism $f \colon X \to Y$ has a lift $f_* \colon F_* \to F'_*$ 
which satisfies for each $b \in B_*$ and each $i \in \N$ the condition
\[ f_i(b) \in V^b_i\]
and if the restriction system is sharp, then this lift is unique.
\end{theorem}

Applying this theorem directly yields \autoref{alg:restricted}, which reduces the run-time complexity of computing $\Hom(X, \, Y)$ to $\Oc\left( n^4 (\thickb{Y} + \thickb{\Omega^1 Y} ) \right)$.

 Computing homomorphisms of modules over finite dimensional algebras is a 
 classical problem in computational algebra. 
 When reviewing the literature we realised that the Ansatz presented by Green, Heath, and Struble 
 in \cite{ghs01} leads, in the special case of persistence modules, also to an algorithm 
which is faster than $\Oc(n^6)$ if $\thickb{Y}$ is small. We call this algorithm "\autoref{alg:hom_exact}" and it has a slightly better theoretical time-complexity of $\Oc\left( n^4 \thickb{Y}^2\right)$ but a larger overhead than \autoref{alg:restricted}.

\begin{corollary}
If $X$ is of size $n$ and has uniquely-graded relations, then \autoref{alg:hom_exact} improves the runtime of \textsc{AIDA} to $\Oc(n^4 \lthickb{X}^3 + n^5)$ when using standard matrix reduction.
\end{corollary}

\autoref{alg:mixed} used in \cite{djk} without proof of correctness. \autoref{thm:main} also closes this gap. Because it has a much smaller overhead than both \autoref{alg:restricted} and \autoref{alg:hom_exact} and was already shown to be effective in \cite{djk} we also included it in our comparative analysis. Surprisingly, we observe that on practically appearing presentations, it is often actually the fastest of the three algorithms.

In \cite{ASASHIBA2025107905} (see next paragraph), it is not the target module, but the domain which has low thickness. 
Therefore we also need to be able to compute homomorphisms faster in this case.
To this end, we utilised the ideas and results of Bauer, Lenzen, and Lesnick \cite{BLL23}, Miller \cite{miller00, Miller2001} and Lenzen \cite{lenzen24}
to dualise \autoref{alg:restricted} and \autoref{alg:hom_exact} resulting in the algorithms \autoref{alg:restricted_dual} and \autoref{alg:exact_dual}. Implicitly, we also get a dual version of \autoref{alg:mixed}.

\subparagraph{Applications in Topological Data Analysis.}

Beyond decompositions, the task of computing homomorphisms arose in two other computational problems in Topological Data Analysis. 
\begin{itemize}
\item Bjerkevik introduces the \emph{Pruning} of a module in \cite{Bjerkevik2025} which promises to stabilise decompositions. To compute the Pruning of a module $X$ one needs to compute the shifted endomorphisms $\Hom\left(X,\, X{[2\epsilon]}\right)$. Using \texttt{AIDA} this reduces the problem to computing, again, homomorphisms between indecomposable modules which we know empirically to have lower thickness.

\item In \cite{ASASHIBA2025107905} Asashiba proves that the relative Betti numbers for the interval-resolutions of a module $X$ can be computed as the homology of a complex of vector spaces $\Hom(\mathcal{K}_*, X)$, where each module in the complex $\mathcal{K}_*$ is interval-decomposable.

\end{itemize}

\subparagraph{Related Work.}

In \cite[4.1 Algorithm Interleaving]{dey_xin}, Dey and Xin explain a strategy to compute the set of homomorphisms between interval modules in linear time for 2-parameters and  in quadratic time in the general case.

For representations over general finite-dimensional algebras, optimised algorithms for homomorphism computation are described by Lux and Sz\H{o}ke \cite{luxsz03} and by Green, Heath, and Struble in the aforementioned \cite{ghs01}. Implementations are available in GAP \cite{GAP4} and MAGMA \cite{MAGMA}. 

A work of similar spirit to ours is due to Schneider \cite{Schneider90}, who exploited the structure of modular group representations to develop faster algorithms for endomorphism rings and general homomorphisms via an iterative procedure.

\subparagraph{Input of our algorithms.}

Minimising a presentation matrix can be done via column-reduction in $\Oc(n^3)$ using the LWKR algorithm \cite{lw-computing, mpfree}, so we can assume the input to all of our algorithms to be \emph{minimal} presentation matrices $M$ and $N$.

We denote by $m, n \in \N$ the number of rows plus the number of columns of $M$ and $N$ respectively, i.e. $ m =|b_0 \br{X} | + |b_1 \br{X} |$ and $ n =|b_0 \br{Y} | + |b_1 \br{Y} |$.

\subparagraph{Output of our algorithms.}

Since we are usually given the persistence modules $X, \, Y$ via their minimal presentations, we want the output of our algorithms to be compatible with this data structure. That is, we want a set of graded matrices $ \{Q_i \}_{i \in I} \subset \K^{G' \times G} $ such that each $Q_i$ induces a homomorphism of persistence modules $\widetilde Q_i \colon X \to Y$ and such that the set $\{\widetilde Q_i\}_{i \in I} \subset \Hom(X, \, Y)$ at least spans - but better even forms a basis of - the $\K$-vector space $\Hom(X, \, Y)$.

}

\section{Resolutions of Multiparameter Persistence Modules}

\subparagraph{Notation}

\begin{itemize}
  \item $A \coloneqq \K[x_1, \dots, x_d]$ denotes the $\Z^d$-graded commutative polynomial algebra.
  \item $X, Y, \dots$ denote (persistence) modules, capital letters $M, N, \dots$ are reserved for matrices.
  \item Greek minuscules $\alpha, \beta, \dots, \omega \in \Z^d$ denote points in the \emph{poset} $\Z^d$. 
\end{itemize}

We recall some basic definitions. We refer to \cite{lenzen24} and \cite{djk_arXiv} for a more comprehensive introduction.

\begin{definition}
 A ($d$-parameter) \emph{persistence module} $X$ over $\K$ is a commutative diagram of $\K$-vector spaces over $\Z^d$ or equivalently a $\Z^d$-graded $A$-module. 
 
 The category of persistence modules is denoted $\grA$.
\end{definition}

\begin{definition}
A \emph{free resolution} of a persistence module $X$ is an exact sequence 
\[  \dots \oto{ \ d_3 \ } F_2 \oto{ \ d_2 \ } F_1 \oto{ \ d_1 \ } F_0 \oto{ \ d_0 \ } X \lra 0 \]
with each $F_i$ free. 
A \emph{presentation} of $X$ is the head $F_1 \oto{ \ d_1 \ } F_0$ of a resolution of $X$.
\end{definition}

\begin{remark}
Every projective persistence module is \emph{free}, even if not finitely generated, by a graded version of Kaplansky's theorem (eg. \cite[Proposition 1.5.15(d)]{Bruns_Herzog}.
\end{remark}

\begin{notation}
Let $\alpha \in \Z^n$.
$F(\alpha)$ and $I(\alpha)$ denote the free and co-free modules generated at $\alpha$. 

We denote by $X[\alpha]$ the $\alpha$-shift (-twist) of $X$, defined by $X[\alpha]_\beta \coloneqq X_{\alpha + \beta}$. Therefore $A[-\alpha] \simeq F(\alpha)$.
\end{notation}

\begin{example}
Graphical representations of the above notations.
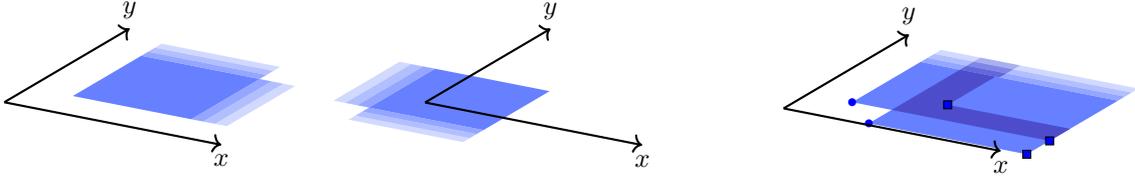
\begin{figure}[H]
\centering
\tdplotsetmaincoords{70}{30}
\begin{tikzpicture}[tdplot_main_coords, scale=0.6]

% Define horizontal shift between copies
\def\xshift{8}
\def\yshift{4.6}

% First copy at x=0 - blue shape starting at (1.9, 1.9) extending in both directions
\begin{scope}[shift={(0,0,0)}]

    % Base blue shape
    \fill[lightblue, opacity=1] (1, 1, 0) -- (1.0, 4, 0) -- (4, 4, 0) -- (4, 1, 0) -- cycle;

    % Fadeout in y direction
    \foreach \yi/\opacity in {4/0.7, 4.3/0.5, 4.6/0.3, 5/0} {
        \pgfmathsetmacro{\yUpper}{ifthenelse(\yi+0.3 > 5, 5, \yi+0.3)}
        \fill[lightblue, opacity=\opacity] (1.0, \yi, 0) -- (1.0, \yUpper, 0) -- (4, \yUpper, 0) -- (4, \yi, 0) -- cycle;
    }
    
    % Fadeout in x direction
    \foreach \xi/\opacity in {4/0.7, 4.3/0.5, 4.6/0.3, 5/0} {
        \pgfmathsetmacro{\xUpper}{ifthenelse(\xi+0.3 > 5, 5, \xi+0.3)}
        \fill[lightblue, opacity=\opacity] (\xi, 1.0, 0) -- (\xUpper, 1.0, 0) -- (\xUpper, 4, 0) -- (\xi, 4, 0) -- cycle;
    }

    % Axes labels
    \draw[->, thick] (-0.1,-0.1,0) -- (5.4,-0.1,0) node[below] {$x$};
    \draw[->, thick] (-0.1,-0.1,0) -- (-0.1,5.4,0) node[above] {$y$};
\end{scope}

% Second copy - mirrored version starting at (1.9, 1.9) extending in negative directions
\begin{scope}[shift={(\xshift,\yshift,0)}]
    % Base blue shape (mirrored)
    \fill[lightblue, opacity=1] (1.9, 1.9, 0) -- (1.9, -1, 0) -- (-1, -1, 0) -- (-1, 1.9, 0) -- cycle;
    
    % Fadeout in negative y direction
    \foreach \yi/\opacity in {-1/0.7, -1.3/0.5, -1.6/0.3, -2/0} {
        \pgfmathsetmacro{\yLower}{ifthenelse(\yi-0.3 < -2, -2, \yi-0.3)}
        \fill[lightblue, opacity=\opacity] (1.9, \yi, 0) -- (1.9, \yLower, 0) -- (-1, \yLower, 0) -- (-1, \yi, 0) -- cycle;

    }
    
    % Fadeout in negative x direction
    \foreach \xi/\opacity in {-1/0.7, -1.3/0.5, -1.6/0.3, -2/0} {
        \pgfmathsetmacro{\xLower}{ifthenelse(\xi-0.3 < -2, -2, \xi-0.3)}
        \fill[lightblue, opacity=\opacity] (\xi, 1.9, 0) -- (\xLower, 1.9, 0) -- (\xLower, -1, 0) -- (\xi, -1, 0) -- cycle;
    }
 
    % Axes labels
    \draw[->, thick] (-0.1,-0.1,0) -- (5.4,-0.1,0) node[below] {$x$};
    \draw[->, thick] (-0.1,-0.1,0) -- (-0.1,5.4,0) node[above] {$y$};
\end{scope}

% Third copy - same as original but shifted by +1 in x direction
\begin{scope}[shift={(2*\xshift-1,2*\yshift-1,0)}]
    \fill[lightblue, opacity=1] (1,1,0) -- (1,4,0) -- (6,4,0) -- (6,0,0) -- (2,0,0) -- (2,1,0) -- cycle;
    \fill[darkblue, opacity=1] (2,1,0) -- (2,4,0) -- (2.9,4,0) -- (2.9,1.9,0) -- (6,1.9,0) -- (6,1,0) -- cycle;
    
    \foreach \yi/\opacity in {4/0.7, 4.3/0.5, 4.6/0.3, 5/0} {
        \pgfmathsetmacro{\yUpper}{ifthenelse(\yi+0.3 > 5, 5, \yi+0.3)}
        \fill[lightblue, opacity=\opacity] (1, \yi, 0) -- (1, \yUpper, 0) -- (6, \yUpper, 0) -- (6, \yi, 0) -- cycle;
        \fill[darkblue, opacity=\opacity] (2, \yi, 0) -- (2, \yUpper, 0) -- (3, \yUpper, 0) -- (3, \yi, 0) -- cycle;
    }
    
    \foreach \x/\y in {1/1, 2/0} {
        \fill[blue] (\x,\y,0) circle (2.5pt);
    }
    
    \foreach \x/\y in {2.9/1.9, 6/0, 6/1} {
        \node[rectangle, draw=black, fill=blue, inner sep=1.5pt] at (\x,\y,0) {};
    }
    
    % Axes labels
    \draw[->, thick] (-0.1,-0.1,0) -- (5.4,-0.1,0) node[below] {$x$};
    \draw[->, thick] (-0.1,-0.1,0) -- (-0.1,5.4,0) node[above] {$y$};
\end{scope}

\end{tikzpicture}
\caption{Left: The free/projective module $A[-(1, 1)] \simeq F\left(1, 1 \right)$. $\ $ Middle: The injective module $I\left(2,2\right)$. $\quad$ Right: the blue module from \autoref{fig:hom1} shifted by $(-1,0)$.}
\label{fig:free}
\end{figure}
\end{example}

\begin{definition}\label{def:injective}
An \emph{injective resolution} of $X$ is an exact sequence
\[ 0 \lra X \oto{ \ \delta_0} I_0 \oto{ \ \delta_1 \ } I_1 \oto{ \ \delta_2 \ } I_2 \oto{ \ \delta_2 \ } \dots  \]
where each $I_i$ is injective. An \emph{injective copresentation} of $X$ is the tail $X \to I_0 \oto{ \ \delta_1 \ } I_1$ of an injective resolution.
\end{definition}

\subparagraph{Bases of Free Modules.}

\begin{definition}
Let $F$ be a free module. A subset $G \subset F$ is a \emph{basis} for $F$ if the map 
\[\varphi_B \colon A[G] \coloneqq \bigoplus\limits_{g \in G} A[-\deg(g)] \to F\]
which sends each $1 \in A[-\deg(g)]$ to $g$ is an isomorphism.
\end{definition}

Let $f\colon X \to Y$ be a map between free modules. Choosing bases $G \subset X$ and $R \subset Y$ induces
\begin{align}\label{eq:gr_matrix}
[f]^R_G \colon A[R] = 
\bigoplus_{r \in R} A[-\deg(r)] \oto{\varphi_R^{-1}}  X
 \oto{f}   
 Y \oto{\varphi_G} \bigoplus_{g \in G} A[-\deg(g)] = A[G].
\end{align}
This is a $G \times R$-\emph{matrix}. 
Because of degree-incompatibilities, some of its entries have to be zero.

\begin{definition}\label{def:graded_matrix}
 A \emph{graded matrix} is a matrix $M \in \K^{m \times n}$ together with functions
 $G \colon [m] \to \Z^d$ and $R \colon [n] \to \Z^d$ which decorate the rows and columns with degrees such that
 \[ \forall (i,j) \in [m] \times [n]: \ \text{ If } R(j) \not\geq G(i) \text{ then } M_{i,j}=0 .\] 
\end{definition}

This structure was first introduced in \cite{miller00} under the name \emph{monomial matrix}.

\begin{definition}
 Given functions $G,R$ as above, we write $\K^{G \times R}$ for the corresponding space of graded matrices. 
 We also define the (ordered) set
$\{G \leq R \}  \coloneqq  \set{(g, r) \in G \times R \ | \ g \leq r } $. 

If $G, \, R$ are instead (ordered) subsets of persistence modules, we also write by abuse of notation $\K^{G \times R}$ and $\{G \leq R \}$ instead of $\K^{\deg(G) \times \deg(R)}$ and $\{ \deg(G) \leq \deg(R) \}$. 
\end{definition}

 In the setting of \autoref{eq:gr_matrix} we have 
$ \Hom(X, \, Y) \simeq \K^{G \times R} = \K \{ G \leq R \}$.

Consequently, the maps in a resolution and the presentation can be represented by graded matrices. They are invariant under a change of basis represented by invertible graded matrices.

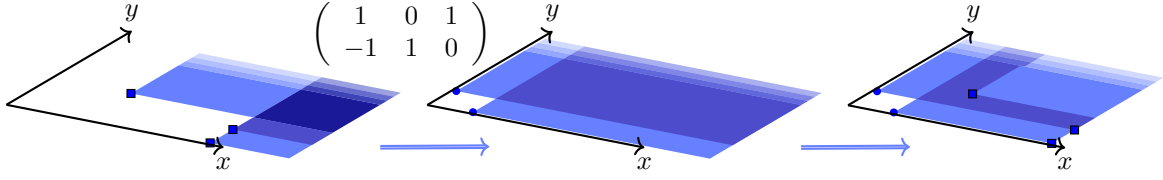
\begin{figure}[H]
\tdplotsetmaincoords{70}{30}
\begin{tikzpicture}[tdplot_main_coords, scale=0.6]
  % Define horizontal shift between copies
  \def\xshift{8} % adjust spacing as needed
    \def\yshift{4.6}
  % First copy at x=0
  \begin{scope}[shift={(0,0,0)}]
    % --- BLUE MODULE (z = 0) ---
    \fill[lightblue, opacity=1]
      (1.9, 1.9, 0) -- (1.9, 4 ,0) -- (5,4,0) -- (5.0, 0, 0) -- (7.0, 0, 0) -- (7.0, 1.9, 0) --  (5, 1.9, 0) --cycle;
      \fill[darkerblue, opacity=1]
      (5, 1.9, 0) -- (5,4,0) -- (7,4,0) -- (7,1.9,0) -- cycle;
      \fill[darkblue, opacity = 1 ]
      (5.0, 1.9, 0) -- (5.0, 1, 0) -- (7.0, 1, 0) -- (7.0, 1.9, 0) -- cycle;
      
    \foreach \yi/\opacity in {4/0.7, 4.3/0.5, 4.6/0.3, 5/0} {
      \pgfmathsetmacro{\yUpper}{ifthenelse(\yi+0.3 > 5, 5, \yi+0.3)}
      \fill[lightblue, opacity=\opacity]
        (1.9, \yi, 0) -- (1.9, \yUpper, 0) -- (7, \yUpper, 0) -- (7, \yi, 0) -- cycle;
      \fill[darkerblue, opacity=\opacity]
        (5, \yi, 0) -- (5, \yUpper, 0) -- (7, \yUpper, 0) -- (7, \yi, 0) -- cycle;  
    }
    % Generators and relations
    \foreach \x/\y in {1.9 / 1.9, 5/0, 5/1.0} {
      \node[rectangle, draw=black, fill=blue, inner sep=1.5pt] at (\x,\y,0) {};
    }
    % Axes labels
    \draw[->, thick] (-0.1,-0.1,0) -- (5.4,-0.1,0) node[below] {$x$};
    \draw[->, thick] (-0.1,-0.1,0) -- (-0.1,5.4,0) node[above] {$y$};
  \end{scope}
  % Second copy shifted right by \xshift and \yshift
  \begin{scope}[shift={(\xshift ,\yshift,0)}]
    % --- BLUE MODULE (z = 0) ---
    \fill[lightblue, opacity=1]
      (0,1,0) -- (0,4,0) -- (7,4,0) -- (7,0,0) -- (1,0,0) -- (1,1,0) -- cycle;
    \fill[darkblue, opacity=1]
      (1,1,0) -- (1,4,0) -- (7,4,0) -- (7,1,0) -- cycle;
    \foreach \yi/\opacity in {4/0.7, 4.3/0.5, 4.6/0.3, 5/0} {
      \pgfmathsetmacro{\yUpper}{ifthenelse(\yi+0.3 > 5, 5, \yi+0.3)}
      \fill[lightblue, opacity=\opacity]
        (0, \yi, 0) -- (0, \yUpper, 0) -- (7, \yUpper, 0) -- (7, \yi, 0) -- cycle;
      \fill[darkblue, opacity=\opacity]
        (1, \yi, 0) -- (1, \yUpper, 0) -- (7, \yUpper, 0) -- (7, \yi, 0) -- cycle;
    }
    % Generators and relations
    \foreach \x/\y in {0/1, 1/0} {
      \fill[blue] (\x,\y,0) circle (2.5pt);
    }
    % Axes labels
    \draw[->, thick] (-0.1,-0.1,0) -- (5.4,-0.1,0) node[below] {$x$};
    \draw[->, thick] (-0.1,-0.1,0) -- (-0.1,5.4,0) node[above] {$y$};
  \end{scope}
  % Third copy shifted right by 2*\xshift
  \begin{scope}[shift={(2*\xshift,2*\yshift,0)}]
    % (repeat entire module code without axes)
    \fill[lightblue, opacity=1]
      (0,1,0) -- (0,4,0) -- (5,4,0) -- (5,0,0) -- (1,0,0) -- (1,1,0) -- cycle;
    \fill[darkblue, opacity=1]
      (1,1,0) -- (1,4,0) -- (1.9,4,0) -- (1.9,1.9,0) -- (5,1.9,0) -- (5,1,0) -- cycle;
    \foreach \yi/\opacity in {4/0.7, 4.3/0.5, 4.6/0.3, 5/0} {
      \pgfmathsetmacro{\yUpper}{ifthenelse(\yi+0.3 > 5, 5, \yi+0.3)}
      \fill[lightblue, opacity=\opacity]
        (0, \yi, 0) -- (0, \yUpper, 0) -- (5, \yUpper, 0) -- (5, \yi, 0) -- cycle;
      \fill[darkblue, opacity=\opacity]
        (1, \yi, 0) -- (1, \yUpper, 0) -- (2, \yUpper, 0) -- (2, \yi, 0) -- cycle;
    }
    \foreach \x/\y in {0/1, 1/0} {
      \fill[blue] (\x,\y,0) circle (2.5pt);
    }
    \foreach \x/\y in {1.9 / 1.9, 5/0, 5/1} {
      \node[rectangle, draw=black, fill=blue, inner sep=1.5pt] at (\x,\y,0) {};
    }
    % Axes labels
    \draw[->, thick] (-0.1,-0.1,0) -- (5.4,-0.1,0) node[below] {$x$};
    \draw[->, thick] (-0.1,-0.1,0) -- (-0.1,5.4,0) node[above] {$y$};
  \end{scope}

% Matrix above the left blue arrow
\node at (\xshift - 2.9, \yshift + 3.6, 0) {$\left(\begin{array}{ccc} 1 & 0 & 1 \\ -1 & 1 & 0 \end{array}\right)$};

  % Arrows between copies  
% From copy 1 to 2 - lightblue arrow  
% Calculate perpendicular offset for parallel lines (small offset perpendicular to arrow direction)  
\draw[thick, lightblue]  
(\xshift - 1, \yshift - 0.6, -1 + 0.03) -- (\xshift + 1, \yshift + 0.6, -1 + 0.03);  
\draw[thick, lightblue]  
(\xshift - 1, \yshift - 0.6, -1 - 0.03) -- (\xshift + 1, \yshift + 0.6, -1 - 0.03);  
% Single arrowhead  
\draw[->, thick, lightblue]  
(\xshift + 1.02, \yshift + 0.612, -1) -- (\xshift + 1.05, \yshift + 0.63, -1);  
% From copy 2 to 3 - lightblue arrow  
\draw[thick, lightblue]  
(2*\xshift - 1, 2*\yshift - 0.6, -1 + 0.03) -- (2*\xshift + 1, 2*\yshift + 0.6, -1 + 0.03);  
\draw[thick, lightblue]  
(2*\xshift - 1, 2*\yshift - 0.6, -1 - 0.03) -- (2*\xshift + 1, 2*\yshift + 0.6, -1 - 0.03);  
% Single arrowhead  
\draw[->, thick, lightblue]  
(2*\xshift + 1.02, 2*\yshift + 0.612, -1) -- (2*\xshift + 1.05, 2*\yshift + 0.63, -1);  
\end{tikzpicture}
\caption{A presentation of the blue module in \autoref{fig:hom1}.}
\label{fig:res}
\end{figure}

\begin{definition}\label{prop:minimal_pres}\cite[7.1]{Peeva}
A free resolution is \emph{minimal} if $\forall i \in \N \colon d_i \otimes_A \K = 0$.
A presentation $P_1 \oto{d_1} P_0 \oto{d_0} X$ is minimal if and only if $d_1$ and its kernel $\ker f \oto{i}  P_1 $ satisfy
 \[d_1 \otimes_A \K  = 0 \ \text{ and }  \ i \otimes_A \K = 0.\]
\end{definition}

A free resolution can be minimised with the LWKR-algorithm \cite{lw-computing, mpfree}

\begin{definition}
The degrees of the generators in a minimal resolution of a module $X$ form multisets in $\Z^d$ called the \emph{graded Betti numbers}. We define them via
\[b_{i}(X) \coloneqq \dim_\K  \Tor_{i}(X, \K ) = \dim_\K  \Ext^{i}(X, \K ).\]
\end{definition}

\section{Computing Homomorphisms as Chain Maps}

For the rest of this text $X$ and $Y$ will denote persistence modules together with implicitly chosen minimal sets of generators $G, \, G'$ and relations $R, \, R'$. The induced minimal presentation \emph{matrices} are denoted by $M$ and $N$. Concretely we are given exact sequences

\begin{alignat}{3}
\dots \lra  & \bigoplus\limits_{r \in R}A[-\deg(r)] &
\oto{ \ M \ \left( \, \coloneqq \, [d_1]^R_G \right) \ }  &
\bigoplus\limits_{g \in G}A[-\deg(g)] &
\oto{ \ d_0 \ }  & 
X \lra 0 \quad  \text{ and } \\
\dots  \lra  & \bigoplus\limits_{r' \in R'}A[-\deg(r')] &
\oto{ \ \ N \ \left( \, \coloneqq \, [e_1]^{R'}_{G'} \right)  \ }   &
\bigoplus\limits_{g' \in G'}A[-\deg(g')] &
\oto{ \ e_0 \ } & 
Y \lra 0 .
\end{alignat}

\subsection{Chain Maps and the Direct Computation}

\begin{lemma}[Fundamental Lemma of Homological Algebra]\label{lem:fundamental}
\cite[III. Theorem 6.1]{maclane}  \ \\
Let $f \colon X \to Y$ be any homomorphism of persistence modules and $(F_*, d_*) \to X, \ (F'_*, e_*) \to Y$ projective resolutions.
Then there exists a chain map $\{f_i\}_{i \in \N}$ which lifts $f$:
% https://q.uiver.app/#q=WzAsMTAsWzMsMCwiWCJdLFszLDEsIlkiXSxbMiwwLCJYXzAiXSxbMiwxLCJZXzAiXSxbMSwwLCJYXzEiXSxbMSwxLCJZXzEiXSxbMCwwLCJcXGRvdHMiXSxbMCwxLCJcXGRvdHMiXSxbNCwwLCIwIl0sWzQsMSwiMCJdLFs2LDQsImRfMiJdLFs3LDUsImVfMiJdLFs0LDIsImRfMSJdLFs1LDMsImVfMSJdLFsyLDAsImRfMCJdLFszLDEsImVfMCJdLFswLDEsImYiXSxbMiwzLCJmXzAiXSxbNCw1LCJmXzEiXSxbMCw4XSxbMSw5XSxbNiw3LCJcXGRvdHMiLDEseyJzdHlsZSI6eyJib2R5Ijp7Im5hbWUiOiJub25lIn0sImhlYWQiOnsibmFtZSI6Im5vbmUifX19XV0=
\[\begin{tikzcd}[ampersand replacement=\&]
	\dots \& {F_1} \& {F_0} \& X \& 0 \\
	\dots \& {F'_1} \& {F'_0} \& Y \& 0
	\arrow["{d_2}", from=1-1, to=1-2]
	\arrow["\dots"{description}, draw=none, from=1-1, to=2-1]
	\arrow["{d_1}", from=1-2, to=1-3]
	\arrow["{f_1}", from=1-2, to=2-2]
	\arrow["{d_0}", from=1-3, to=1-4]
	\arrow["{f_0}", from=1-3, to=2-3]
	\arrow[from=1-4, to=1-5]
	\arrow["f", from=1-4, to=2-4]
	\arrow["{e_2}", from=2-1, to=2-2]
	\arrow["{e_1}", from=2-2, to=2-3]
	\arrow["{e_0}", from=2-3, to=2-4]
	\arrow[from=2-4, to=2-5]
\end{tikzcd}\]

If $f_i$ and $\widetilde f_i$ are chain maps which both lift $f$ then they are homotopic. That is there are maps $h_i \colon F_i \to F'_{i+1}$ satisfying
\[ f_i - \widetilde f_i = e_{i+1} \circ h_i + h_{i-1} \circ d_i.\]
\end{lemma}

\begin{example}
We can visualise chain maps graphically. In the following \autoref{fig:chain} you can see the presentations and a lift of the homomorphism shown in \autoref{fig:hom1}. 

In this example we can also see that there are more chain maps between the projective resolutions than there are homomorphisms between the red and blue module: While in the rightmost picture the generator of the red module can only be mapped to a one-dimensional vector space lying below, instead on the level of projective covers (the middle picture) there is a two-dimensional vector space below the generator.
\begin{figure}[H]
\tdplotsetmaincoords{70}{30}

\begin{tikzpicture}[tdplot_main_coords, scale=0.6]

  % Define horizontal shift between copies
  \def\xshift{8} % adjust spacing as needed
    \def\yshift{4.6}
  % First copy at x=0
  \begin{scope}[shift={(0,0,0)}]
    % Put your entire module drawing here (copy your full code without \begin{tikzpicture} and \end{tikzpicture})
    % --- BLUE MODULE (z = 0) ---
    \fill[lightblue, opacity=1]
      (1.9, 1.9, 0) -- (1.9, 4 ,0) -- (5,4,0) -- (5.0, 0, 0) -- (7.0, 0, 0) -- (7.0, 1.9, 0) --  (5, 1.9, 0) --cycle;
      \fill[darkerblue, opacity=1]
      (5, 1.9, 0) -- (5,4,0) -- (7,4,0) -- (7,1.9,0) -- cycle;
      \fill[darkblue, opacity = 1 ]
      (5.0, 1.9, 0) -- (5.0, 1, 0) -- (7.0, 1, 0) -- (7.0, 1.9, 0) -- cycle;

    \foreach \yi/\opacity in {4/0.7, 4.3/0.5, 4.6/0.3, 5/0} {
      \pgfmathsetmacro{\yUpper}{ifthenelse(\yi+0.3 > 5, 5, \yi+0.3)}
      \fill[darkblue, opacity=\opacity]
        (1.9, \yi, 0) -- (1.9, \yUpper, 0) -- (7, \yUpper, 0) -- (7, \yi, 0) -- cycle;
      \fill[darkerblue, opacity=\opacity]
        (5, \yi, 0) -- (5, \yUpper, 0) -- (7, \yUpper, 0) -- (7, \yi, 0) -- cycle;  
        
    }
    
    % --- RED MODULE (z = 1) ---
    \fill[lightred, opacity=1]
      (6,4,1) -- (6,2,1) -- (8,2,1) -- (8,4,1) -- cycle;
    \foreach \yi/\opacity in {4/0.7, 4.3/0.5, 4.6/0.3, 5/0} {
      \pgfmathsetmacro{\yUpper}{ifthenelse(\yi+0.3 > 5, 5, \yi+0.3)}
      \fill[lightred, opacity=\opacity]
        (8, \yi, 1) -- (8, \yUpper, 1) -- (6, \yUpper, 1) -- (6, \yi, 1) -- cycle;
    }
    
    % First scope - replace the vertical depth arrows
\foreach \x/\y in {6/2} {
  % Two parallel lines for the shaft (no arrowheads)
  \draw[black, line width=1.0pt, opacity=0.8]
    (\x-0.05,\y,0.9) -- (\x-0.05,\y,0.15);
  \draw[black, line width=1.0pt, opacity=0.8]
    (\x+0.05,\y,0.9) -- (\x+0.05,\y,0.15);
  % Single arrowhead
  \draw[->, black, line width=1.5pt, opacity=0.8]
    (\x,\y,0.15) -- (\x,\y,0.05);
}

    % Generators and relations
    \foreach \x/\y in {1.9 / 1.9, 5/0, 5/1} {
      \node[rectangle, draw=black, fill=blue, inner sep=1.5pt] at (\x,\y,0) {};
    }
    \foreach \x/\y in {6/2} {
      \node[rectangle, draw=black, fill=red, inner sep=1.5pt] at (\x,\y,1) {};
    }
    % Axes labels 
    \draw[->, thick] (-0.1,-0.1,0) -- (5.4,-0.1,0) node[below] {$x$};
    \draw[->, thick] (-0.1,-0.1,0) -- (-0.1,5.4,0) node[above] {$y$};
  \end{scope}

  % Second copy shifted right by \xshift and \yshift
  \begin{scope}[shift={(\xshift ,\yshift,0)}]
    
    % --- BLUE MODULE (z = 0) ---
    \fill[lightblue, opacity=1]
      (0,1,0) -- (0,4,0) -- (7,4,0) -- (7,0,0) -- (1,0,0) -- (1,1,0) -- cycle;
    \fill[darkblue, opacity=1]
      (1,1,0) -- (1,4,0) -- (7,4,0) -- (7,1,0) -- cycle;
    \foreach \yi/\opacity in {4/0.7, 4.3/0.5, 4.6/0.3, 5/0} {
      \pgfmathsetmacro{\yUpper}{ifthenelse(\yi+0.3 > 5, 5, \yi+0.3)}
      \fill[lightblue, opacity=\opacity]
        (0, \yi, 0) -- (0, \yUpper, 0) -- (7, \yUpper, 0) -- (7, \yi, 0) -- cycle;
      \fill[darkblue, opacity=\opacity]
        (1, \yi, 0) -- (1, \yUpper, 0) -- (7, \yUpper, 0) -- (7, \yi, 0) -- cycle;
    }
    % --- RED MODULE (z = 1) ---
    \fill[lightred, opacity=1]
      (2,2,1) -- (2,4,1) -- (8,4,1) -- (8,2,1) -- cycle;
    \foreach \yi/\opacity in {4/0.7, 4.3/0.5, 4.6/0.3, 5/0} {
      \pgfmathsetmacro{\yUpper}{ifthenelse(\yi+0.3 > 5, 5, \yi+0.3)}
      \fill[lightred, opacity=\opacity]
        (2, \yi, 1) -- (2, \yUpper, 1) -- (8, \yUpper, 1) -- (8, \yi, 1) -- cycle;
    }
    
    % Second scope - replace the vertical depth arrows  
\foreach \x/\y in {2/2} {
  % Two parallel lines for the shaft (no arrowheads)
  \draw[black, line width=1.0pt, opacity=0.8]
    (\x-0.05,\y,0.9) -- (\x-0.05,\y,0.15);
  \draw[black, line width=1.0pt, opacity=0.8]
    (\x+0.05,\y,0.9) -- (\x+0.05,\y,0.15);
  % Single arrowhead
  \draw[->, black, line width=1.5pt, opacity=0.8]
    (\x,\y,0.15) -- (\x,\y,0.05);
}

    % Generators and relations
    \foreach \x/\y in {0/1, 1/0} {
      \fill[blue] (\x,\y,0) circle (2.5pt);
    }

    \foreach \x/\y in {2/2} {
      \fill[red] (\x,\y,1) circle (2.5pt);
    }
    
    % Axes labels 
    \draw[->, thick] (-0.1,-0.1,0) -- (5.4,-0.1,0) node[below] {$x$};
    \draw[->, thick] (-0.1,-0.1,0) -- (-0.1,5.4,0) node[above] {$y$};
    
  \end{scope}

  % Third copy shifted right by 2*\xshift
  \begin{scope}[shift={(2*\xshift,2*\yshift,0)}]
    % (repeat entire module code without axes)
    \fill[lightblue, opacity=1]
      (0,1,0) -- (0,4,0) -- (5,4,0) -- (5,0,0) -- (1,0,0) -- (1,1,0) -- cycle;
    \fill[darkblue, opacity=1]
      (1,1,0) -- (1,4,0) -- (1.9,4,0) -- (1.9,1.9,0) -- (5,1.9,0) -- (5,1,0) -- cycle;
    \foreach \yi/\opacity in {4/0.7, 4.3/0.5, 4.6/0.3, 5/0} {
      \pgfmathsetmacro{\yUpper}{ifthenelse(\yi+0.3 > 5, 5, \yi+0.3)}
      \fill[lightblue, opacity=\opacity]
        (0, \yi, 0) -- (0, \yUpper, 0) -- (5, \yUpper, 0) -- (5, \yi, 0) -- cycle;
      \fill[darkblue, opacity=\opacity]
        (1, \yi, 0) -- (1, \yUpper, 0) -- (2, \yUpper, 0) -- (2, \yi, 0) -- cycle;
    }
    \fill[lightred, opacity=1]
      (2,2,1) -- (2,4,1) -- (6,4,1) -- (6,2,1) -- cycle;
    \foreach \yi/\opacity in {4/0.7, 4.3/0.5, 4.6/0.3, 5/0} {
      \pgfmathsetmacro{\yUpper}{ifthenelse(\yi+0.3 > 5, 5, \yi+0.3)}
      \fill[lightred, opacity=\opacity]
        (2, \yi, 1) -- (2, \yUpper, 1) -- (6, \yUpper, 1) -- (6, \yi, 1) -- cycle;
    }

    % Third scope - replace the vertical depth arrows
\foreach \x/\y in {2/2, 6/2} {
  % Two parallel lines for the shaft (no arrowheads)
  \draw[black, line width=1.0pt, opacity=0.8]
    (\x-0.05,\y,0.9) -- (\x-0.05,\y,0.15);
  \draw[black, line width=1.0pt, opacity=0.8]
    (\x+0.05,\y,0.9) -- (\x+0.05,\y,0.15);
  % Single arrowhead
  \draw[->, black, line width=1.5pt, opacity=0.8]
    (\x,\y,0.15) -- (\x,\y,0.05);
}

    \foreach \x/\y in {0/1, 1/0} {
      \fill[blue] (\x,\y,0) circle (2.5pt);
    }
    \foreach \x/\y in {1.9 / 1.9, 5/0} {
      \node[rectangle, draw=black, fill=blue, inner sep=1.5pt] at (\x,\y,0) {};
    }
    \foreach \x/\y in {2/2} {
      \fill[red] (\x,\y,1) circle (2.5pt);
    }
    \foreach \x/\y in {6/2} {
      \node[rectangle, draw=black, fill=red, inner sep=1.5pt] at (\x,\y,1) {};
    }
    % Axes labels 
    \draw[->, thick] (-0.1,-0.1,0) -- (5.4,-0.1,0) node[below] {$x$};
    \draw[->, thick] (-0.1,-0.1,0) -- (-0.1,5.4,0) node[above] {$y$};
  \end{scope}

% Arrows between copies  
% From copy 1 to 2 - lightblue arrow  
% Calculate perpendicular offset for parallel lines (small offset perpendicular to arrow direction)  
\draw[thick, lightblue]  
(\xshift - 1, \yshift - 0.6, -1 + 0.03) -- (\xshift + 1, \yshift + 0.6, -1 + 0.03);  
\draw[thick, lightblue]  
(\xshift - 1, \yshift - 0.6, -1 - 0.03) -- (\xshift + 1, \yshift + 0.6, -1 - 0.03);  
% Single arrowhead  
\draw[->, thick, lightblue]  
(\xshift + 1.02, \yshift + 0.612, -1) -- (\xshift + 1.05, \yshift + 0.63, -1);  
% From copy 1 to 2 - lightred arrow  
\draw[thick, lightred]  
(\xshift + 0.5, \yshift + 0.3, 1 + 0.03) -- (\xshift + 2, \yshift + 1.2, 1 + 0.03);  
\draw[thick, lightred]  
(\xshift + 0.5, \yshift + 0.3, 1 - 0.03) -- (\xshift + 2, \yshift + 1.2, 1 - 0.03);  
% Single arrowhead  
\draw[->, thick, lightred]  
(\xshift + 2.02, \yshift + 1.212, 1) -- (\xshift + 2.05, \yshift + 1.23, 1);  
% From copy 2 to 3 - lightblue arrow  
\draw[thick, lightblue]  
(2*\xshift - 1, 2*\yshift - 0.6, -1 + 0.03) -- (2*\xshift + 1, 2*\yshift + 0.6, -1 + 0.03);  
\draw[thick, lightblue]  
(2*\xshift - 1, 2*\yshift - 0.6, -1 - 0.03) -- (2*\xshift + 1, 2*\yshift + 0.6, -1 - 0.03);  
% Single arrowhead  
\draw[->, thick, lightblue]  
(2*\xshift + 1.02, 2*\yshift + 0.612, -1) -- (2*\xshift + 1.05, 2*\yshift + 0.63, -1);  
% From copy 2 to 3 - lightred arrow  
\draw[thick, lightred]  
(2*\xshift + 0.5, 2*\yshift + 0.3, 1 + 0.03) -- (2*\xshift + 2, 2*\yshift + 1.2, 1 + 0.03);  
\draw[thick, lightred]  
(2*\xshift + 0.5, 2*\yshift + 0.3, 1 - 0.03) -- (2*\xshift + 2, 2*\yshift + 1.2, 1 - 0.03);  
% Single arrowhead  
\draw[->, thick, lightred]  
(2*\xshift + 2.02, 2*\yshift + 1.212, 1) -- (2*\xshift + 2.05, 2*\yshift + 1.23, 1);
  
\end{tikzpicture}
\caption{A chain map lifting a homomorphism of persistence modules in \autoref{fig:hom1}}
\label{fig:chain}
\end{figure}
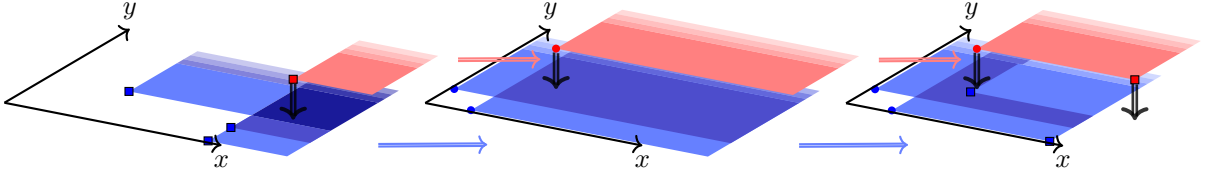

\end{example}

The first part of the fundamental lemma \autoref{lem:fundamental} guarantees the existence of the set of matrices $\{Q_j\}_{j \in J}\colon \K^{G' \times G}$.
If we want it to be minimal, i.e. for it to form a basis, we can use the second part of the fundamental lemma \autoref{lem:fundamental}:
\begin{corollary}\label{cor:nullhomotopic}
    A tuple of graded matrices $\{Q_j\}_{j \in J}$ in $\K^{G' \times G}$ descends to a linearly dependent set of homomorphisms $X \to Y$ if and only if there are $\lambda_j \in \K$, which are not all zero, and a graded matrix $H \colon A[G] \to A[R']$ such that
\begin{align}\label{eq:nullhomotopic}
\sum\limits_{j \in J} \lambda_j Q_j = N H.
\end{align}
\end{corollary}

\begin{proof}
The induced set of homomorphisms $\{\widetilde Q_i\}$ is linearly dependent iff a linear combination of any choice of lifts induces the zero map. The latter is equivalent to the existence of the homotopy $H$.  
\end{proof}

\subparagraph{Direct Computation}

We have already seen how to compute a basis of the vector space of morphisms of presentations in \autoref{eq:naive}. \autoref{cor:nullhomotopic} tells us how to extend this to an algorithm which computes a basis for $\Hom(X,Y)$.

\begin{algorithm}[H]
\caption{Computation from Morphisms of Presentations}
\label{alg:direct}
\DontPrintSemicolon
\KwIn{Presentation matrices $M \in \K^{G \times R}$, $N \in \K^{G' \times R'}$}
\KwOut{Graded matrices $\{Q_i\}_{i \in I} \in \K^{G' \times G}$ which descend to a basis of $\Hom(X, \, Y)$}
\label{gauss} Calculate a basis $(Q_j, \, P_j)_{j \in J} \in \K^{G' \times G}, \K^{R' \times R}$ of the solution space of the system $Q M = N P$ \;
Impose any order on $\set{G \leq G' }$ and set $d_Q \gets |\set{G \leq G' }|$\;
Treat the tuple $(Q_j)_{j \in J}$ as a matrix in $\K^{ d_Q \times J}$ \;
Initialise a matrix $\bar N \in \K^{d_Q \times \set{R' \leq G}}$ \;
\For{$(r', g) \in \set{R' \leq G}$}{ \label{line:homotopy_components}
    Initialise $N^{r',g} \in \K^{d_Q}$\;
    \For{$(\bar g, g') \in \set{G \leq G'}$}{
        \If{$\bar g = g$ }{
            $\left( N^{r',g}\right)_{g',\bar g} \coloneqq N_{g',r'}$\;
        } \Else {
            $\left( N^{r',g}\right)_{g',\bar g} \coloneqq 0$ \;
        }
    }
    $\bar N_{\bullet, (r',g)} \gets N^{r',g}$\;
}

Column-reduce the matrix $\left[ \bar N \ \ (Q_j) \right] \in \K^{d_Q \times (\{G \leq R'\} \, \cup \, J )}$\; \label{homotopy_reduction} 
$J' \gets \{j \in J \, | \ Q_j \neq 0 \}$\;
\KwRet{$\left(Q_j\right)_{j \in J'}$}
\end{algorithm}

\begin{example}\label{ex:running1}
We perform the computation for the example homomorphism in (\autoref{fig:hom1_repeat}). Consider the minimal presentations $M$ and $N$ of the red module, $X$, and the blue module, $Y$.

\begin{figure}[H]
    \centering
    \begin{minipage}[c]{0.48\textwidth}
    \tdplotsetmaincoords{70}{30}
\begin{tikzpicture}[tdplot_main_coords, scale=1]

  % ---- BLUE MODULE (z = 0) ----
  % base blue shape
  \fill[lightblue, opacity=1]
    (0,1,0) -- (0,4,0) -- (5,4,0) -- (5,0,0) -- (1,0,0) -- (1,1,0) -- cycle;
  \fill[darkblue, opacity=1]
    (1,1,0) -- (1,4,0) -- (1.9,4,0) -- (1.9,1.9,0) -- (5,1.9,0) -- (5,1,0) -- cycle;

    \foreach \yi/\opacity in {4/0.7, 4.3/0.5, 4.6/0.3, 5/0} {
      \pgfmathsetmacro{\yUpper}{ifthenelse(\yi+0.3 > 5, 5, \yi+0.3)}
      \fill[lightblue, opacity=\opacity]
        (0, \yi, 0) -- (0, \yUpper, 0) -- (5, \yUpper, 0) -- (5, \yi, 0) -- cycle;
      \fill[darkblue, opacity=\opacity]
        (1, \yi, 0) -- (1, \yUpper, 0) -- (2, \yUpper, 0) -- (2, \yi, 0) -- cycle;
    }
  
  % ---- RED MODULE (z = 1) ----
  \fill[lightred, opacity=1]
    (2,2,1) -- (2,4,1) -- (6,4,1) -- (6,2,1) -- cycle;

  \foreach \yi/\opacity in {4/0.7, 4.3/0.5, 4.6/0.3, 5/0} {
  \pgfmathsetmacro{\yUpper}{ifthenelse(\yi+0.3 > 5, 5, \yi+0.3)}
  \fill[lightred, opacity=\opacity]
    (2, \yi, 1) -- (2, \yUpper, 1) -- (6, \yUpper, 1) -- (6, \yi, 1) -- cycle;
    }
    
% Vertical homomorphism arrows with bold transparent double stroke
\foreach \x/\y in {2/2, 6/2} {
  % Two parallel lines for the shaft (no arrowheads)
  \draw[black, line width=1.5pt]
    (\x-0.05,\y,0.9) -- (\x-0.05,\y,0.15);
  \draw[black, line width=1.5pt]
    (\x+0.05,\y,0.9) -- (\x+0.05,\y,0.15);
  % Single arrowhead (lower)
  \draw[->, black, line width=1.5pt]
    (\x,\y,0.15) -- (\x,\y,0.05);
}

  % --- GENERATORS & RELATIONS for blue module (z = 0) ---
  \foreach \x/\y in {0/1, 1/0} {
    \fill[blue] (\x,\y,0) circle (2.5pt);
  }
  \foreach \x/\y in {1.9 / 1.9, 5/0, 5/1} {
    \node[rectangle, draw=black, fill=blue, inner sep=1.5pt] at (\x,\y,0) {};
  }

  % --- GENERATORS & RELATIONS for red module (z = 1) ---
  \foreach \x/\y in {2/2} {
    \fill[red] (\x,\y,1) circle (2.5pt);
  }
  \foreach \x/\y in {6/2} {
    \node[rectangle, draw=black, fill=red, inner sep=1.5pt] at (\x,\y,1) {};
  }

    % Axes (x and y only)
  \draw[->, thick] (-0.1,-0.1,0) -- (5.4,-0.1,0) node[below] {$x$};
  \draw[->, thick] (-0.1,-0.1,0) -- (-0.1,5.4,0) node[above] {$y$};
  
\end{tikzpicture}
    \end{minipage}
    \begin{minipage}[c]{0.45\textwidth}
        \[
        \begin{array}{| c |  c |}
            \hline 
             M & (6,2) \\
            \hline
            \lrcell (2,2) &  1 \\
         \hline
        \end{array} \] 

        \[
        \begin{array}{| c | c c c |}
            \hline 
             N &  (2,2) & (5, 0) & (5,1) \\
            \hline
            \lbcell (0,1) &  1 & \gcell 0 & 1\\
            \lbcell (1,0) &  -1 & 1 & 0\\ 
         \hline
        \end{array}
        \]
    \end{minipage}
    \caption{\autoref{fig:hom1} endowed with presentations. The gray cell indicates an entry which has to be zero.}
    \label{fig:hom1_repeat}
\end{figure}
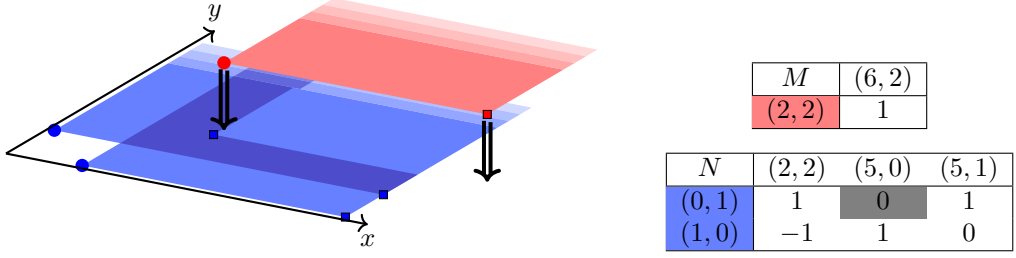

The full linear system to compute $\Hom(M, \ N)$ is: 
\begin{align}
\text{Find } Q = \begin{bmatrix} 
    q_{1,1} \\
    q_{2,1} \\
\end{bmatrix} \in \K^{\left( (2,2) \right) \times \left( (0,1), \, (1,0) \right)} \quad 
& \text{ and } P = \begin{bmatrix}
    p_{1,1} \\
    p_{2,1} \\
    p_{3,1}\\ 
\end{bmatrix} \in \K^{\left( (6,2) \right) \times \left( (2,2), \, (5,0), \, (5,1) \right)} \text{ such that } \notag
\\
\label{eq:hom_system}
\begin{bmatrix}
    q_{1,1} \\
    q_{2,1}
\end{bmatrix}  = \begin{bmatrix}
    q_{1,1} \\
    q_{2,1}
\end{bmatrix} 
\begin{bmatrix}
    1 \\
\end{bmatrix} & =
\begin{bmatrix}
    1 & 0 & 1\\
    -1 & 1 & 0\\
\end{bmatrix}
\begin{bmatrix}
    p_{1,1} \\
    p_{2,1} \\
    p_{3,1}\\ 
\end{bmatrix} = 
\begin{bmatrix}
    p_{1,1} + p_{3,1} \\
    - p_{1,1} + p_{2,1} \\
\end{bmatrix}
\end{align}

A basis of the solution space is given by the three pairs
\[ \left( \begin{bmatrix}
    1 \\
    -1
\end{bmatrix}, \begin{bmatrix}
    1 \\
    0 \\
    0\\ 
\end{bmatrix}\right)  , \ \left(
 \begin{bmatrix}
    0 \\
    1
\end{bmatrix}, 
\begin{bmatrix}
    0 \\
    1\\
    0\\ 
\end{bmatrix}\right) , \ \left(
 \begin{bmatrix}
    1 \\
    0
\end{bmatrix}, \begin{bmatrix}
    0 \\
    0 \\
    1\\ 
\end{bmatrix} \right) \]
A homotopy is given by a map from the generators of $X$ to the relations of $Y$. The only compatible pair of degrees here is $\{G \leq R'\} = \left\{ (2,2)  \leq \left( (2,2), \, (5,0), \, (5,1) \right) \right\} = \left\{ \left[ (2,2), \, (2,2) \right] \right\}$, so the iteration in \autoref{line:homotopy_components} is carried out once for this pair.

We have $\{G \leq G'\} = \left\{ \left[ (2,2), \, (0,1)  \right], \left[ (2,2), \, (1,0)  \right] \right\}$ and so $d_Q = 2$ and $N^{ (2,2), \, (2,2) }$ is a column of length $2$. Since in both of these pairs the chosen generator in $G$ is the same as the one in $\{G \leq R'\}$, we have
\[ N^{ (2,2), \, (2,2) } = N_{\bullet, (2,2)} =  \begin{bmatrix}
    1 \\
    -1
\end{bmatrix}\]
This column in the matrix $\bar N$ witnesses that there is a potential homotopy $H$ for which $NH$ is exactly this column. We therefore delete all null-homotopic morphisms of presentations by reducing $\left[\bar N \left(Q_j\right)_{j \in J} \right] $:
\[ \left[\bar N \left(Q_j\right)_{j \in J} \right]  = 
\begin{bmatrix}
1 & 1 & 0 & 1 \\
-1 & -1 & 1 & 0 \\
\end{bmatrix} \overset{\text{Column Reduction}}{\Longrightarrow}
\begin{bmatrix}
1 & 0 & 1 & 0 \\
-1 & 0 & 0 & 0 \\
\end{bmatrix} \]
We find that $J' = \{1\}$ and $(Q_j)_{j \in J'} = \begin{bmatrix}
1  \\
 0 \\
\end{bmatrix} $, which indeed induces a basis for $\Hom(X, \, Y) \simeq \K$.

\end{example}

\begin{proposition}
\autoref{alg:direct} returns a basis of $\Hom(X, \, Y)$.
\end{proposition}

\begin{proof}
By \autoref{cor:nullhomotopic} we only need to show that there are no $ (\lambda_j)_{j \in J'}$ and no $H \colon A[G] \to A[R']$ and such that the output $(Q_j)_{j \in J'}$ satisfies \autoref{eq:nullhomotopic}. So assume that these did exist. Let $j_0$ be the largest index, such that $\lambda_{j_0} \neq 0$. Then
\[ Q_{j_0} = \frac{1}{\lambda_{j_0}} \sum\limits_{j \leq j_0} - \lambda_j Q_j + \frac{1}{\lambda_{j_0}} \sum\limits_{ \deg(g) \leq \deg(r')}
H_{r', g}\left[ 0 \dots 0 \ \underset{\text{column } g}{N_{\bullet, r'}} \ 0 \dots 0 \right].
\]
Observe that the matrices to the right correspond exactly to the vectors $\bar N_{\bullet, (r', g)}$ and so the reduction in \autoref{homotopy_reduction} would have cleared the column at $j_0$.
\end{proof}

\begin{proposition}
When using Gauss-Elimination to solve \autoref{eq:naive}, \autoref{alg:direct} runs at most in time $\Oc\left( (mn)^2(m^2+n^2) \right)$ and uses $\Oc \left((m+n)^4\right)$ memory.
\end{proposition}

\begin{proof}
Since the linear system \autoref{eq:naive} contains $m^2+n^2$ many variables, we know that $|J| \leq m^2 + n^2$. The matrix $\bar N$ has at most $mn$ columns and it follows that the system $[\bar N \ (Q_j)_{j \in J}]$ has at most $mn$ rows and $mn + m^2 + n^2$ columns.
\end{proof}

\subsection{Restriction of Chain Maps and Algorithm A}

\subparagraph{Restricting the Variables in $Q$.}
Solving the linear system \autoref{eq:naive} should be interpreted as follows: For each generator $g \in G$ of $X$ try to map it, via the matrix $Q$, to a linear combination of elements in $A[G']$ which span 
$Y_{\deg(g)}$.

\begin{observation}\label{obs:omega}
The set $\set{ x^{\deg(g)-\deg(g')}\cdot g' | \ g' \in G', \ \deg(g') \leq \deg(g) } \subset A[G']_{\deg(g)}$ generates $Y_{\deg(g)}$, so $G'$ has a subset of size $\dim_\K Y_{\deg(g)}$ whose images at $\deg(g)$ form a basis of $Y_{\deg(g)}$.
\end{observation}
Let $G'_g$ be this subset. Then, 
the vector space $\left\langle \set{ x^{\deg(g)-\deg(g')}\cdot g' | \ g' \in G_g' } \right \rangle$ is a $\{g\}$-restriction system for $A[G'] \onto Y$ (\autoref{def:restriction_system}).

\begin{example}\label{ex:running2}
Consider again the setting of \autoref{ex:running1}. The single generator of the red module $X$ sits at degree $(2,2)$, where the blue module $Y$ has dimension $1$. Although both of the generators of $Y$ have degrees smaller than $(2,2)$, any one of them would already be enough to generate $Y_{(2,2)}$. Therefore we could restrict our search for homomorphisms of presentations to those which map the red generator only to, say, the blue generator at $(0,1)$. 

\begin{figure}[H]
    \centering
    \tdplotsetmaincoords{70}{30}
\begin{tikzpicture}[tdplot_main_coords, scale=1]

  % ---- BLUE MODULE (z = 0) ----
  % base blue shape
  \fill[lightblue, opacity=1]
    (0,1,0) -- (0,4,0) -- (5,4,0) -- (5,0,0) -- (1,0,0) -- (1,1,0) -- cycle;
  \fill[darkblue, opacity=1]
    (1,1,0) -- (1,4,0) -- (1.9,4,0) -- (1.9,1.9,0) -- (5,1.9,0) -- (5,1,0) -- cycle;

    \foreach \yi/\opacity in {4/0.7, 4.3/0.5, 4.6/0.3, 5/0} {
      \pgfmathsetmacro{\yUpper}{ifthenelse(\yi+0.3 > 5, 5, \yi+0.3)}
      \fill[lightblue, opacity=\opacity]
        (0, \yi, 0) -- (0, \yUpper, 0) -- (5, \yUpper, 0) -- (5, \yi, 0) -- cycle;
      \fill[darkblue, opacity=\opacity]
        (1, \yi, 0) -- (1, \yUpper, 0) -- (2, \yUpper, 0) -- (2, \yi, 0) -- cycle;
    }
  
  % ---- RED MODULE (z = 1) ----
  \fill[lightred, opacity=1]
    (2,2,1) -- (2,4,1) -- (6,4,1) -- (6,2,1) -- cycle;

  \foreach \yi/\opacity in {4/0.7, 4.3/0.5, 4.6/0.3, 5/0} {
  \pgfmathsetmacro{\yUpper}{ifthenelse(\yi+0.3 > 5, 5, \yi+0.3)}
  \fill[lightred, opacity=\opacity]
    (2, \yi, 1) -- (2, \yUpper, 1) -- (6, \yUpper, 1) -- (6, \yi, 1) -- cycle;
    }
    
% Vertical homomorphism arrows with bold transparent double stroke
\foreach \x/\y in {2/2, 6/2} {
  % Two parallel lines for the shaft (no arrowheads)
  \draw[black, line width=1.5pt]
    (\x-0.05,\y,0.9) -- (\x-0.05,\y,0.15);
  \draw[black, line width=1.5pt]
    (\x+0.05,\y,0.9) -- (\x+0.05,\y,0.15);
  % Single arrowhead (lower)
  \draw[->, black, line width=1.5pt]
    (\x,\y,0.15) -- (\x,\y,0.05);
}

  % --- GENERATORS & RELATIONS for blue module (z = 0) ---
  \foreach \x/\y in {0/1, 1/0} {
    \fill[blue] (\x,\y,0) circle (2.5pt);
  }
  \foreach \x/\y in {1.9 / 1.9, 5/0, 5/1} {
    \node[rectangle, draw=black, fill=blue, inner sep=1.5pt] at (\x,\y,0) {};
  }

  % --- GENERATORS & RELATIONS for red module (z = 1) ---
  \foreach \x/\y in {2/2} {
    \fill[red] (\x,\y,1) circle (2.5pt);
  }
  \foreach \x/\y in {6/2} {
    \node[rectangle, draw=black, fill=red, inner sep=1.5pt] at (\x,\y,1) {};
  }

    % Axes (x and y only)
  \draw[->, thick] (-0.1,-0.1,0) -- (5.4,-0.1,0) node[below] {$x$};
  \draw[->, thick] (-0.1,-0.1,0) -- (-0.1,5.4,0) node[above] {$y$};

    % target point
  \coordinate (target) at (2,1.6);

  % label node with arrow
  \node[draw=none,anchor=south west] (label) at (2.5, - 2.8) {$\dim_\mathds{K} Y_{(2,2)} =1$};
  \draw[->] (label) -- (target);
  
\end{tikzpicture}
    \label{fig:hom1_repeat2}
\end{figure}

That is, we force in \autoref{eq:hom_system} the variable $q_{2,1}$ to be zero and instead solve

\[  \begin{bmatrix}
    q_{1,1} \\
    0 \\
\end{bmatrix} =
\begin{bmatrix}
    p_{1,1} + p_{3,1} \\
    - p_{1,1} + p_{2,1} \\
\end{bmatrix}
\]
This would cut the number of variables in $Q$ by half. A basis of this solution space is (cf \autoref{ex:running1}) 
\[ \left( \begin{bmatrix}
    1 \\
    0
\end{bmatrix}, \begin{bmatrix}
    1 \\
    1 \\
    0\\ 
\end{bmatrix}\right)  , \ \left(
 \begin{bmatrix}
    1 \\
    0
\end{bmatrix}, \begin{bmatrix}
    0 \\
    0 \\
    1\\ 
\end{bmatrix} \right) \]
After reduction as in \autoref{homotopy_reduction} we would arrive at the same matrix as in \autoref{ex:running1}, although in general, we might end up with a different basis $(Q_j)_{j \in J'}$ for $\Hom(X, \, Y)$.
\end{example}

\ignore{
\subparagraph{Restricting the Variables $P$.}

\begin{observation}
A map $Q\colon F_0 \to F'_0$ descends to a homomorphism $X \to Y$ iff it sends $\Img M$ to $\Img N$. Equivalently, in the algorithm we try to find for each relation $r \in R$ an element $P(r) = \sum_{\deg(r') \leq \deg(r)} x^{\deg(r)-\deg(r')} \{r'\} \in \bigoplus_{r' \in R'} A[-\deg(r')]$ such that $QMr = N P(r) \in \left( \Img N \right)_{\deg(r)}$. Again, we could restrict ourselves to any subset of $R'$ whose image at degree $\deg(r)$ spans $\left( \Img N \right)_{\deg(r)}$. Therefore, we can use the strategy of \autoref{ex:running2} for the module $\Img N$ to reduce the number of variable entries in $P$ in \autoref{eq:naive}.
\end{observation}

\begin{definition}
Let $\dots\oto{d_2} F_1 \oto{d_1} F_0 \oto{d_0} X \to 0$ be a minimal free resolution. For $i \in \N$ the $i$-th \emph{Syzygy}
-module of $X$ is
\[ \Omega^{i}(X) \coloneqq \Img(d_i) = \ker(d_{i-1}).  \]
In particular, there are surjections $\bar d_i \colon F_i \onto \Omega^i(X)$.
\end{definition}

\begin{remark}
    The modules $\Omega^{i}(X)$ are independent of the chosen resolution, since all minimal resolutions of the same module are isomorphic.
\end{remark}

\begin{example}\label{ex:running3}
We continue \autoref{ex:running2}. At $(6,2)$ (red dot in \autoref{fig:omega1_example}) - which is the degree of the single relation of the red module $X$ - the second module in the resolution of $Y$ ("the relations")  has dimension $3$. This number corresponds to the $3$ variable entries of $P$. The module $\Omega^1(Y)$, on the other hand, has dimension $2$ at this degree.
\begin{figure}[H]
\centering
\label{fig:omega1_example}
\tdplotsetmaincoords{70}{30}
\begin{tikzpicture}[tdplot_main_coords, scale=0.6]
% Define horizontal shift between copies
  \def\xshift{8} % adjust spacing as needed
    \def\yshift{4.6}
  % First position (originally third copy)
  \begin{scope}[shift={(0,0,0)}]
    \fill[lightblue, opacity=1]
      (0,1,0) -- (0,4,0) -- (5,4,0) -- (5,0,0) -- (1,0,0) -- (1,1,0) -- cycle;
    \fill[darkblue, opacity=1]
      (1,1,0) -- (1,4,0) -- (1.9,4,0) -- (1.9,1.9,0) -- (5,1.9,0) -- (5,1,0) -- cycle;
    \foreach \yi/\opacity in {4/0.7, 4.3/0.5, 4.6/0.3, 5/0} {
      \pgfmathsetmacro{\yUpper}{ifthenelse(\yi+0.3 > 5, 5, \yi+0.3)}
      \fill[lightblue, opacity=\opacity]
        (0, \yi, 0) -- (0, \yUpper, 0) -- (5, \yUpper, 0) -- (5, \yi, 0) -- cycle;
      \fill[darkblue, opacity=\opacity]
        (1, \yi, 0) -- (1, \yUpper, 0) -- (2, \yUpper, 0) -- (2, \yi, 0) -- cycle;
    }
    \foreach \x/\y in {0/1, 1/0} {
      \fill[blue] (\x,\y,0) circle (2.5pt);
    }
    \foreach \x/\y in {1.9 / 1.9, 5/0, 5/1} {
      \node[rectangle, draw=black, fill=blue, inner sep=1.5pt] at (\x,\y,0) {};
    }
    % Axes labels
    \draw[->, thick] (-0.1,-0.1,0) -- (5.4,-0.1,0) node[below] {$x$};
    \draw[->, thick] (-0.1,-0.1,0) -- (-0.1,5.4,0) node[above] {$y$};
  \end{scope}
  % Second position (originally first copy)
  \begin{scope}[shift={(\xshift ,\yshift,0)}]
    % --- BLUE MODULE (z = 0) ---
    \fill[lightblue, opacity=1]
      (1.9, 1.9, 0) -- (1.9, 4 ,0) -- (5,4,0) -- (5.0, 0, 0) -- (7.0, 0, 0) -- (7.0, 1.9, 0) --  (5, 1.9, 0) --cycle;
      \fill[darkerblue, opacity=1]
      (5, 1.9, 0) -- (5,4,0) -- (7,4,0) -- (7,1.9,0) -- cycle;
      \fill[darkblue, opacity = 1 ]
      (5.0, 1.9, 0) -- (5.0, 1, 0) -- (7.0, 1, 0) -- (7.0, 1.9, 0) -- cycle;
    \foreach \yi/\opacity in {4/0.7, 4.3/0.5, 4.6/0.3, 5/0} {
      \pgfmathsetmacro{\yUpper}{ifthenelse(\yi+0.3 > 5, 5, \yi+0.3)}
      \fill[lightblue, opacity=\opacity]
        (1.9, \yi, 0) -- (1.9, \yUpper, 0) -- (7, \yUpper, 0) -- (7, \yi, 0) -- cycle;
      \fill[darkerblue, opacity=\opacity]
        (5, \yi, 0) -- (5, \yUpper, 0) -- (7, \yUpper, 0) -- (7, \yi, 0) -- cycle;  
    }
    % Generators and relations
    \foreach \x/\y in {1.9 / 1.9, 5/0, 5/1.0} {
      \fill[blue] (\x,\y,0) circle (2.5pt);
    }

        \node[rectangle, draw = black, fill = lightred, inner sep = 1.5pt] at (6,2.1,0) {};
        
    % Axes labels
    \draw[->, thick] (-0.1,-0.1,0) -- (5.4,-0.1,0) node[below] {$x$};
    \draw[->, thick] (-0.1,-0.1,0) -- (-0.1,5.4,0) node[above] {$y$};
  \end{scope}
  % Third position (originally second copy)
  \begin{scope}[shift={(2*\xshift,2*\yshift,0)}]
        % --- BLUE MODULE (z = 0) ---
    \fill[lightblue, opacity=1]
      (1.9, 1.9, 0) -- (1.9, 4 ,0) -- (5,4,0) -- (5.0, 0, 0) -- (7.0, 0, 0) -- (7.0, 1.9, 0) --  (5, 1.9, 0) --cycle;

      \fill[darkblue, opacity = 1 ]
      (5.0, 4, 0) -- (5.0, 1, 0) -- (7.0, 1, 0) -- (7.0, 4, 0) -- cycle;
    \foreach \yi/\opacity in {4/0.7, 4.3/0.5, 4.6/0.3, 5/0} {
      \pgfmathsetmacro{\yUpper}{ifthenelse(\yi+0.3 > 5, 5, \yi+0.3)}
      \fill[lightblue, opacity=\opacity]
        (1.9, \yi, 0) -- (1.9, \yUpper, 0) -- (7, \yUpper, 0) -- (7, \yi, 0) -- cycle;
      \fill[darkblue, opacity=\opacity]
        (5, \yi, 0) -- (5, \yUpper, 0) -- (7, \yUpper, 0) -- (7, \yi, 0) -- cycle;  
    }
    % Generators and relations
    \foreach \x/\y in {1.9 / 1.9, 5/0, 5/1.0} {
        \fill[blue] (\x,\y,0) circle (2.5pt);
      
    }
    \foreach \x/\y in {5/1.9} {
        \node[rectangle, draw=black, fill=blue, inner sep=1.5pt] at (\x,\y,0) {};
    }

    \node[rectangle, draw = black, fill = lightred, inner sep = 1.5pt] at (6,2.1,0) {};
    
    % Axes labels
    \draw[->, thick] (-0.1,-0.1,0) -- (5.4,-0.1,0) node[below] {$x$};
    \draw[->, thick] (-0.1,-0.1,0) -- (-0.1,5.4,0) node[above] {$y$};
  \end{scope}

  % Arrows between copies  

% From copy 2 to 3 - lightblue arrow  
\draw[thick, lightblue]  
(2*\xshift - 1, 2*\yshift - 0.6, -1 + 0.03) -- (2*\xshift + 1, 2*\yshift + 0.6, -1 + 0.03);  
\draw[thick, lightblue]  
(2*\xshift - 1, 2*\yshift - 0.6, -1 - 0.03) -- (2*\xshift + 1, 2*\yshift + 0.6, -1 - 0.03);  
% Single arrowhead  
\draw[->, thick, lightblue]  
(2*\xshift + 1.02, 2*\yshift + 0.612, -1) -- (2*\xshift + 1.05, 2*\yshift + 0.63, -1);  
\end{tikzpicture}
\caption{From left to right: \quad $Y$, \quad \quad \quad $P_1 $ \quad $\Rightarrow$ \quad $\Omega^1(Y)$}
\end{figure}

A set of generators that span $\Omega^1(Y)_{(6,2)}$ is given, for example, by the relations of $Y$ at $(2,2)$ and $(5,0)$. Consequently, we could also force the variable $p_{3,1}$ in $P$ to be zero. This will again reduce the number of variables of the linear system \autoref{eq:hom_system}. Its solution is now spanned by the single vector
\[ \left( \begin{bmatrix}
    1 \\
    0
\end{bmatrix}, \begin{bmatrix}
    1 \\
    1 \\
    0\\ 
\end{bmatrix}\right). \]
This morphism of presentations already induces a basis for $\Hom(X, \, Y)$, so the reduction in \autoref{homotopy_reduction} would not be necessary. We will show that this is always the case in \autoref{prop:nullhomotopic}.
\end{example}

\subparagraph{WHAT WE SHOULD DO}

}

\subparagraph{Computing Restriction Systems.}
\autoref{obs:omega} tells us exactly how to find
the subsets $(G')^{b}_i$. For each $b \in B_i$, let $N_{\leq \deg(b)}$ be the submatrix of $N$ containing all generators and relations of degree $\leq b$. If we forget the grading, then it presents the \emph{vector space} $Y_{\deg(b)}$ and so a basis for this vector space can be found by computing the kernel of $N_{\leq \deg(b)}^t$. We have written pseudocode for \autoref{alg:cokernel} in the appendix \autoref{sec:algs}. Its input is $N$ and a degree $\alpha = \deg(b)$ and it returns both the set $(G')^{\alpha}_i$ and the cokernel of $N_{\leq \alpha}$ as a matrix $d_\alpha \colon \K[G'_{\leq \alpha}] \to Y_{\deg(b)}$.

\begin{proposition}
\autoref{alg:cokernel} returns the correct result.
\end{proposition}

\begin{proof}
After the computation in \autoref{solve_transpose} it holds that $d_\alpha \cdot N_{\leq \alpha}= 0$ and so each column of $N_{\leq \alpha}$ defines a linear combination of columns of $d_\alpha$ which is zero. But $d^t_\alpha$ also spans $\ker M^t_{\leq \alpha}$ and so there is no column-vector outside of the column-span of $M_{\leq \alpha}$ with that property. Therefore, if there is a linear combination of the columns of $\left(d_\alpha\right)_{\bullet, G_\alpha}$ which evaluates to zero, it must be a linear combination of the columns of $N_{\leq \alpha}$. Each of these will be non-zero for some pivot-element which is, per construction, not in $G_\alpha$.
\end{proof}

\subparagraph{Proving \autoref{thm:main}}

We first show formally that we can always lift (uniquely) into a (sharp) restriction system.

\begin{lemma}\label{lem:lifting}
Let $(V^g)_{g \in G}$ be a $G$-restriction system for $p \colon Y \to X$. \\
\begin{minipage}{0.69\textwidth}
Any map $f\colon A[G] \to X$ has a lift $\widetilde f$ along $p$ such that $f(g) \in V^g$ and if the system is sharp, then this lift is unique.
\end{minipage}
\begin{minipage}{0.29\textwidth}
% https://q.uiver.app/#q=WzAsMyxbMCwxLCJBIFtHXSJdLFsxLDAsIlkiXSxbMSwxLCJYIl0sWzEsMiwicCIsMCx7InN0eWxlIjp7ImhlYWQiOnsibmFtZSI6ImVwaSJ9fX1dLFswLDIsImYiXSxbMCwxLCJcXHdpZGV0aWxkZSBmIiwwLHsiY3VydmUiOi0xLCJzdHlsZSI6eyJib2R5Ijp7Im5hbWUiOiJkb3R0ZWQifX19XV0=
\[\begin{tikzcd}
	& Y \\
	{A [G]} & X
	\arrow["p", two heads, from=1-2, to=2-2]
	\arrow["{\widetilde f}", curve={height=-6pt}, dotted, from=2-1, to=1-2]
	\arrow["f", from=2-1, to=2-2]
\end{tikzcd}\]
\end{minipage}
\end{lemma}

\begin{proof}
A map from $A[G]$ to any module is just a $G$-indexed tuple of homogenous elements of degree $\deg(g)$. If $(p_{\deg(g)})_{|V^g}$ is surjective, then we can choose the preimages to be in its domain. If its an isomorphism, then this preimage has to be unique.
\end{proof}

\subparagraph{Thickness and Sparsity.}
Before constructing \autoref{alg:restricted}, we use \autoref{lem:lifting} to show that every module of low thickness has a sparse minimal presentation.
\begin{proof}[Proof of \autoref{prop:thickness_sparsity}]
By induction, we can assume that $X$ is presented by a matrix $[M \ N] \in \K^{G \times ( R \ \cup \ R')}$ with $M$ already satisfying the condition and $R' = (\omega, \dots, \omega)$ for an $\omega \in \Z^d$ with $r < \omega$ for each $r \in R$. 
These matrices induce a short exact sequence
$ 0 \to A^k[-\omega] \oto{f} \coker M \to X \to 0$.
It follows that $\dim_\K (\coker M)_\omega$ = $\dim_\K X_\omega + k$.

By \autoref{obs:omega} we can find a subset $G_\omega \subset G$ of size $\dim_\K X_\omega + k$ that spans a sharp $(\bigcup_{j = 1}^k \omega)$-restriction system for the map $A[G] \onto \coker M$. By \autoref{lem:lifting}, $f$ has a lift $\widetilde N \colon A^k[-\omega] \to A[G]$ which factors through $A[{G_\omega}]$. $\widetilde N$ is a graded matrix of size $(\dim_\K X_\omega + k)\times k$, with all columns of the same degree. Bringing it into reduced column echelon form by changing the basis of $A^k[-\omega]$, we construct a graded matrix $\bar N \colon A^k[-\omega] \to A[{G_\omega}] \into A[G]$ where every column has at most $\dim_\K X_\omega+1$ non-zero entries. By
\cite[Proposition 4.14]{djk_arXiv}, $[M \ \bar N]$ is a presentation for $X$.
\end{proof}

This proof also implies a way to compute this sparse matrix, but it would require using \autoref{alg:cokernel} for every set of relations of the same degree, which puts the runtime in $\Oc(n^{4})$. For $d=2$, the minimisation procedure described in \cite{lw-computing, mpfree} brings a presentation in this form using column-reduction in $\Oc(n^3)$ time.

\begin{tcolorbox}[thmbox]
    \mainthm*
\end{tcolorbox}

\begin{proof}
We can assume that the resolution $(A[C_i], e_i)$ is minimal by choosing appropriate subsets of $C_i$.
Recall that each $e_i$ induces a map onto its image $\bar e_i \colon A[C_i] \to \Omega^{i}(Y)$.

With \autoref{obs:omega}, we find the subsets $C_i^b$ of size $\dim_\K (\Img \bar e_i)_b = \dim_\K \Omega^{i}(Y)_{\deg(b)}$ for each $b \in B_i$. By construction, $(\bar e_i)_{\deg(b}$ restricts to an isomorphism on the sub vector space \\
$\K \{ x^{\deg(b)-\deg(c)}c \ | \ c \in C^b_i \} \subset A[C^b_i]_{\deg(b)}$. These are sharp restriction systems for each $\bar e_i$, so by \autoref{lem:lifting} we can lift any $f\colon X \to Y$ uniquely into these.
\end{proof}

\subparagraph{Modules with low $\thick{\Omega^1(X)}$.}
\begin{example}
Consider the following modules $X$ (red) and $Y$ (blue). At the degree $(4,4)$ where the single relation of $Y$ lives, the free module of minimal relations of $Y$ has dimension $6$. Therefore, in the system $QM = NP$, the matrix $P$ could contain $6$ non-zero entries which become variables in this system. The module $\Omega^1(Y)$ is, however, only of dimension $2$ at $(4,4)$. 
\begin{figure}[H]
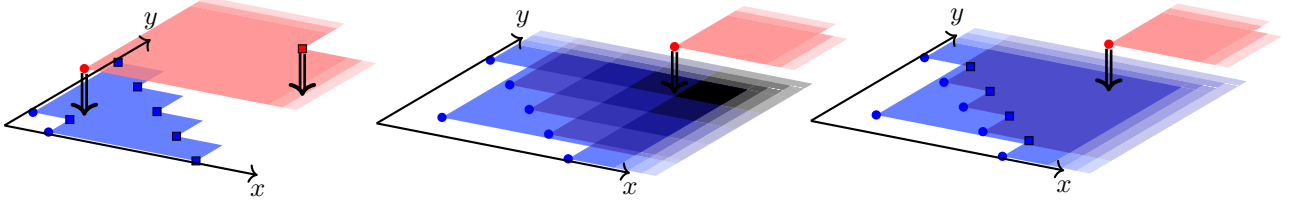

\centering
\label{fig:omega1}
\include{include/fig_omega1}
\caption{From left to right: \quad  $X$, $Y$ \quad $X_1$, $Y_1$ \quad $\Omega^1(X)$, $\Omega^1(Y)$}
\end{figure}
\end{example}

\subparagraph{The Restricted Computation of Homomorphisms of Presentations} \hspace{1em} \newline

\begin{algorithm}[H]
\renewcommand{\thealgocf}{A} % override label
\caption{Restricted Computation of Homomorphisms of Presentations.}
\label{alg:restricted}
\DontPrintSemicolon
\KwIn{Presentation matrices $M \in \K^{G \times R}$, $N \in \K^{G' \times R'}$}
\KwOut{Graded matrices $\{Q_i\}_{i \in I} \in \K^{G' \times G}$ which descend to a basis of $\Hom(X, \, Y)$}

\For{$\alpha \in \deg(G)$}{
    $d_\alpha, \, G'_\alpha   \gets$ \autoref{alg:cokernel}$(N, \alpha)$\label{line:cokernel}\;
    \For{$g \in G$ with $\deg(g) = \alpha$}{
    $Q_{g', g} \gets \begin{cases}
        \text{a variable } q_{g',g} &\text{ if } g' \in G'_\alpha \\
        0 &\text{ otherwise}
    \end{cases}$
    }
} 
$O \in \K^{R' \times S} \gets \ker N$\label{eq:kernel} \;
\For{$\alpha \in \deg(R)$}{
    $e_\alpha, \, R'_\alpha  \gets$ \autoref{alg:cokernel}$(O, \alpha)$\label{line:cokernel2}\;
    \For{$r \in R$ with $\deg(r) = \alpha$}{
    $P_{r', r} \gets \begin{cases}
        \text{a variable } p_{r',r} &\text{ if } r' \in R'_\alpha \\
        0 &\text{ otherwise}
    \end{cases}$
    }
}
Calculate a basis $(Q_j, \, P_j)_{j \in J} \in \K^{G' \times G}, \K^{R' \times R}$ of the solution space of the system $Q M = N P$ \label{line:restricted2}\;
Treat the tuple $(Q_j)_{j \in J}$ as a matrix in $\K^{ d_Q \times J}$ and column-reduce it\label{line:reduction}\;
$J' \gets \{j \in J \, | \ Q_j \neq 0 \}$\;
\KwRet{$\left(Q_j\right)_{j \in J'}$}
\end{algorithm}

\begin{proposition}\label{prop:A_runtime}
\autoref{alg:restricted} runs in
$\Oc( m^3 n \left( \thick{Y} + \thick{  \Omega^{1}(Y) } \right)^2 + mn^3 + T_{\ker}(d, n))$ time where $T_{\ker}$ is the time to compute kernels of $d$-parameter graded matrices.
\end{proposition}

\begin{proof}
The computations in \autoref{line:cokernel} and \autoref{line:cokernel2} reduce an $\Oc(n) \times \Oc(n)$-matrix at most $m$ times and \autoref{eq:kernel} adds $T_{\ker}(n)$ to the runtime.

The sets $\{G'_\alpha\}$ and $\{R'_\alpha\}$ contain exactly $\dim_\K Y_\alpha$ and $\dim_\K \Omega^{1}Y_\alpha$ elements and span the first two parts of a sharp restriction system. The linear system in \autoref{line:restricted2} contains at most $m\left( \thick{Y} + \thick{ \Omega^{1}(Y)} \right)$ variables and $mn$ equations. 
\end{proof}

\begin{proof}[Proof of \autoref{thm:alg_a}]
By \autoref{thm:main}, every homomorphism has a lift to the restriction systems calculated with \autoref{alg:cokernel} and since the lift is unique the solution space to the system $QM = NP$ descends to a basis of $\Hom(X,Y)$. For the runtime, set $m=n$ in \autoref{prop:A_runtime}.
\end{proof}

\subparagraph{Restricting only $Q$}
Since the computation of the kernel in \autoref{eq:kernel} may be costly in practice - especially for higher-dimensional modules - and $\max \dim \Omega^1(Y)$ may be large - even for pointwise low-dimensional modules - we propose another algorithm that mixes the two approaches.
It is this algorithm which we have already implemented and used in \texttt{AIDA}, \cite[3.2 "Fast Hom-Computation" and 8, Table 1]{djk}. Pseudocode can be found in the appendix \autoref{sec:algs}.

\ignore{

\section{Algorithm A, issac version}

We review some basic homological algebra to explain how to compute a \emph{basis} for $\Hom(X,Y)$ from \autoref{eq:naive}.

\begin{lemma}[Fundamental Lemma of Homological Algebra]\label{lem:fundamental}
\cite[III. Theorem 6.1]{maclane} 
Let $f \colon X \to Y$ be any homomorphism of persistence modules and $(F_*, d_*) \to X, \ (F'_*, e_*) \to Y$ free resolutions.
Then there exists a chain map $\{f_i\}_{i \in \N}$ which lifts $f$:

% https://q.uiver.app/#q=WzAsMTAsWzMsMCwiWCJdLFszLDEsIlkiXSxbMiwwLCJYXzAiXSxbMiwxLCJZXzAiXSxbMSwwLCJYXzEiXSxbMSwxLCJZXzEiXSxbMCwwLCJcXGRvdHMiXSxbMCwxLCJcXGRvdHMiXSxbNCwwLCIwIl0sWzQsMSwiMCJdLFs2LDQsImRfMiJdLFs3LDUsImVfMiJdLFs0LDIsImRfMSJdLFs1LDMsImVfMSJdLFsyLDAsImRfMCJdLFszLDEsImVfMCJdLFswLDEsImYiXSxbMiwzLCJmXzAiXSxbNCw1LCJmXzEiXSxbMCw4XSxbMSw5XSxbNiw3LCJcXGRvdHMiLDEseyJzdHlsZSI6eyJib2R5Ijp7Im5hbWUiOiJub25lIn0sImhlYWQiOnsibmFtZSI6Im5vbmUifX19XV0=
\[\begin{tikzcd}[ampersand replacement=\&]
	\dots \& {F_1} \& {F_0} \& X \& 0 \\
	\dots \& {F'_1} \& {F'_0} \& Y \& 0
	\arrow["{d_2}", from=1-1, to=1-2]
	\arrow["\dots"{description}, draw=none, from=1-1, to=2-1]
	\arrow["{d_1}", from=1-2, to=1-3]
	\arrow["{f_1}", from=1-2, to=2-2]
	\arrow["{d_0}", from=1-3, to=1-4]
	\arrow["{f_0}", from=1-3, to=2-3]
	\arrow[from=1-4, to=1-5]
	\arrow["f", from=1-4, to=2-4]
	\arrow["{e_2}", from=2-1, to=2-2]
	\arrow["{e_1}", from=2-2, to=2-3]
	\arrow["{e_0}", from=2-3, to=2-4]
	\arrow[from=2-4, to=2-5]
\end{tikzcd}\]

If $f_i$ and $\widetilde f_i$ are chain maps which both lift $f$ then they are homotopic. That is there are maps $h_i \colon F_i \to F'_{i+1}$ satisfying
\[ f_i - \widetilde f_i = e_{i+1} \circ h_i + h_{i-1} \circ d_i.\]
\end{lemma}

The first part of \autoref{lem:fundamental} guarantees the existence of the set of matrices $\{Q_j\}_{j \in J}\colon \K^{G' \times G}$ which we require.
If we want it to form a basis, we can use the second part of \autoref{lem:fundamental}.

\begin{corollary}\label{cor:nullhomotopic}
    A tuple of graded matrices $\{Q_j\}_{j \in J}$ in $\K^{G' \times G}$ descends to a linearly dependent set of homomorphisms $X \to Y$ if and only if there are $\lambda_j \in \K$, which are not all zero, and a graded matrix $H \colon A[G] \to A[R']$ such that
\begin{align}\label{eq:nullhomotopic}
\sum\limits_{j \in J} \lambda_j Q_j = N H.
\end{align}
\end{corollary}

\begin{proof}
The induced set of homomorphisms $\{\widetilde Q_i\}$ is linearly dependent iff a linear combination of any choice of lifts induces the zero map. This is equivalent to the existence of the homotopy $H$.  
\end{proof}

\begin{algorithm}
\caption{Computation from $\Hom(M,N)$.}
\label{alg:direct}
\DontPrintSemicolon
\KwIn{Graded matrices $M \in \K^{G \times R}$, $N \in \K^{G' \times R'}$}
\KwOut{Graded matrices $\{Q_i\}_{i \in I} \in \K^{G' \times G}$ which descend to a basis of $\Hom(X, \, Y)$}
\label{gauss} Calculate a basis $(Q_j, \, P_j)_{j \in J}$ for the solution of  \autoref{eq:naive} \;
Impose any order on $\set{G \leq G' }$ and set $d_Q \gets |\set{G \leq G' }|$\;
Treat the tuple $(Q_j)_{j \in J}$ as a matrix in $\K^{ d_Q \times J}$ \;
Initialise a matrix $\bar N \in \K^{d_Q \times \set{R' \leq G}}$ \;
\For{$(r', g) \in \set{R' \leq G}$}{ \label{line:homotopy_components}
    Initialise $N^{r',g} \in \K^{d_Q}$\;
    \For{$(\bar g, g') \in \set{G \leq G'}$}{
        \If{$\bar g = g$ }{
            $\left( N^{r',g}\right)_{g',\bar g} \coloneqq N_{g',r'}$\;
        } \Else {
            $\left( N^{r',g}\right)_{g',\bar g} \coloneqq 0$ \;
        }
    }
    $\bar N_{\bullet, (r',g)} \gets N^{r',g}$\;
}

Column-reduce the matrix $\left[ \bar N \ \ (Q_j) \right] \in \K^{d_Q \times (\{G \leq R'\} \, \cup \, J )}$\; \label{homotopy_reduction} 
$J' \gets \{j \in J \, | \ Q_j \neq 0 \}$\;
\KwRet{$\left(Q_j\right)_{j \in J'}$}
\end{algorithm}

\ignore{

\begin{algorithm}[H]
\renewcommand{\thealgocf}{B} % override label
\caption{Computation via Hom-Exactness ISSAC}
\label{alg:hom_exact_issac}
\DontPrintSemicolon
\KwIn{Presentation matrices $M \in \K^{G \times R}$, $N \in \K^{G' \times R'}$}
\KwOut{Graded matrices $\{Q_i\}_{i \in I} \in \K^{G' \times G}$ which descend to a basis of $\Hom(X, \, Y)$}
\For{$\alpha \in \deg(G') \cup \deg(R')$}{
    $e_\alpha, \, G'_\alpha   \gets$ \autoref{alg:cokernel}$(N, \alpha)$\label{line:cokernel3}\;
} 
Initialise $\vec Y \in \K^{ \left( \sum_R \dim Y_{\deg(r)} \right) \times  \left( \sum_G \dim Y_{\deg(g)} \right)}$ \;
\For{$(g,r) \in \{G \leq R\}$ }{
        ${\vec Y}_{r,g} \coloneqq \left(e_{\deg(r)}\right)_{|G'_{\deg(g)}}$\;
}
$M^t \otimes {\vec Y} \gets \left(M_{g,r} \cdot {\vec Y}_{r,g}\right)_{{G \leq R}} \in \K^{ \left( \sum_R \dim Y_{\deg(r)} \right) \times  \left( \sum_G \dim Y_{\deg(g)} \right)}$\;
$K \in \K^{\sum_G \dim Y_{\deg(g)} \times J } \gets \ker M^t \otimes {\vec Y}$
\label{line:kernel_exact}\;
\For{ $j \in J$}{
    Initialise $Q_j \in \K^{G' \times G}$\;
    \For{$g \in G$}{
        \For{$g' \in G'$}{
            \If{$\deg(g') \geq \deg(g)$}{
                $(Q_j)_{g', g} \gets K_{j, (g, g')}$\;
            }
            \Else{ $(Q_j)_{g', g} \gets 0$
            }
        }     
    }
}
\Return ${\left(  Q_j \right)_{j \in J} }$\;
\end{algorithm}

}

\begin{proposition}
\autoref{alg:direct} returns a basis of $\Hom(X, \, Y)$.
\end{proposition}

\begin{proof}
By \autoref{cor:nullhomotopic} we only need to show that there are no $ (\lambda_j)_{j \in J'}$ and no $H \colon A[G] \to A[R']$ such that the output $(Q_j)_{j \in J'}$ satisfies \autoref{eq:nullhomotopic}. So assume that these did exist. Let $j_0$ be the largest index, such that $\lambda_{j_0} \neq 0$. Then
\[ Q_{j_0} = \frac{1}{\lambda_{j_0}} \sum\limits_{j \leq j_0} - \lambda_j Q_j + \frac{1}{\lambda_{j_0}} \sum\limits_{ \deg(g) \leq \deg(r')}
H_{r', g}\left[ 0 \dots 0 \ \underset{\text{column } g}{N_{\bullet, r'}} \ 0 \dots 0 \right].
\]
Observe that the matrices to the right correspond exactly to the vectors $\bar N_{\bullet, (r', g)}$ and so the reduction in \autoref{homotopy_reduction} would have cleared the column at $j_0$.
\end{proof}

\begin{proposition}
With standard matrix reduction, \autoref{alg:direct} needs $\Oc\left( (mn)^2(m+n)^2 \right)$ time.
\end{proposition}

\begin{proof}

Since the linear system \autoref{eq:naive} contains at most $m^2+n^2$ many variables, we know that also $|J| \leq m^2 + n^2$. The matrix $\bar N$ has at most $mn$ columns and it follows that the system $[\bar N \ (Q_j)_{j \in J}]$ has at most $mn$ rows and $mn + m^2 + n^2$ columns.
\end{proof}

}

\ignore{
\section{Restriction of Chain Maps and Algorithm A ISSAC}

Solving the linear system \autoref{eq:naive} should be interpreted as follows: For each generator $g \in G$ of $X$, try to map it, via the matrix $Q$, to a $\K$-linear combination of elements in $\set{G'_{\leq \deg(g)}}$, the subset of generators of $Y$ of degree compatible ($\leq$) with $g$.

\begin{observation}\label{obs:omega}
Since the set $\set{ x^{\deg(g)-\deg(g')}\cdot g' | \ g' \in G'_{\leq \deg(g)} } \subset \bigoplus_{g' \in G'} A[-\deg(g')]_{\deg(g)}$ generates $Y_{\deg(g)}$, it has a subset of size $\dim_\K Y_{\deg(g)}$ which maps to a basis.
\end{observation}
Let $G_g \subset G'_{\leq \deg(g)}$ be the chosen generators from this subset. In the language of \autoref{def:restriction_system} these generators span the space $V_0^g$. We use this formulation to avoid a concrete choice of basis.

The proof of \autoref{thm:main} will be split in two steps.
First, we will prove that given a restriction system $ \{V^b_*\}_{b \in B_*}$ we can force the elements in $B_i$ to be mapped to $V^b_i$. Then, we show that if the restriction system is sharp, every null-homotopic lift into the system is $0$. This will imply that every 
linearly independent set of graded matrices $(Q_j)_{j \in J} \in \Hom\left( A[G], A[G']\right)$ which satisfies the condition of \autoref{thm:main} already forms a basis and the reduction step in \autoref{homotopy_reduction} of \autoref{alg:direct} becomes obsolete.

\begin{proof}[Proof of \autoref{thm:main}]
$f_i(b) \in V^b_i$:

Let $\widetilde f_*$ be any lift of $f$. Assume that the condition is already true for all $i < j$ for some $j \in \N$. For each $b \in B_{j}$, $e_j \circ \widetilde f_j(b) \in \Img e_j \eqqcolon \Omega^j(Y)$. Since the map $V^b_j \to \Omega^{j}(Y)$ is surjective at $\deg(b)$ we can find $f_b \in (V^b_j)_{\deg(b)}$ such that $e_j (f_b) = e_j \circ \widetilde f_j(b)$. Then $B_j$ is a basis, so the $f_b$ assemble to a map $f_j \colon F_j \to F'_j$.

It follows that $e_j \circ ( f_j - \widetilde  f_j) = 0 \ (\ast)$ and so by exactness of $F_*'$ there is a map $h_j \colon F_j \to F'_{j+1}$ such that $ f_j - \widetilde  f_j = e_{j+1} \circ h_j \ (\ast \ast)$. We define $f_{j+1} \coloneqq \widetilde f_{j+1} + h_j \circ d_{j+1}$ and $f_i \coloneqq \widetilde f_i$ for all $i > j+1$ and $i < j$. $f_*$ is a chain map:
\begin{align*} 
e_j \circ f_j \overset{(\ast)}{=}& e_j \circ \widetilde f_j = \widetilde f_{j-1} \circ d_{j} = f_{j-1} \circ d_{j} \\
e_{j+1} \circ f_{j+1} =& e_{j+1} \circ \left( \widetilde f_{j+1} + h_j \circ d_{j+1} \right) \overset{(\ast \ast)}{=} f_j \circ d_{j+1} \\
e_{j+2} \circ f_{j+2} =& \widetilde f_{j+1} \circ d_{j+2} = \left( f_{j+1} - h_j \circ d_{j+1} \right) \circ d_{j+2}  = f_{j+1} \circ d_{j+2}.
\end{align*}
$f_*$ is a lift of $f$ and satisfies the condition for all $i \leq j$.

If, in addition, $f_*$ is nullhomotopic then $f_*=0$:

Let $h_* \colon f_* \Rightarrow 0$ be a null-homotopy, i.e. $h_i\colon F_i \to F'_{i+1}$ and $\forall i \in \N \colon f_i = e_{i+1} \circ h_{i} + h_{i-1} \circ d_i$. 
Let $b \in B_i$, then by assumption $f_i(b) \in V_i^b$ and $( e_i)_{\deg(b)}$ is injective when restricted to $V_i^b$. Therefore, it suffices to prove that $\forall i \in \N, e_i \circ f_i = 0$. We prove that $f_*=0$ by induction.

For the induction start, recall that $e_0 \colon F'_0 \to Y$ denotes the first part of the free resolution. It holds that $e_0 \circ f_0 = e_0 \circ e_1 \circ h_0= 0$. We conclude that $f_0 = 0$ and $e_1 \circ h_0 = 0$, too.

Assume now that for $i \in \N$ it holds that
$f_{i-1} = 0$.
It follows that
$e_i \circ f_i = e_i \circ ( e_{i+1} \circ h_i + h_{i-1} \circ d_i) = - h_{i-2} \circ d_{i-1} \circ d_i  = 0.$

\end{proof}

\subparagraph{The Restricted Computation of Homomorphisms of Presentations}
\autoref{obs:omega} also implies a strategy to find a basis for the subspaces $V_i^g$: Restrict the presentation $N$ of $Y$ to a presentation $N_{\leq \deg(g)}$ of $Y_{\deg(g)}$. Then reduce this matrix to find the set $G'_g \subset G'$ such that $A[G'_g] \into A[G'] \oto{\coker N_{\leq \deg(g)}} Y$ is an isomorphism at $\deg(g)$. The details are found in the appendix in \autoref{alg:cokernel}, which returns the pair $\left( \coker N_{\leq \deg(g)}, G'_g \right)$.

\begin{algorithm}[ht]
\renewcommand{\thealgocf}{A} % override label
\caption{Restricted Computation of Momorphisms of Presentations.}
\label{alg:restricted}
\DontPrintSemicolon
\KwIn{Presentation matrices $M \in \K^{G \times R}$, $N \in \K^{G' \times R'}$}
\KwOut{Graded matrices $\{Q_i\}_{i \in I} \in \K^{G' \times G}$ which descend to a basis of $\Hom(X, \, Y)$}

\For{$\alpha \in \deg(G)$}{
    $d_\alpha, \, G'_\alpha   \gets$ \autoref{alg:cokernel}$(N, \alpha)$\label{line:cokernel}\;
    \For{$g \in G$ with $\deg(g) = \alpha$}{
    $Q_{g', g} \gets \begin{cases}
        \text{a variable } q_{g',g} &\text{ if } g' \in G'_\alpha \\
        0 &\text{ otherwise}
    \end{cases}$
    }
} 
$O \in \K^{R' \times S} \gets \ker N$\label{eq:kernel} \;
\For{$\alpha \in \deg(R)$}{
    $e_\alpha, \, R'_\alpha  \gets$ \autoref{alg:cokernel}$(O, \alpha)$\label{line:cokernel2}\;
    \For{$r \in R$ with $\deg(r) = \alpha$}{
    $P_{r', r} \gets \begin{cases}
        \text{a variable } p_{r',r} &\text{ if } r' \in R'_\alpha \\
        0 &\text{ otherwise}
    \end{cases}$
    }
}
Calculate a basis $(Q_j, \, P_j)_{j \in J}$ of the solution of $Q M = N P$ \label{line:restricted2}\;
\KwRet{$\left(Q_j\right)_{j \in J}$}
\end{algorithm}

\begin{proposition}\label{prop:A_runtime}
\autoref{alg:restricted} computes a basis of $\Hom(X,Y)$ in
$\Oc( m^3 n \left( \thick{Y} + \thick{  \Omega^{1}(Y) } \right)^2 + mn^3 + T_{\ker}(d))$ time where $T_{\ker}$ is the time to compute kernels of $d$-parameter graded matrices.
\end{proposition}

\begin{proof}
The computations in \autoref{line:cokernel} and \autoref{line:cokernel2} reduce an $\Oc(n) \times \Oc(n)$-matrix at most $m$ times and \autoref{eq:kernel} adds $T_{\ker}(n)$ to the runtime.

The sets $\{G'_\alpha\}$ and $\{R'_\alpha\}$ contain exactly $\dim_\K Y_\alpha$ and $\dim_\K \Omega^{1}Y_\alpha$ elements and span the first two parts of a sharp restriction system. The linear system in \autoref{line:restricted2} contains at most $m\left( \thick{Y} + \thick{ \Omega^{1}(Y)} \right)$ variables and $mn$ equations. By \autoref{thm:main} the solution space to this system descends to a basis of $\Hom(X,Y)$.
\end{proof}

\begin{corollary}\label{cor:hom_size}
    \[\dim_K \Hom(X,Y) \leq b_0(X) \thick{Y}.\]
\end{corollary}

\begin{proof}[Proof of \autoref{thm:alg_a}]
Follows from \autoref{prop:A_runtime} and that kernels can be computed in $\Oc(n^3)$ time for $2$ parameters.
\end{proof}

\subparagraph{Modules with low $\thick{\Omega^1(X)}$.}
\begin{example}
Consider the following modules $X$ (red) and $Y$ (blue). The minimal presentation of $X$ has a relation of degree $(4,4)$. The free module induced by the relations of $Y$ has dimension $6$ at $(4,4)$. In the system $QM = NP$, the matrix $P$ is a single column with $6$ variable entries. However, $\dim_\K\Omega^1(Y)_{(4,4)}=2$.

\begin{figure}[ht]
\centering
\include{include/fig_omega1}
\caption{From left to right: \quad  $X$, $Y$ \quad $X_1$, $Y_1$ \quad $\Omega^1(X)$, $\Omega^1(Y)$}
\label{fig:omega1}
\end{figure}
\end{example}

\subparagraph{Restricting only $Q$}
Since the computation of the kernel in \autoref{eq:kernel} may be costly in practice -- especially for more than $2$ parameters -- and $\thick{ \Omega^1(Y)}$ may be large even for low-thickness modules (see e.g. \autoref{fig:omega1}) we propose another algorithm that mixes the two approaches.
It is this algorithm, which has already been tested \cite[3.2 "Fast Hom-Computation" and 8, Table 1]{djk}. Since for every $Q$ there may now be more than one $P$ that solves \autoref{line:semirestricted} we need to reduce the result $(Q_j)$, but the costly reduction step in \autoref{homotopy_reduction} is unnecessary thanks to \autoref{thm:main}.

\begin{algorithm}[ht]
\renewcommand{\thealgocf}{A-$1/2$}
\caption{Mixed Computation of Homomorphisms from Presentations.}
\label{alg:mixed}
\DontPrintSemicolon
\KwIn{Presentation matrices $M \in \K^{G \times R}$, $N \in \K^{G' \times R'}$}
\KwOut{Graded matrices $\{Q_i\}_{i \in I} \in \K^{G' \times G}$ which descend to a basis of $\Hom(X, \, Y)$}

\For{$\alpha \in \deg(G)$}{
    $d_\alpha, \, G'_\alpha   \gets$ \autoref{alg:cokernel}$(N, \alpha)$\;
    \For{$g \in G$ with $\deg(g) = \alpha$}{
    $Q_{g', g} \gets \begin{cases}
        \text{a variable } q_{g',g} &\text{ if } g' \in G'_\alpha \\
        0 &\text{ otherwise}
    \end{cases}$
    }
} 
Calculate a basis $(Q_j, \, P_j)_{j \in J}$ of the solution of $Q M = N P$ \label{line:semirestricted}\;
Treat $(Q_j)_{j\in J}$ as a matrix and reduce it\;
\KwRet{$\left( Q_j\ \ | Q_j \neq 0 \right)_{j \in J}$}
\end{algorithm}

}

\section{Computing Homomorphisms using Exactness}

The second algorithm follows the approach of \cite{ghs01}.
Since $\Hom(-,Y)$ is left exact, it induces an exact sequence of persistence modules

\begin{align}\label{eq:exact} 0 \lra \Hom(X, \, Y) \oto{d_0^*} \Hom(F_0, \, Y) \oto{d_1^*} \Hom(F_1, \, Y).\end{align}

The vector space $\Hom(X,Y)$ can then be computed as the kernel of $d_1^*$. 
In \cite{ghs01}, the authors actually use a dual strategy, which we quickly review. 

\begin{definition}\label{def:matlis}
Let $I(0)$ be the injective module co-generated at $0$ (\autoref{fig:free}). \\
The \emph{Matlis duality} functor is the equivalence $\D \coloneqq \Hom_A\left(-, \, I(0) \right) \colon \grA^{\text{op}} \to \grA$, where we are using the $\grA$-valued $\Hom$-functor (\autoref{def:enrichment}). Denoting vector space duals by $(-)^*$, the Matlis dual can also be defined by
\[(\D X)_{\alpha} \coloneqq X_{-\alpha}^* \quad (\D X)_{\alpha \to \beta } \coloneqq X_{-\beta \to - \alpha}^* \quad \forall \alpha \leq \beta \in \R^d.\]
\end{definition}

\begin{remark}\label{rem:hom_duality}
Since $Y$ is finitely generated, there is an isomorphism of persistence modules 
\[
\D (X \otimes \D Y) = \Hom(X \otimes \Hom(Y, I(0)), I(0)) \iso \Hom(X, \Hom(\Hom(Y, I(0)), I(0))) \iso \Hom(X,Y).
\]
\end{remark}

 Therefore, one can just as well apply the functor $- \otimes \D Y$ to the presentation $M$ of $X$. Then by right-exactness of the tensor product, the \emph{module} $\D \Hom(X,Y)$ is isomorphic to $\coker \left( d_1 \otimes \D Y \right)$. The degree-$0$ part of this sequence is then the vector space dual of the sequence \autoref{eq:exact}.

\subparagraph{Computing $\ker d_1^*$.}
We need to describe $d_1^*$ and $d_0^*$ in the given bases $G, R$: The map $d_1$ becomes a graded matrix $M \colon A[R] \to A[G]$ and $G, R$ also induce isomorphisms
\[ \varphi \colon \Hom(A[G], \, Y )_0 \iso \bigoplus_{g \in G} Y_{\deg(g)} \ \text{ and } \  \Hom(A[R], \, Y )_0 \iso \bigoplus_{r \in R} Y_{\deg(r)}.\]

The vector space $\Hom(X, \, Y)$ is then isomorphic to the kernel of

\begin{align}\label{eq:exact_matrix} M^t \otimes Y \colon \ \bigoplus_{g \in G} Y_{\deg(g)} \oto{ \left(M^*_{r,g}\right)_{r \in R, g \in G}} \bigoplus_{r \in R} Y_{\deg(r)}.\end{align}

Afterwards, for every $(x_g)_{g \in G}$ in the kernel of $M^t \otimes Y$, we only need to find coordinates of each $x_g$ in the generators $G'$ of $Y$.

\begin{proposition}\label{prop:M_star}
$ (M^t \otimes Y)_{r,g} = M_{g,r} \cdot Y_{\deg(g) \to \deg(r)}
$ for all $(g,r) $ $\in \{G \leq R\}$,
where $Y_{\deg(g) \to \deg(r)}$ is the structure map of $Y$.
\end{proposition}

\begin{proof}

Let $g \in G, r \in R$, $y \in Y_{\deg(g)}$, and $\bar y = (0, \dots, y,  \dots, 0) \in \bigoplus_{h \in G} Y_{\deg(h)}$. The vector $\bar y$ corresponds to the map $\varphi^{-1}(\bar y) \colon A[G] \to Y$ that sends $g \in A[G]$ to $y$ and the other elements of $G$ to $0$. 
We describe the map $M^* (\varphi^{-1}(\bar y)) \colon A[R] \to Y$ at $r \in A[R]$:

\begin{align*} \varphi^{-1}(\bar y) \circ M (r) =  \varphi^{-1}(\bar y)\left( \sum_{h \in G} M_{h, r} x^{\deg(r) - \deg(h)} \cdot h \right) \\
= \sum_{h \in G} M_{h, r} x^{\deg(r) - \deg(h)} \cdot  \varphi^{-1}(\bar y)(h)
=  M_{g, r} x^{\deg(r) - \deg(g)} \cdot y  \\
= M_{g, r} \cdot Y_{\deg(g) \to \deg(r)} \left( y \right) .\end{align*}

If $\deg(r) \ngeq \deg(g)$ we set $(M^t \otimes Y)_{r,g}=0$.
\end{proof}

\subparagraph{Computing the structure maps $Y_\alpha \to Y_\beta$.}
Since, for any $\alpha \in \R^d$, \autoref{alg:cokernel} supplies us with the map $A[G']_\alpha \onto Y_\alpha$ we can use it to compute the structure maps of $Y$. For details, we refer to
Paragraph 6.3 in \cite{djk_arXiv}, where this problem is the special case of considering the degree-$0$ part of the map $(x^{\beta - \alpha} \cdot -) \colon Y[\alpha] \to Y[\beta]$ for any $\alpha \leq \beta$.

\autoref{alg:hom_exact} then assembles these to compute the matrix from \autoref{prop:exact_correct}. The full pseudocode can be found in the appendix \autoref{sec:algs}.

\begin{proposition}\label{prop:exact_correct}
 \autoref{alg:hom_exact} is correct and has a run time in $\Oc(m^3 \thickb{Y}^3+ mn^3)$.
\end{proposition}

\begin{proof}
Correctness follows directly from \autoref{prop:M_star}. 
The matrix $M^t \otimes N$ is at most of size $m \thickb{Y} \times m \thickb{Y}$ and so computing the kernel will need less than $m^3 \thickb{Y}^3$ operations.

For the computation of $Y_{\alpha \to \beta}$, we are reducing an $\Oc(n) \times \Oc(n)$ matrix with \autoref{alg:cokernel} at most $m$ times, adding the $mn^3$-term. This implicitly produces the subsets $G'_\alpha \subset G$ which form a basis of $Y_{\alpha}$. For each pair $(\alpha \leq \beta)$, we can use this to compute the matrix $Y_{\alpha \to \beta}$ with respect to these bases. This matrix is of size at most $\thickb{Y}^2$ and consists exactly of the columns indexed by $G'_{\alpha}$ of $\coker N_{\leq \beta}$, so we only need to copy the entries and this assembly takes $\Oc(m^2 \thickb{Y}^2)$ time.
 $|J| = \dim_\K \Hom(X,Y) \leq m \thick{Y}$ by \autoref{cor:hom_size} and the matrices $Q_j$ have size $\Oc(n) \times \Oc(m)$, but at most $m \thick{Y}$ entries, so the computation of the inverses under $\varphi$ takes at most $m^2\thickb{Y}^2$ time, too.
\end{proof}

\begin{remark}\label{rem:fast_mm}
When using fast matrix multiplication to solve the systems in \autoref{alg:hom_exact} its runtime is $\Oc( m^\omega \thickb{Y}^\omega + mn^\omega)$.
\end{remark}

In fact $\thick{Y}$ isn't the sharpest bound here in \autoref{alg:hom_exact} or in \autoref{alg:restricted}, because the vector spaces $Y_{\deg(g)}$ and $Y_{\deg(r)}$ only sit at the locations of generators and relations.

\begin{corollary}\label{cor:thickness_area}
In \autoref{prop:exact_correct} we can replace $\thick{Y}$ by $\max_{\alpha \in b_0(X) \cup b_1(X)} \dim_\K Y_\alpha$ and in \autoref{prop:A_runtime} we can replace $\thick{Y}$ by $\max_{\alpha \in b_0(X)} \dim_\K X_\alpha$ and $\thick{\Omega^{-1}Y}$ by $\max_{\alpha \in b_1(X)} \dim_\K \left(\Omega^{-1}Y\right)_\alpha$.
\end{corollary}

\subparagraph{Computing the Module of Homomorphisms.}

\begin{definition}\label{def:enrichment}
The category of persistence modules, $\grA$, is enriched over itself: $\Hom(X,Y)$ becomes the degree-$0$ part of a graded $A$-module by setting $\Hom(X, Y)_\alpha \coloneqq \Hom(X, Y[\alpha])$.
\end{definition}

With only a little more work, we can compute a presentation of the \emph{module} $\Hom(X,Y)$ with the ideas from this section. The enriched version of the exact sequence \autoref{eq:exact} yields

\begin{align}\label{eq:module} 0 \lra \Hom(X, \, Y) \to \bigoplus_{g \in G} Y[\deg(g)] \oto{ \left(M^*_{r,g}\right)_{r \in R, g \in G}} \bigoplus_{r \in R} Y[\deg(r)].\end{align}
where now $\left(M^*_{r,g}\right)_{r \in R, g \in G}$ is a map of persistence modules. After extending a presentation of $Y$ to presentations $N', N''$ of the direct products and lifting $M^*_{r,g}$ to a graded matrix $Q$, a cover for $\Hom(X, \, Y)$ is given by $\ker Q \oto{i} \bigoplus_g A[G'-\deg(g)]$ and a presentation $P$ by $\ker [ -i \ N']$. We remind ourselves again that this can be done efficiently (so far) only for $2$ adn $3$ parameters.

Alternatively, using the dual short exact sequence, we can compute a presentation as a cokernel using duality of projective and injective resolutions \cite{Miller2001, BLL23} and Proposition 4.14. in \cite{djk_arXiv}. See also the next chapter.

\begin{question}\label{qu:enriched}
Can the approach of \autoref{alg:restricted} be extended to compute the module $\Hom(X,Y)$?
\end{question}

\section{Duality and Homomorphisms}

The preceding algorithms exploited the low-dimensionality of the target module $Y$. Assume now that, instead, 
the domain $X$ is low-dimensional compared to $Y$.

Since Matlis duality is a contravariant equivalence, we have
$\Hom(X, \, Y) \simeq \Hom(\D Y, \, \D X)$
and so we could just compute the latter with our algorithms to exploit the low-dimensionality of $X$. Equivalently, this amounts to reformulating \autoref{alg:restricted} and \autoref{alg:hom_exact} for \emph{injective} resolutions, by the following standard result (E.g. \cite[Proposition 3.5]{miller00})

\begin{proposition}\label{prop:proj_inj}
If $P_\bullet \to X$ is a projective resolution, then $\D X \to \D P_\bullet$ is an injective resolution.
\end{proposition}

We use cohomological grading for chain complexes from this point on. That is, a projective resolution of a module will sit in negative degrees and an injective resolution in positive degrees.
We remind the reader that $D$ reverses the homological grading of a chain complex.

\begin{itemize}
\item Dual of \autoref{alg:restricted}: Instead of computing the \emph{lifts} of a homomorphism $X \to Y$ to presentations, we compute the \emph{extensions} to injective co-presentations.
\item Dual of \autoref{alg:hom_exact}: Instead of using left-exactness of $\Hom(-,Y)$, we use right-exactness of $\Hom(X,-)$ and an injective co-presentation of $Y$.
\end{itemize} 

Fortunately, we know how to compute injective resolutions from projective presentations thanks to the work of Miller \cite{miller00, Miller2001} and a recent re-discovery of these ideas in \cite{BLL23}.
We will give a short overview of these results

\subparagraph{Equivalence of Injective and Free Resolutions.}

\begin{theorem}[{\cite[Theorem 5.3]{Miller2001}, \cite[Theorem 10]{BLL23}}]\label{thm:local_duality}
Let $C_\bullet$ be a chain complex of free modules whose homology has bounded support. 
There is an isomorphism in the derived category 
\[ C_\bullet \iso \Sigma^{-d} \D \Hom(C_\bullet, \, A) [-\vec{1}]. \]
\end{theorem}

\begin{corollary}\label{cor:resolution_duality}
If $X$ has bounded support and $\dots P_{-n} \to \dots \to P_{-1} \to P_0 \to X$ is a projective resolution of $X$, then 
\[ X \to \D \Hom(P_n, \, A) [-\vec{1}] \to \D \Hom(P_{n-1}, \, A) [-\vec{1}] \to \dots \D \Hom(P_0, \, A) [-\vec{1}] \to 0 \] 
 is an injective resolution of $X$.
\end{corollary}

This means that we get injective resolutions 
basically for free if we compute the whole projective resolution, since the boundary matrices in the injective resolutions are, as graded matrices, the same as the ones in the projective resolution.

In practice, the modules $X$, $Y$ will not have bounded support, but this is not actually an obstacle. We can truncate both modules without changing the homomorphisms between them as long as their Betti numbers are bounded. This is always the case if they are finitely presented.

\begin{notation}
Let $\omega \in \Z^d$. We denote by $\iota \colon \downarrow \omega \into \Z^d$ the inclusion of the lower set generated by $\omega$ and by $\iota^* \dashv \iota_* \colon \grA \to \fun(\downarrow \omega, \vect_\K)$ the restriction and right Kan-extension.
\end{notation}

\begin{proposition}\label{prop:fully_faithful}
Let $X, \ Y \in \grA$ have their Betti numbers bounded by $\omega$. Then \[
\Hom(X, Y) \simeq \Hom(\iota_* \iota^* X, \iota_* \iota^* Y)
\] 
\end{proposition}

\begin{proof}
By \cite[Proposition 4.5]{djk_arXiv} the functor $\iota^*$ is an equivalence on the subcategory of modules with Betti-numbers bounded by $\omega$. Furthermore, the counit $\iota^*\iota_* \to \Id$ is an isomorphism since $\downarrow \omega$ is downward closed. We deduce:
\begin{align*} 
    & \Hom(\iota_* \iota^* X, \iota_* \iota^* Y)  
    \iso \Hom_{\fun(\downarrow \omega, \vect_\K)}(\iota^* \iota_* \iota^* X, \iota^* Y) \\
    \iso & \Hom_{\fun(\downarrow \omega, \vect_\K)}(\iota^* X, \iota^* Y) 
    \iso \Hom(X, Y).
\end{align*}
\end{proof}

We can construct a presentation of $\iota_* \iota^* X$ from $X$ by introducing, for each generator $g \in G$ and each $i \in [d]$, a relation $g = 0$ at $\omega_i \in \Z$. 
Moreover, we essentially compute the $\Ext^d(-,\, -)$ groups for free, but we have to be careful: \autoref{prop:fully_faithful} does not work for higher $\Hom$s, so we will compute $\Ext^d(-,\, -)$ of the truncated modules and not of the unbounded ones.

\begin{proposition}\label{prop:duality}
Let $X, Y \in \Der(\grA)$ where $X$ is a perfect complex such that its homology has bounded support. There is a natural equivalence of functors
\[ \Hom_{\Der(\grA)}(Y,\,X) \iso \Sigma^{-d} \D \left( \Hom_{\Der(\grA)}(X, \, Y[-\vec{1}] ) \right)\]
\end{proposition}

\begin{proof}
\begin{align*}
    & \Hom_{\Der(\grA)}(Y,\,X) \\
    \overset{\makebox[3cm]{$X$ is perfect + \autoref{thm:local_duality}}}{\iso}
    & \Hom_{\Der(\grA)}(Y,\,  \Sigma^{-d} \D \Hom_{\Der(\grA)}( X, \, A[- \vec{1}]) ) \\
     \overset{\makebox[3cm]{Def. of $\D$}}{\iso}
     & \Sigma^{-d} \Hom_{\Der(\grA)}(\Hom_{\Der(\grA)}( X, \, A[- \vec{1}]), \Hom_{\Der(\grA)}(Y, I(0)) ) \\
      \overset{\makebox[3cm]{$\otimes \, \dashv \,  \Hom$}}{\iso}
      & \Sigma^{-d} \Hom_{\Der(\grA)}(\Hom_{\Der(\grA)}( X, \, A[-\vec{1}]) \otimes^L Y, I(0) ) \\
      \overset{\makebox[3cm]{$X$ is dualizable}}{\iso}
      & \Sigma^{-d} \Hom_{\Der(\grA)}(\Hom_{\Der(\grA)}( X, \, Y[-\vec{1}]), I(0) ) \\
      \overset{\makebox[3cm]{Def. of $\D$}}{\iso}
      &\Sigma^{-d} \D \left( \Hom_{\Der(\grA)}( X, \, Y[-\vec{1}]) \right)
\end{align*}
\end{proof}

\begin{corollary}\label{cor:duality}
Let $X, Y \in \grA$ be finitely generated and either $X$ or $Y$ have bounded support. There is a natural equivalence
\[\Ext^i(X, \, Y) \iso D \Ext^{N-i}(Y, \, X)[-1].\]
\end{corollary}

\subparagraph{Output of the dual Algorithms.}
If we compute $\Hom(X,Y)$ via an injective resolution, then we will get a representative of a homomorphism not in terms of the generators of $X$ and $Y$, but in terms of their \emph{co}generators, that is, in a basis of the injective envelopes of $X$ and $Y$. By \autoref{thm:local_duality}, this is equivalently a set of generators for $DX$ and $DY$ and so they exist generally only after truncating or have to be assumed to lie at $\infty$ in some coordinates. The latter can be seen as passing to a compactification $\fun( \Z \ \cup \{\infty, - \infty\}, \vect_\K)$ of $\grA$.

Depending on the application, one can continue working with the injective envelopes. If instead we need to work with the presentations, this leaves the question of how to compute the lifts to the projective covers. Let now $f \in \Hom(X,Y)$ be such a map and $f^0 \colon I_0 \to I'_0$ its extension to envelopes, the output from either of the dual algorithms. 

One possibility is to compute iterative extensions of $f^0$ to get a chain map $f^* \colon I_* \to I'_*$ between the whole injective resolutions. Then, using \autoref{thm:local_duality} again, the desired lift to covers $f_0 \colon F_0 \to F_0'$ is just $D\Hom(f^d, A)$, as in the following diagram.

% https://q.uiver.app/#q=WzAsMTQsWzIsMCwiWCJdLFsyLDEsIlkiXSxbMywwLCJJXzAiXSxbNCwwLCJJXzEiXSxbNSwwLCJcXGRvdHMiXSxbNiwwLCJJX2QiXSxbMywxLCJJXzAnIl0sWzQsMSwiSV8xJyJdLFs2LDEsIklfZCciXSxbNSwxLCJcXGRvdHMiXSxbMSwwLCJEXFxIb20oSV9kLCBBKSJdLFsxLDEsIkRcXEhvbShJX2QnLCBBKSJdLFswLDAsIlxcZG90cyJdLFswLDEsIlxcZG90cyJdLFswLDIsImReMCJdLFsxLDYsImVeMCJdLFsyLDMsImReMSJdLFs2LDcsImVeMSJdLFszLDQsImReMiJdLFs3LDksImVeMiJdLFs0LDVdLFs5LDhdLFsyLDYsImZeMCJdLFszLDcsImZeMSJdLFs1LDgsImZeZCJdLFswLDEsImYiXSxbMTAsMCwiZF8wIl0sWzExLDEsImVfMCJdLFsxMCwxMSwiRFxcSG9tKGZeZCwgQSkiLDJdLFsxMiwxMF0sWzEzLDExXV0=
\[\begin{tikzcd}[ampersand replacement=\&]
	\dots \& {D\Hom(I_d, A)} \& X \& {I_0} \& {I_1} \& \dots \& {I_d} \\
	\dots \& {D\Hom(I_d', A)} \& Y \& {I_0'} \& {I_1'} \& \dots \& {I_d'}
	\arrow[from=1-1, to=1-2]
	\arrow["{d_0}", from=1-2, to=1-3]
	\arrow["{D\Hom(f^d, A)}"', from=1-2, to=2-2]
	\arrow["{d^0}", from=1-3, to=1-4]
	\arrow["f", from=1-3, to=2-3]
	\arrow["{d^1}", from=1-4, to=1-5]
	\arrow["{f^0}", from=1-4, to=2-4]
	\arrow["{d^2}", from=1-5, to=1-6]
	\arrow["{f^1}", from=1-5, to=2-5]
	\arrow[from=1-6, to=1-7]
	\arrow["{f^d}", from=1-7, to=2-7]
	\arrow[from=2-1, to=2-2]
	\arrow["{e_0}", from=2-2, to=2-3]
	\arrow["{e^0}", from=2-3, to=2-4]
	\arrow["{e^1}", from=2-4, to=2-5]
	\arrow["{e^2}", from=2-5, to=2-6]
	\arrow[from=2-6, to=2-7]
\end{tikzcd}\]

If $f^d$ is a graded matrix after choosing bases, then both duality functors just transpose this matrix, so as a graded matrix, $f^d$ is already the correct result. Still, doing this would involve solving the linear system $f^{i} \circ d^{i-1} = e^{i-1} \circ f^{i-1}$ for each $i \in 1 \dots d$, for every computed $f^0 \in \Hom(X,Y)$. If instead we knew the maps $d^0 \circ d_0$ and $e^0 \circ e_0$, which can also be represented by graded matrices, we would only have to solve one linear system to compute $f^d$: $ (e^0 \circ e_0) \circ D\Hom(f^d,A) = f^0 \circ d^0 \circ d_0 $. These composites are 
\emph{flat-injective presentations} of $X$ and $Y$ \cite{Miller2001}. Thankfully, Lenzen recently showed how to compute these in $\Oc(n^3)$/$\Oc(m^3)$ time \cite[Theorem 5.7]{lenzen24} \cite[Corollary 3.8.3]{Lenzen_diss}, and also provides an implementation.

\subsection{Duality and Algorithm A}

Given a minimal injective resolution $X \oto{d^0} I_0 \oto{d^1} \dots$ there is a dual version of the syzygy module: $\Omega^{-i}(X) \coloneqq \ker d^{i}$. It follows that there is an equivalent notion of a restriction systems for injective resolutions. Using Matlis duality, the dual of \autoref{thm:main} then says that any homomorphism can be extended uniquely to to an injective restriction system.

\begin{algorithm}[H]
\renewcommand{\thealgocf}{A*}
\caption{Dual Restricted Computation of Homomorphisms of Presentations.}
\label{alg:restricted_dual}
\DontPrintSemicolon
\KwIn{Presentation matrices $M \in \K^{G \times R}$, $N \in \K^{G' \times R'}$ of \emph{bounded} modules.}
\KwOut{Co-generators $S, S'$ of $\coker M, \coker N$ and graded matrices $\{Q_i\}_{i \in I} \in \K^{S' \times S}$ which represent a basis of $\Hom(X, \, Y)$ coordinatised in $S, S'$.}

Complete $M, N$ to projective resolutions with last terms $O \colon F^S \to F^T$ and $P \colon F^{S'} \to F^{T'}$ ;

$\{Q^t_i\}_{i \in I} \in \K^{S' \times S} \coloneqq $ \autoref{alg:restricted}( $P^t[(1, \dots ,1)], \, O^t[(1, \dots ,1)]$ ) ;

return $\{Q_i\}_{i \in I}$ ;

\end{algorithm}

\begin{proposition}
\autoref{alg:restricted_dual} is correct and for $2$-parameter modules, using \emph{column-reduction}, needs 
$\Oc( m^3 n \left( \thick{X} + \thick{  \Omega^{-1}(X) } \right)^2 + nm^3)$ time and 
$\Oc \left( n^2 m \left( \thickb{X} + \thick{ \Omega^{-1}(X)} \right)  +nm^2\right)$ memory.
\end{proposition}

\begin{proof}
Follows from \autoref{thm:alg_a} and \autoref{thm:local_duality}. For the run-time and memory observe that $\Omega^1(DX) \simeq D\Omega^{-1}(X)$.
\end{proof}

\subsection{Duality and Algorithm B}
To not only appeal to abstract duality, we will describe the linear system that one would compute.
Let $0 \to Y \to I^S \oto{P} I^R$ be an injective co-presentation. Applying $\Hom(X, - )$ yields the exact sequence
\[0 \to \Hom(X, Y) \to \Hom(X, I^S)  \oto{P_*} \Hom(X, I^R). \]

\begin{proposition}\label{prop:dual_description}
The map $\bigoplus\limits_{s \in S} X_s \iso \Hom(X, I^S)  \oto{P_*} \Hom(X, I^R) \iso \bigoplus\limits_{r \in R} X_r $ is given by
\[ \left( P_{r, s} \cdot X^t_{\deg(r) \to \deg(s)} \right)_{r \in R,s \in S} .\]
\end{proposition}

\begin{proof}
Dual to \autoref{prop:M_star}.
\end{proof}

\begin{algorithm}[H]
\renewcommand{\thealgocf}{B*}
\caption{Dual Restricted Computation of Homomorphisms of Presentations.}
\label{alg:exact_dual}
\DontPrintSemicolon
\KwIn{Presentation matrices $M \in \K^{G \times R}$, $N \in \K^{G' \times R'}$ of \emph{bounded} modules.}
\KwOut{Co-generators $S, S'$ of $\coker M, \coker N$ and graded matrices $\{Q_i\}_{i \in I} \in \K^{S' \times S}$ which represent a basis of $\Hom(X, \, Y)$ coordinatised in $S, S'$.}

Complete $M, N$ to projective resolutions with last terms $O \colon F^S \to F^T$ and $P \colon F^{S'} \to F^{T'}$ ;

$\{Q^t_i\}_{i \in I} \in \K^{S' \times S} \coloneqq $ \autoref{alg:hom_exact}( $P^t[(1, \dots ,1)], \, O^t[(1, \dots ,1)]$ ) ;

return $\{Q_i\}_{i \in I}$ ;

\end{algorithm}

\begin{proposition}\label{prop:dual_exact_correct}
\autoref{alg:exact_dual} is correct and has a worst-case runtime in $\Oc(n^3 \thickb{X}^3 + nm^3)$.
\end{proposition}

\begin{proof}
Follows from \autoref{prop:dual_description} or \autoref{prop:exact_correct} and \autoref{thm:local_duality}.
\end{proof}

\section{Experiments}

We will use the "$1.5$mm" data set from \cite{Vipond21} containing the locations of different immune cells (CD8, CD68, FoxP3) in histological samples of a tumour. We have computed density-Delaunay bifiltrations \cite{akll-delaunay} and both $H_0$ and $H_1$ of these $2$-dimensional point sets. Our code is publicly available \footnote{https://github.com/JanJend/Persistence-Algebra}

\subparagraph{Layer-Thickness and Thickness}
We have computed the layer-thickness of every module from the data set for \autoref{fig:layer_thick}. Observe that $H_0$ produces relatively thick indecomposables because all generators appear at scale $0$, whereas every edges has a non-zero length and thus appears later. $H_1$-modules have rather thin layers.

\begin{figure}[H]
\centering
\includegraphics[width=0.8\linewidth]{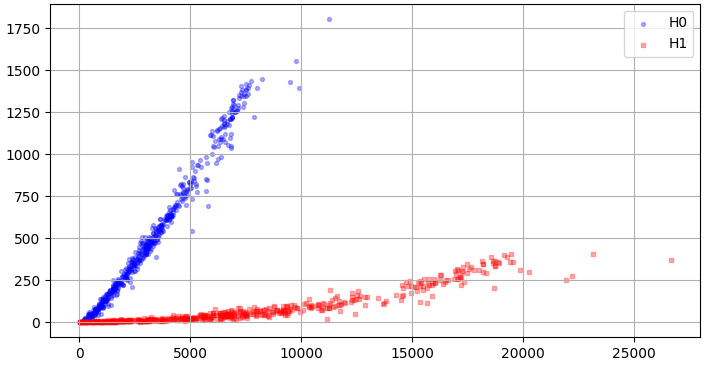}
\caption{Layer-thickness (y) and $|b_0+b_1|$ (x) for $H_0$ and $H_1$.}
\label{fig:layer_thick}
\end{figure}

The thickness of indecomposable summands also correlates strongly with the cardinality of their graded Betti-numbers in this dataset as visible in \autoref{fig:thick_distribution}.

\subparagraph{Runtimes.}
We have tested every non-dual algorithm by computing a basis of $\End(X)$ for each non-interval indecomposable summand $X$ appearing in the $H_1$-modules from the $1.5$mm dataset. 
A logarithmic plot, \autoref{fig:runtime_log} is found in the appendix.

For this comparison, we also included a version, "No Reduction", of \autoref{alg:direct} \emph{without} the reduction step in \autoref{homotopy_reduction}. Surprisingly, "No Reduction" is generally faster than \autoref{alg:hom_exact} in this experiment (\autoref{fig:runtime_log}). Because it is difficult to see how the algorithms fare on average on this dataset, we have the functions $n^{c_i}$ that best approximate the run-times.  We show these curves, overlaid on the non-scaled runtimes2, in the next \autoref{fig:regression}. 

\begin{figure}[H]
\includegraphics[width=0.9\linewidth]{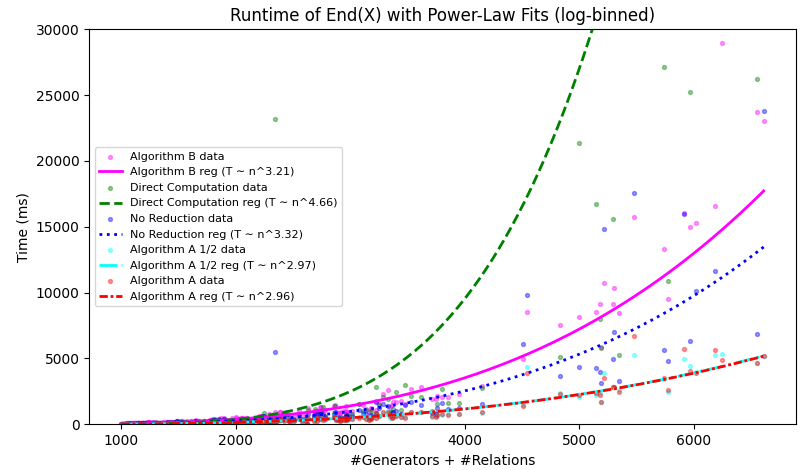}
\caption{Estimates for exponents of runtime on this dataset. Some runtimes of \autoref{alg:direct} are too high to be visible in this plot.}
\label{fig:regression}
\end{figure}

The $c_i$ are between $3$ for \autoref{alg:restricted} and \autoref{alg:mixed} and $4.7$ for \autoref{alg:direct}. Since \autoref{fig:thick_distribution} shows a roughly linear correlation between $\thick{X}$ and size, we suspect that these low exponents come from the sparsity of the input matrices (\autoref{prop:thickness_sparsity}).

\subparagraph{Multicover Bifiltrations.}
We consider modules coming from the multicover bifiltration and a $3d$-mesh filtered by $x$- and $y$-coordinate because these are qualitatively different than the modules from the density-alha bifiltrations we have used in the previous paragraphs. 

In \autoref{tab:experiment_2_results} we calculate the sizes and sparsity of the main linear systems solved by the algorithms.
We see that, as predicted by the complexity analysis, \autoref{alg:hom_exact} produces the smallest linear system. Similarly to \autoref{fig:regression}, the other calculations in \autoref{alg:hom_exact} and \autoref{alg:restricted} need so much time that the order of run-times of the algorithms is exactly the opposite of the order of the system-sizes.

\begin{table}[H]
\centering
\footnotesize
\begin{tabular}{@{}l l
    S[table-format=5.0, table-space-text-post=k]
    S[table-format=5.0, table-space-text-post=k]
    S[table-format=1.2]
    S[table-format=3.2]@{}}
\toprule
\textbf{Module Type} & \textbf{Variable} & \textbf{Alg B} & \textbf{Alg A} & \textbf{Alg A}$\frac{1}{2}$ & \textbf{Standard} \\
\midrule
\multirow{5}{*}{\shortstack[l]{Multicover \\ 46 Points \\$H_1$}}
    & Variables        & {\textbf{8k}}       & {425k}     & {652k}     & {803k} \\
    & Equations        & {14k}      & {431k}     & {431k}     & {431k} \\
    & Avg no. Entries  & {1.00}     & {1.46}     & {1.50}     & {1.90} \\
    & Time (s)         & {0.80}     & {0.53}     & {\textbf{0.49}}     & {0.98} \\
\midrule
\multirow{5}{*}{\shortstack[l]{Multicover \\ 96 Points \\ $H_1$}}
    & Variables        & {\textbf{44k}}     & {3710k}    & {5968k}    & {7396k} \\
    & Equations        & {78k}      & {3745k}    & {3745k}    & {3745k} \\
    & Avg no. Entries  & {1.18}     & {1.62}     & {2.06}     & {1.80} \\
    & Time (s)         & {18.40}    & {11.76}    & {\textbf{7.62}}     & {19.26} \\
\midrule
\multirow{5}{*}{\shortstack[l]{Multicover \\ 175 Points  \\$H_1$}}
    & Variables        & {\textbf{163k}}     & {17154k}   & {27983k}   & {34748k} \\
    & Equations        & {300k}     & {17291k}   & {17291k}   & {17291k} \\
    & Avg no. Entries  & {0.96}     & {1.80}     & {2.50}     & {2.16} \\
    & Time (s)         & {222.95}   & {117.21}   & {\textbf{68.09}}    & {251.70} \\
\midrule
\midrule
\multirow{5}{*}{\shortstack[l]{x-y-filtration \\
of 3D Model \\ of a Hand 
$H_0$}}
    & Variables        & {\textbf{6014k}}    & {26116k}   & {26116k}   & {38766k} \\
    & Equations        & {7517k}    & {27620k}   & {27620k}   & {27620k} \\
    & Avg no. Entries  & {2.18}     & {2.86}     & {2.60}     & {2.76} \\
    & Time (s)         & {474.79}   & {183.40}   & {\textbf{172.64}}   & {timeout} \\
\bottomrule
\end{tabular}
\caption{Size of linear system and wall time to compute a basis of $\End(X)$ for large sparse presentation matrices from the multicover bi-filtration.}
\label{tab:experiment_2_results}
\end{table}

\section{Conclusion and Open Problems}

\subsection*{Improved Runtime of \textsc{AIDA}}

\begin{corollary}
Let $X$ be be presented by a graded matrix of size at most $n \times n$ and have at most $k$ relations of the same degree.  \autoref{alg:hom_exact} improves the runtime of \textsc{AIDA} on $X$ to 
\[ \Oc \left( n^{\omega +1} \left(  n +  k^{\omega -1}(k+\thick{X})q^{k^2/4 + \Oc(k)} + \sum_{Y \in \text{Ind}(X)} \thickb{Y}^\omega  \right) \right) \] 
\end{corollary}

\begin{proof}

\textsc{AIDA} traverses the columns -- or rather groups of columns of the same degree -- of the presentation matrix of $X$ in some order compatible with $\R^d$. There are at most $n$ of these groups, assuming $k$ to be small. Let $\alpha$ be the degree of the group of columns currently being processed. Then there is an upset $U \subset \R^d$ containing $\alpha$, in which $\alpha$ is minimal, such that the processed part of the matrix is decomposed into blocks which present indecomposable summands $(X_b)_{b \in \mathcal{B}}$. We denote their size by $n_b \coloneqq b_0(X_b) + b_1(X_b)$. Formally they form an indecomposable decomposition of a $U$-cover \cite[Definition 4.11]{djk_arXiv} of $X$.

The first term of \textsc{AIDA}'s runtime \cite[Theorem 5.5]{djk_arXiv} comes from computing a basis of the space $\Hom(X_b, X_c)$ of homomorphisms for some pairs of indecomposables $X_b, X_c$, but at most those for which the current group of columns has any non-zero entries. We denote this set by $\mathcal{B}' \subset \mathcal{B}$. 
 For any $b \in \mathcal{B}'$, $X_b$ cannot have any generators in $U$: Since $X_b$ is $U$-projective, any generator in $U$ would form a free summand which can be split off or it already generates the whole module. If this generator has degree $\alpha$ then no column in the current group can have a nonzero entry at this generator due to minimality, and if it has degree not equal to $\alpha$, then it cannot be part of a relation due to degree-incompatibility.

Recall that the sum $\bigcup_{b \in \mathcal{B}} X_b$ forms a $U$-cover of $X$, so over $\R^d \setminus U$ this module is already isomorphic to $X$, whereas over $U$, $X$ is a quotient of this sum.
Then, the sum of the thicknesses of the indecomposables $X_b$ is not bounded by the thickness of $X$, but it is over $\R^d \setminus U$! We denote by $\mathfrak{t}^{\R^d \setminus U} (X) \coloneqq \max\limits_{\alpha \in \R^d \setminus U} \dim X_\alpha$ the thickness over $\R^d \setminus U$. We can therefore use \autoref{cor:thickness_area} to calculate an upper bound for the time to compute a basis of the space $\Hom(X_b, X_c)$ for each pair $b \neq c \in \mathcal{B}'$ with \autoref{alg:hom_exact}. Using fast matrix multiplicationwith $\omega$ the matrix multiplication constant, and with $c_0 \in \mathcal{B}$ a block where $\mathfrak{t}^{\R^d \setminus U} \left(X_c \right)$ is maximal this bound is, using \autoref{rem:fast_mm}

\begin{align*} \sum\limits_{b \neq c \in \mathcal{B}'} n_b^\omega   \mathfrak{t}^{\R^d \setminus U} \left(X_c \right)^\omega + n_b n_c^\omega
\leq \left( \sum\limits_{b \in \mathcal{B}'} n_b^\omega \right) \left( \sum\limits_{c \in \mathcal{B}'} \mathfrak{t}^{\R^d \setminus U} \left(X_c \right)^\omega \right) 
+ \left( \sum\limits_{b \in \mathcal{B}'} n_b \right) \left( \sum\limits_{c \in \mathcal{B}'} n^\omega_c \right) \\
\leq  \left( \sum\limits_{b \in \mathcal{B}'} n_b \right)^\omega \left( \sum\limits_{c \in \mathcal{B}'} \mathfrak{t}^{\R^d \setminus U} \left(X_{c_0} \right)^\omega  \right) 
+ n \left( \sum\limits_{c \in \mathcal{B}'} n_c \right)^\omega
 \leq  n^\omega \left( n + \sum_{Y \in \text{Ind}(X)} \thickb{Y}^\omega \right).
\end{align*}

Then for each of the $\Oc(q^{k^2/4 + \Oc(k)})$ iterations over subspaces and each $c \in \mathcal{B}'$  we need to solve the linear system $( \ast \ast)$ in \cite[Proposition 3.4]{djk_arXiv}. It contains one variable for each element in the computed basis of $\bigoplus\limits_{b \in \mathcal{B}'} \Hom(X_b, X_c)$. With \autoref{cor:hom_size} we can improve the bound on the size of this basis, too, to $n \mathfrak{t}^{\R^d \setminus U} (X_c)$. Then there are an additional $k n_c$ variables which control possible column operations and also $k n_c$ equations in this system.  This puts the computations in this step to at most

\[ \sum\limits_{c \in \mathcal{B'}} (k n_c)^{\omega-1}\left(n  \mathfrak{t}^{\R^d \setminus U} \left(X_c \right) +k n_c\right) \leq k^{\omega-1} \left( \sum\limits_{c \in \mathcal{B'}} n_c \right)^{\omega-1}\left(\lthick{X} n  +k n\right) = n^{\omega}k^{\omega -1}\left( \lthick{X}+k\right) .\]

\end{proof}

\subparagraph{Sparsity.}
All linear systems constructed by the algorithms are as sparse as $M$ and $N$, so by \autoref{prop:thickness_sparsity} their number of entries are roughly bounded by $\thick{X}$ and $\thick{Y}$, too. Because of possible fill up, this does not change the worst case run-times, but gives a heuristic about the practical run-time of \autoref{alg:restricted} and \hyperref[alg:hom_exact]{B}.

\subparagraph{Intervals.}

In \cite{dey_xin}, the authors sketch an algorithm which computes the set of homomorphisms for interval modules in $\Oc(n)$ for $d=2$ and $\Oc(n^2)$ for $d>2$. This is much faster than the algorithms we have presented. We are thus wondering if there is an algorithm which exploits pointwise low-dimensionality of both domain and target. On top, these algorithms for interval-modules have not been implemented yet and so we cannot compare their practical run-time against our algorithms.

\subparagraph{Open Problems and Future Work.}
\begin{itemize}
\item Find a way to measure how far away a module is from having thickness $k$ in an algebraic sense.
    \item \autoref{qu:enriched}: Can the idea of \autoref{alg:restricted} be extended to compute presentations of the \emph{persistence module} $\Hom(X,Y)$ similar to how this is possible for \autoref{alg:hom_exact}?
   \item Implement the algorithms for interval-modules from \cite{dey_xin}.
     \item Is there an algorithm that benefits from both $X$ \emph{and} $Y$ having low thickness?
    \item We conjecture that by subdividing $Z^d$, there is a faster way to compute all cokernels of $M_{\leq g}$ for $g \in G$ in one sweep with runtime in $\Oc(dn^3)$, instead of calling \autoref{alg:cokernel} each time. This would cut the runtime of \autoref{alg:hom_exact} down to $\Oc(m^\omega \thickb{Y}^\omega + dn^3)$ and would similarly effect the runtime of \textsc{aida}. 
\end{itemize}

\subsection*{Summary}
We have presented $3$ algorithms, and their duals, to compute the vector space of degree-$0$ homomorphisms.

The classical approach in \autoref{alg:hom_exact} achieves the best worst-case complexity and is independent of the number of parameters
$d$. We use it to accelerate \textsc{aida}, thereby improving on the $\mathcal{O}(n^{2\omega+1})$ runtime of Dey and Xin \cite{DeyXin} for decomposing modules with uniquely graded relations. The new \autoref{alg:restricted} also improves on the standard computation, at least in the important case $d=2$; however, for higher $d$ its runtime additionally depends on the complexity of computing kernels of graded matrices, which remains unknown for $d \geq 3$.

On large indecomposable modules coming from relevant real-world data (point sets in $\R^2$ from immune-cell locations \cite{Vipond21}), \autoref{alg:restricted} and its simpler version \autoref{alg:mixed} fare substantially better
than the classical approach (\autoref{alg:hom_exact}) and the direct computation (\autoref{alg:direct}), even if we omit the reduction to a basis in the latter (\autoref{fig:runtime_log}). 

Our results directly accelerate the computation of indecomposable decompositions with \textsc{aida}, the computation of Bjerkevik's \emph{Pruning} \cite{Bjerkevik2025} to stabilise decompositions (with H\aa vard Bjerkevik and Fabian Lenzen \footnote{https://github.com/JanJend/Stable-Decomposition}), and the algorithm in \cite{FJ26} that we devised with Marc Fersztand to compute the \emph{Skyscraper Invariant} because it uses \textsc{aida} heavily. 
Because the computation of homomorphisms is such a fundamental operation, our results promise to improve many future algorithms for Multiparameter Persistence Modules.

At last, we want to remark again that our theorems and algorithms do not depend on $\Z^d$. Every statement is valid for finitely presented modules over any partially ordered set. 

\printbibliography

\appendix

\section{Algorithms}\label{sec:algs}

\begin{algorithm}[H]
\caption{Computation of Local Cokernels}
\label{alg:cokernel}
\DontPrintSemicolon
\KwIn{A presentation $N\in \K^{G \times R}$ of a module $X$ and a degree $\alpha \in \R^d$}
\KwOut{\begin{itemize} \item The cokernel $d_\alpha \colon \left(A[G]\right)_\alpha \simeq \K^{|{G\leq\alpha}|} \to X_\alpha$ of $N_\alpha$ 
\item
An ordered sub-set $G_\alpha \subset G$, such that $d_\alpha(G_\alpha)$ forms a basis for $X_\alpha$
\end{itemize}}
The matrix $N_{\leq \alpha} \colon \left(A[R]\right)_\alpha \to \left(A[G]\right)_\alpha$ is given by the sub-matrix of $N$ on the sets of generators and relations $\set{G \leq \alpha}$ and $\set{R \leq \alpha}$ whose degrees are smaller than $\alpha$\;
Solve $N_{\leq \alpha}^{t}x = 0$ with Gaussian Elimination to compute the cokernel $d_\alpha$\; \label{solve_transpose} 
Initialise $G_\alpha\coloneqq \emptyset$  \;
\For{$g \in \set{G \leq \alpha}$ }{
    \If{$\left(N_{\leq \alpha}\right)_{g, \bullet}$ does not contain a pivot}{
        Append $g$ to \texttt{$G_\alpha$} \;
    }
}
\Return $G_\alpha$ and $d_\alpha$\;
\end{algorithm}

\begin{algorithm}[H]
\renewcommand{\thealgocf}{A-$1/2$}
\caption{Mixed Computation of Homomorphisms from Presentations.}
\label{alg:mixed}
\DontPrintSemicolon
\KwIn{Presentation matrices $M \in \K^{G \times R}$, $N \in \K^{G' \times R'}$}
\KwOut{Graded matrices $\{Q_i\}_{i \in I} \in \K^{G' \times G}$ which descend to a basis of $\Hom(X, \, Y)$}

\For{$\alpha \in \deg(G)$}{
    $d_\alpha, \, G'_\alpha   \gets$ \autoref{alg:cokernel}$(N, \alpha)$\;
    \For{$g \in G$ with $\deg(g) = \alpha$}{
    $Q_{g', g} \gets \begin{cases}
        \text{a variable } q_{g',g} &\text{ if } g' \in G'_\alpha \\
        0 &\text{ otherwise}
    \end{cases}$
    }
} 
Calculate a basis $(Q_j, \, P_j)_{j \in J} \in \K^{G' \times G}, \K^{R' \times R}$ of the solution space of the system $Q M = N P$ \label{line:semirestricted}\;
Perform the reduction of \autoref{alg:direct}, \autoref{homotopy_reduction}\;
$J' \gets \{j \in J \, | \ Q_j \neq 0 \}$\;
\KwRet{$\left(Q_j\right)_{j \in J'}$}
\end{algorithm}

\begin{algorithm}[H]
\renewcommand{\thealgocf}{B} % override label
\caption{Computation via Hom-Exactness}
\label{alg:hom_exact}
\DontPrintSemicolon
\KwIn{Presentation matrices $M \in \K^{G \times R}$, $N \in \K^{G' \times R'}$ for $X$ and $Y$}
\KwOut{Graded matrices $\{Q_i\}_{i \in I} \in \K^{G' \times G}$ which descend to a basis of $\Hom(X, \, Y)$}
\For{$\alpha \in \deg(G') \cup \deg(R')$}{
    $e_\alpha, \, G'_\alpha   \gets$ \autoref{alg:cokernel}$(N, \alpha)$\label{line:cokernel3}\;
} 
Initialise $\vec N \in \K^{ \left( \sum_R \dim Y_{\deg(r)} \right) \times  \left( \sum_G \dim Y_{\deg(g)} \right)}$ \;
\For{$g \in G$}{
    \For{$r \in R$}{
        ${\vec N}_{r,g} \coloneqq \left(e_{\deg(r)}\right)_{|G'_{\deg(g)}}$\;
    }
}
$M^t \otimes {\vec N} \gets \left(M_{g,r} \cdot {\vec N}_{r,g}\right)_{(r,g) \in R \times G} \in \K^{ \left( \sum_R \dim Y_{\deg(r)} \right) \times  \left( \sum_G \dim Y_{\deg(g)} \right)}$\;
$K \in \K^{\sum_G \dim Y_{\deg(g)} \times J } \gets \ker M^t \otimes {\vec N}$
\label{line:kernel_exact}\;
\For{ $j \in J$}{
    Initialise $Q_j \in \K^{G' \times G}$\;
    \For{$g \in G$}{
        \For{$g' \in G'$}{
            \If{$g' \in G'_{\deg(g)}$}{
                $(Q_j)_{g', g} \gets K_{j, (g, g')}$\;
            }
            \Else{ $(Q_j)_{g', g} \gets 0$\;
            }
        }     
    }
}
\Return$(Q_{j})_{j \in J}$\;
\end{algorithm}

\section{Plots}\label{sec:plots}

\begin{figure}[H]
\centering
\includegraphics[width=0.8\linewidth]{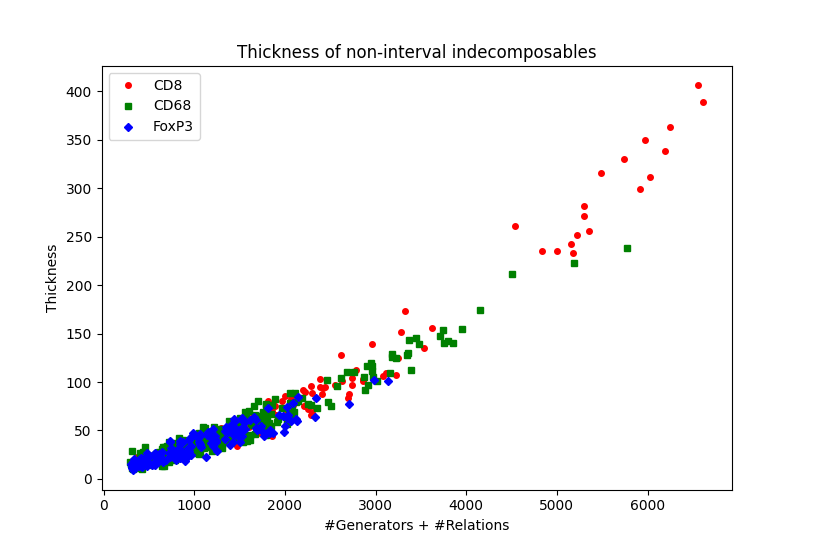}
\caption{Thickness (y) and $|b_0+b_1|$ (x) of all non-interval indecomposables in $H_1$.}
\label{fig:thick_distribution}
\end{figure}

\begin{figure}[H]
\centering
\includegraphics[width=0.8\linewidth]{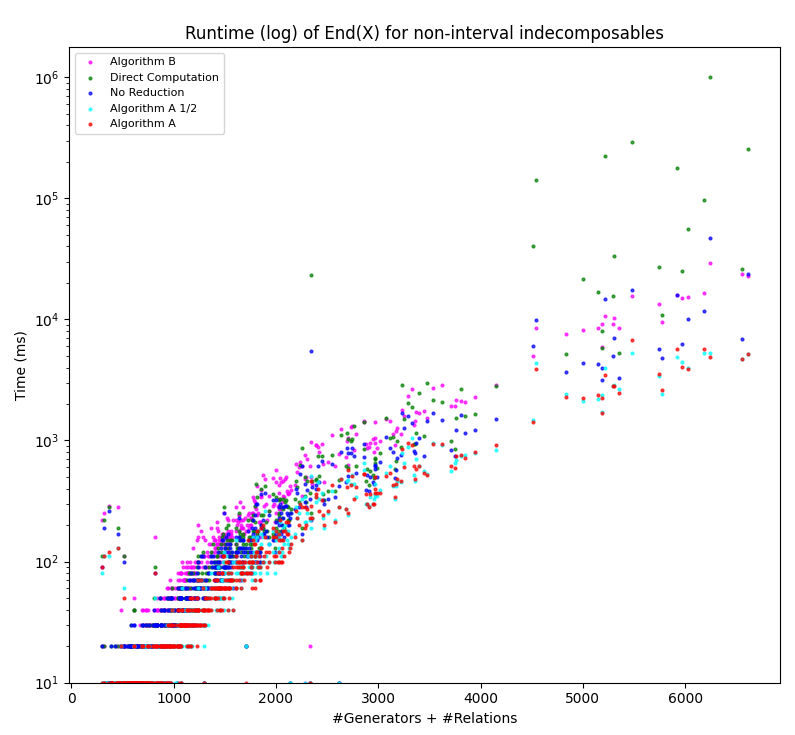}
\caption{Runtimes (log-scale) to compute $\End(X)$ for non-interval indecomposables in the dataset.}
\label{fig:runtime_log}
\end{figure}

\end{document}